\documentclass[11pt,a4paper]{article}
\usepackage{amsmath,amssymb}
\usepackage{amsthm}
\usepackage{amsfonts}
\usepackage[dvipdfmx]{graphicx,color,pict2e}        

\makeatletter
\def\multilimits@{\bgroup
  \Let@
  \restore@math@cr
  \default@tag
 \baselineskip\fontdimen10 \scriptfont\tw@
 \advance\baselineskip\fontdimen12 \scriptfont\tw@
 \lineskip\thr@@\fontdimen8 \scriptfont\thr@@
 \lineskiplimit\lineskip
 \vbox\bgroup\ialign\bgroup\hfil$\m@th\scriptstyle{##}$\hfil\crcr}
\def\Sb{_\multilimits@}
\def\Sp{^\multilimits@}
\def\endSb{\crcr\egroup\egroup\egroup}

\def\smallmatrix{\null\,\vcenter\bgroup
 \Let@\restore@math@cr\default@tag
 \baselineskip6\ex@ \lineskip1.5\ex@ \lineskiplimit\lineskip
 \ialign\bgroup\hfil$\m@th\scriptstyle{##}$\hfil&&\thickspace\hfil
 $\m@th\scriptstyle{##}$\hfil\crcr}
\def\endsmallmatrix{\crcr\egroup\egroup\,}
\newcount\c@MaxMatrixCols
\makeatother

\theoremstyle{plain}
\newtheorem{theorem}{Theorem}[section]
\newtheorem{lemma}[theorem]{Lemma}
\newtheorem{proposition}[theorem]{Proposition}
\newtheorem{corollary}[theorem]{Corollary}

\theoremstyle{definition}
\newtheorem{definition}[theorem]{Definition}
\newtheorem{example}[theorem]{Example}

\theoremstyle{remark}
\newtheorem{remark}[theorem]{Remark}

\newtheorem*{rem-not}{Remark on the Notation}

\numberwithin{equation}{section}
\def\sgn{\mathop{\rm sgn}}

\def\Res{\mathop{\rm Res}}

\def\Z{\mathbb{Z}}
\def\Q{\mathbb{Q}}
\def\R{\mathbb{R}}
\def\C{\mathbb{C}}

\def\calA{\mathcal{A}}

\def\abs#1{\left|#1\right|}

\newcommand{\nidan}[2]{\stackrel{\scriptstyle{#1}}{#2}}

\begin{document}
\title{
Converse theorems for automorphic distributions and 
Maass forms of level $N$
}
\author{Tadashi Miyazaki
\!\!\thanks{The first and the second authors are supported by JSPS KAKENHI Grant Number 15K04800. }
\footnote{Department of Mathematics, College of Liberal Arts and Sciences, Kitasato University, 1-15-1 Kitasato, Minamiku, Sagamihara, Kanagawa, 252-0373, Japan. 
E-mail:\texttt{miyaza@kitasato-u.ac.jp}}\,,
Fumihiro Sato${}^*$\!\!
\footnote{Institute for Mathematics and Computer Science, Tsuda College, 2-1-1 Tsuda-machi, Kodaira-shi,
Tokyo, 187-8577, Japan. E-mail:\texttt{fsato@tsuda.ac.jp }}\,, 
Kazunari Sugiyama\!\!
\footnote{Department of Mathematics, Chiba Institute of Technology, 2-1-1 
Shibazono, Narashino, Chiba, 275-0023, Japan. 
E-mail:\texttt{skazu@sky.it-chiba.ac.jp}}\,\,
and Takahiko Ueno\!\!
\footnote{School of Medicine, St.\ Marianna University, 2-16-1 Sugao, Miyamae-ku, Kawasaki, Kanagawa, 216-8511, Japan. 
E-mail:\texttt{t2ueno@marianna-u.ac.jp}}
}
\date{\today}
\maketitle

\begin{abstract}
We investigate the relations for $L$-functions satisfying certain functional equation,  summationa formulas of Voronoi-Ferrar type and Maass forms of integral and half-integral weight. 
Summation formulas of Voronoi-Ferrar type can be viewed as an automorphic property of distribution vectors of non-unitary principal series representations of the double covering group of $SL(2)$. 
Our goal is converse theorems for automorphic distributions and Maass forms of level $N$ characterizing them by analytic properties of the associated $L$-functions. 
As an application of our converse theorems, we construct Maass forms from the two-variable zeta functions related to quadratic forms studied by Peter and the fourth author. 
\end{abstract}

\section*{Introduction}

It is well-known that the functional equation of the Riemann zeta function, the Poisson summation formula, and the transformation formula for the Jacobi theta function (Thetanullwert) are equivalent (see Hamburger \cite{Hamburger}). 
In the present paper we investigate similar relations for Maass forms of integral and half-integral weight, and give converse theorems for automorphic distributions and Maass forms of level $N$. 
As an application of our converse theorems, we construct Maass forms from the two-variable zeta functions related to quadratic forms studied by Peter \cite{Peter} and the fourth author \cite{Ueno}.

{\bf 0.1}\quad 
Let $\ell$ be an integer, $N$ a positive integer, and $\lambda$ a complex number different from $\frac{1-k}2$ $(k = 0,1,2,\ldots)$. 
We assume that $N \equiv 0 \pmod 4$ if $\ell$ is odd. 
Our main concern is the relation between the following three objects:
\begin{enumerate}
\def\labelenumi{(\Roman{enumi})}
\item
$L$-functions  
\[
\xi_\pm(\alpha;s)=\sum_{n=1}^\infty \frac{\alpha(\pm n)}{n^s}, \quad 
\xi_\pm(\beta;s)=\sum_{n=1}^\infty \frac{\beta(\pm n)}{n^s}
\]
satisfying the functional equation
\begin{eqnarray}
\lefteqn{
N^{s+2\lambda-2}
 \begin{pmatrix}
  e^{\pi s i/2}  & e^{-\pi s i/2} \\
  e^{-\pi s i/2}  & e^{\pi s i/2} 
 \end{pmatrix}
 \left(
 \begin{array}{c}
 \Xi_{+}(\alpha;s) \\[2pt] \Xi_{-}(\alpha;s)
 \end{array}
 \right)
} \nonumber \\
 & & = 
 \begin{pmatrix}
  i^{\ell} \cdot e^{-\pi (2-2\lambda-s) i/2}  & i^{\ell} \cdot e^{\pi (2-2\lambda-s) i/2} \\
  e^{\pi (2-2\lambda-s) i/2}  & e^{-\pi (2-2\lambda-s) i/2} 
 \end{pmatrix}
 \left(
 \begin{array}{c}
 \Xi_{+}(\beta;2-2\lambda-s) \\[2pt] \Xi_{-}(\beta;2-2\lambda-s)
 \end{array}
 \right)
\label{eqn:FE in Introduction}
\end{eqnarray}
and some additional analytic conditions. 
Here we put $\Xi_{\pm}(\ast;s)=(2\pi)^{-s}\Gamma(s)\xi_{\pm}(\ast;s)$.
\item Summation formula of the form 
\begin{eqnarray}
\lefteqn{
\alpha(\infty) f(\infty)+
 \sum_{n=-\infty}^{\infty} \alpha(n) (\mathcal{F}f)(n)
} \nonumber \\
 & & = \beta(\infty) f_{\infty}(\infty) 
      + \sum_{n=-\infty}^{\infty} \beta(n) (\mathcal{F}f_{\infty})  \left(\frac{n}{N}\right)  \quad (f \in \mathcal{V}_{\lambda,\ell}^\infty).
\label{eqn:Ferrar-Suzuki in Introduction}
\end{eqnarray}
Here $\mathcal{V}_{\lambda,\ell}^\infty$ is the space of $C^\infty$-functions $f$  on $\R$ such that 
\[
f_\infty(x):=(\mathrm{sgn}(x))^{\ell/2}|x|^{-2\lambda}f(-1/x)
\]
is also a $C^\infty$-function on $\R$, $f(\infty)=f_\infty(0)$, $f_\infty(\infty)=i^\ell f(0)$, and $\mathcal{F}$ denotes the Fourier transform. 
\end{enumerate}

Note that the space $\mathcal{V}^\infty_{\lambda,\ell}$ is the space of smooth vectors in the line model of the principal series representation of $SL_2(\R)$ for even $\ell$ and of the double covering group of  $SL_2(\R)$ for odd $\ell$.   
In this paper, we introduce a group $\widetilde{G}$ isomorphic to the direct product of $\Z/2\Z$ and the double covering group of $SL_2(\R)$, and consider the action of $\widetilde{G}$ on $\mathcal{V}_{\lambda,\ell}^\infty$, which we call the principal series representation of $\widetilde{G}$. 

\begin{enumerate}
\def\labelenumi{(\Roman{enumi})}
\setcounter{enumi}{2}
\item Maass forms of weight $\ell/2$ 
\begin{align*}
 F_\alpha(z) &= \alpha(\infty) \cdot y^{\lambda-\ell/4}+
 \alpha(0)\cdot i^{-\ell/2} \cdot 
 \frac{(2\pi) 2^{1-2\lambda} \Gamma(2\lambda-1)}
 {\Gamma\left(\lambda+\frac{\ell}{4}\right)
 \Gamma\left(\lambda-\frac{\ell}{4}\right)} \cdot y^{1-\lambda-\ell/4} \\[3pt]
  &\quad + \sum
 \begin{Sb}
 n=-\infty \\
 n\neq 0
 \end{Sb}^{\infty} \alpha(n) \cdot 
 \frac{i^{-\ell/2}\cdot \pi^{\lambda} \cdot  |n|^{\lambda-1}}
{\Gamma\left(\lambda+\frac{\sgn(n)\ell}{4}\right)} \cdot 
 y^{-\ell/4}\, W_{\frac{\sgn(n)\ell}{4}, \lambda-\frac{1}{2}}\left(4\pi|n|y\right)
\cdot e^{2\pi i nx},  \\
G_\beta(z) &=  N^\lambda \beta(\infty) \cdot y^{\lambda-\ell/4} 
  + N^{1-\lambda}  \beta(0) \cdot i^{-\ell/2} \cdot
 \frac{(2\pi) 2^{1-2\lambda} \Gamma(2\lambda-1)}
 {\Gamma\left(\lambda+\frac{\ell}{4}\right)
 \Gamma\left(\lambda-\frac{\ell}{4}\right)} \cdot y^{1-\lambda-\ell/4} \\[3pt]
 \label{form:Results2}
  &\quad + N^{1-\lambda}  \sum
 \begin{Sb}
 n=-\infty \\
 n\neq 0
 \end{Sb}^{\infty} \beta(n) \cdot 
 \frac{i^{-\ell/2} \cdot \pi^{\lambda} \cdot  |n|^{\lambda-1}}
{\Gamma\left(\lambda+\frac{\sgn(n)\ell}{4}\right)} \cdot 
 y^{-\ell/4}\, W_{\frac{\sgn(n)\ell}{4}, \lambda-\frac{1}{2}}\left(4\pi|n|y\right)
\cdot e^{2\pi i nx} 
\end{align*}
satisfying 
\begin{equation}
F_\alpha\left(-\frac{1}{Nz}\right)(\sqrt{N} z)^{-\ell/2}=G_\beta(z).
\label{eqn:transformation formula Introduction}
\end{equation}
Here $z=x+iy$ is in the Poincar\'e upper half plane, and $W_{\mu,\nu}(y)$ denotes the Whittaker function. 

\end{enumerate}

{\bf 0.2}\quad 
First we explain the relation between (I) and (II). 
A general recipe of deriving a summation formula of Voronoi type from Dirichlet series with functional equation was given by Ferrar \cite{Ferrar1}, \cite{Ferrar2}. 
The summation formula in (II) can be viewed as the Ferrar summation formula obtained from the $L$-functions satisfying the functional equation in (I), though the formulation in (II) is different from Ferrar's. 
In Ferrar's formulation, the test function is $\varphi=\mathcal{F}f$, not $f$, and  the left hand side and the right hand side of the summation formula carry information about $\varphi$ and the integral transform $\mathcal{I}\varphi=\mathcal{F}((\mathcal{F}^{-1}\varphi)_\infty)$ of $\varphi$, respectively. 
The advantage of the formulation in (II) is that  it reveals a relation of the summation formula to the (non-unitary) principal series representation of $\widetilde{G}$, as explained in \cite[Chapter III.5]{LewisZagier} in the case $\ell=0$ and $N=1$. 
Indeed, if the summation formula  \eqref{eqn:Ferrar-Suzuki in Introduction}  holds for all $f \in \mathcal{V}^\infty_{\lambda,\ell}$, then both sides of  \eqref{eqn:Ferrar-Suzuki in Introduction}  give two different expressions of the same continuous linear functional  on $\mathcal{V}^\infty_{\lambda,\ell}$, and the expressions obtained from the left and right hand sides are invariant under the actions of $\tilde n(1)$ and 
$\tilde{\bar n}(N)$, respectively. 
Here $\tilde n(1)$ and $\tilde{\bar n}(N)$ are the lifts to $\widetilde G$ of $n(1)=\left( \begin{smallmatrix} 1 & 1 \\ 0 & 1 \end{smallmatrix}\right)$ and $\bar n(N)= \left( \begin{smallmatrix} 1 & 0 \\ N & 1 \end{smallmatrix}\right)$, respectively.  
Hence the continuous linear functional obtained from the summation formula defines a distribution vector of the dual principal series representation automorphic under the group $\widetilde \Delta(N)$ generated by $\tilde n(1)$ and $\tilde{\bar n}(N)$ (see Lemma \ref{lemma:distribution interpretation of sum formula}). 
In the following we call an element in the continuous dual of $\mathcal{V}^\infty_{\lambda,\ell}$ a {\it distribution}. 

To the extent of the authors' knowledge, 
the summation formula as given in (II) 
was first introduced by Suzuki \cite{SuzukiAD} 
(in a more general setting of prehomogeneous vector spaces $(GL(n),Sym(n))$) 
in the name of ``distribution with automorphy''. 
So we call \eqref{eqn:Ferrar-Suzuki in Introduction} {\it the Ferrar-Suzuki
summation formula (of level $N$)}. 
Suzuki \cite{SuzukiAD} associated $L$-functions $\xi_\pm(\alpha;s)$ and $\xi_\pm(\beta;s)$ with both sides of the summation formula and proved their functional equation by a modification of the method in the theory of prehomogeneous vector spaces. 
Tamura extended the result of Suzuki to more general prehomogeneous vector spaces in \cite{Tamura}. 
In \cite{SuzukiAD} (and also in \cite{Tamura}), the space of test functions is restricted to the space of $C^\infty$-functions on $\R^\times=\R\setminus\{0\}$ with compact support, and hence the terms involving $\alpha(\infty)$ and $\beta(\infty)$ vanish.  
This means that  information on the poles of $\xi_\pm(\alpha;s)$ and $\xi_\pm(\beta;s)$ is lost 
(see \cite[Remark to Theorem 1]{SuzukiAD} and \cite[Theorem 3]{Tamura}).  
To consider the summation formula for any $f \in \mathcal{V}^\infty_{\lambda,\ell}$, we have to extend the definition of the Fourier transform $\mathcal{F}f$, since $f \in \mathcal{V}^\infty_{\lambda,\ell}$ is not necessarily integrable unless $\Re(\lambda)>\frac12$.  
We discuss this problem in \S 2 and establish basic properties of the generalized  Fourier transformation. 
In particular we prove the local functional equation relating the Mellin transform of $f \in \mathcal{V}^\infty_{\lambda,\ell}$ to the Mellin transform of the Fourier transform $\mathcal{F}f$ (Theorem \ref{thm:LFE}). 
If $f$ belongs to the subspace of rapidly decreasing functions in  $\mathcal{V}^\infty_{\lambda,\ell}$, then the local functional equation is the Tate local functional equation over $\R$ (the simplest example of the local functional equations in the theory of prehomogeneous vector spaces; see \cite[\S 4.2, Example 1]{PVBook}). 
This extension of the local functional equation is a special case of a more general result due to Lee \cite[Theorem 5.2]{Lee16}. 
We need, however, more precise information, which has not yet been obtained in the general context, such as the location of poles and the exact values of residues. 

Based on the results in \S 2 on the Fourier transformation, 
we discuss the passage from the summation formula (II) to $L$-functions (I) and from (I) to (II) in \S3. 
First in \S 3.1, we derive from the Ferrar-Suzuki summation formula \eqref{eqn:Ferrar-Suzuki in Introduction} that $(s-1)(s-2+2\lambda)\xi_\pm(\alpha;s)$ and  $(s-1)(s-2+2\lambda)\xi_\pm(\beta;s)$ can be continued to entire functions  and satisfy the functional equation \eqref{eqn:FE in Introduction} (Theorem \ref{thm:FEofZeta1}). 
We also calculate the residues at the poles, which are expressed in terms of $\alpha(0), \beta(0), \alpha(\infty), \beta(\infty)$. 
In \S 3.2, conversely, we derive the Ferrar-Suzuki summation formula assuming some  analytic properties of $L$-functions (Theorem \ref{thm:ConverseTheoremForSummationFormula}). 
Here we follow basically the method of Suzuki \cite{SuzukiAD}. 
His method is based on the idea in the theory of prehomogeneous vector spaces 
 (namely, the combination of the summation formula for arbitrary test functions and the local functional equation), and
is suited well to the proof of the converse theorem as explained in \cite[\S 2]{MSintro}
in the case of the Poisson summation formula (see also \cite{SatoHamburger}). 
Miller and Schmid \cite{MSjfa} discussed the passage from automorphic distributions to Dirichlet series with functional equation based on their detailed study  on regularity of distributions  (see also \cite{MSrankin} and \cite{Schmid}). Our method here is more elementary than theirs.

{\bf 0.3}\quad 
The passage from the summation formula (II) to Maass forms (III) is given by the Poisson transformation, which is a $\widetilde{G}$-equivariant mapping from the continuous dual of $\mathcal{V}^\infty_{\lambda,\ell}$ to the space of eigenfunctions of the hyperbolic Laplacian of weight $\ell/2$ of moderate growth. 
As we observed above, the Ferrar-Suzuki summation formula \eqref{eqn:Ferrar-Suzuki in Introduction} defines an automorphic distribution for the group $\widetilde\Delta(N)=\langle \tilde{n}(1), \tilde{\bar n}(N) \rangle$, and hence its Poisson transform, which is given by $F_\alpha(z)$ in (III), is a Maass form automorphic for $\widetilde\Delta(N)$. 
The transformation formula \eqref{eqn:transformation formula Introduction} is derived from the comparison of the Poisson transforms of both sides of the summation formula. 

The group $\widetilde\Delta(N)$ is a subgroup of the lift $\widetilde\Gamma_0(N)$ of $\Gamma_0(N)$ to $\widetilde{G}$. 
In \S 4 we prove a converse theorem that gives a condition for the automorphic distribution for $\widetilde\Delta(N)$ associated with the Ferrar-Suzuki summation  formula to be automorphic for the larger group $\widetilde\Gamma_0(N)$ in terms of twists of the corresponding $L$-functions by Dirichlet characters (Theorem \ref{thm:ConverseTheoremForCongSubgp}). 
By the Poisson transformation, this yields immediately a converse theorem for Maass forms for $\widetilde\Gamma_0(N)$ (Theorem \ref{corollary:Maassforms}), which is an analogue of the converse theorem of Weil \cite{Weil} for holomorphic modular forms of integral weight (in the case of even $\ell$) and its generalization by Shimura \cite[Section 5]{Shimura73} to holomorphic modular forms of half-integral weight (in the case of odd $\ell$).  
A merit of our approach is to keep calculations involving the Whittaker function to a minimum.  The information on the Whittaker function we need is only the fact that  the Whittaker function appears as the Fourier transform of the Poisson kernel (see \eqref{eqn:Fourier transform of Poisson kernel}). 
We say now a few words about the assumptions in the converse theorem. 
In our converse theorem some analytic properties (functional equations, location of poles and so on) are assumed for the $L$-functions twisted by Dirichlet characters of prime modulus  including the principal characters. 
In the original converse theorem of Weil for holomorphic modular forms (\cite{Weil}), the assumptions are imposed for  the $L$-functions twisted  only by primitive Dirichlet characters of prime modulus. 
As pointed out in Gelbart-Miller \cite[\S 3.4]{GM},  it has been considered difficult to transfer  the argument of Weil to the Maass form case. 
In \cite{DG} and \cite{DG2}, Diamantis and Goldfeld proved a converse theorem for double Dirichlet series by considering the twists of Dirichlet series  by Dirichlet characters  including imprimitive characters (the principal characters), namely by the method originally found by Razar \cite{Razar} for holomorphic modular forms.  We follow the approach of Razar and Diamantis-Goldfeld. 
In a paper \cite{NO} that appeared very recently, Nuerurer and Oliver succeeded in modifying the argument of Weil to get a converse theorem for Maass forms of weight $0$ avoiding the twists by imprimitive characters.   
However, in the case of half-integral weight, the functional equation relates the $L$-function twisted  by the Legendre character to the $L$-function twisted  by the principal  character. We do not therefore try to weaken the assumptions in the converse theorem. 

{\bf 0.4}\quad 
In \S 5, applying our converse theorem for Maass forms (Theorem \ref{corollary:Maassforms}), we show that the two-variable zeta functions related to quadratic forms studied by Peter \cite{Peter} and the fourth author \cite{Ueno} independently can be viewed as $L$-functions associated with Maass forms for $\widetilde{\Gamma}_0(N)$ (Theorem \ref{thm:Maass form for Ueno zeta}). 
Let $A=(a_{ij})$ be a nondegenerate half-integral symmetric matrix of size $m$. 
Denote by $D$ and $N$, respectively, the determinant and the level of $2A$. 
Put 
\begin{eqnarray*}
 Z(n,w) &=& \sum_{l=1}^{\infty}\frac{\sharp\left\{
 v\in \Z^{m}/l\Z^m \, |\, {}^tvAv\equiv n \pmod{l}
 \right\}}{l^{w}}, \\
 Z^{*}(n,w) &=& \sum_{l=1}^{\infty} \frac{\sharp\left\{
  v^* \in \Z^{m}/2lA \Z^m \, |\, 4^{-1}N\cdot {}^tv^*A^{-1}v^*\equiv n
  \pmod{Nl}\right\}}{l^w}.
\end{eqnarray*}
Then the zeta functions considered in  \cite{Peter} and \cite{Ueno} are 
\[
  \zeta_{\pm}(w, s) =  \sum_{n=1}^{\infty}Z(\pm n, w)n^{-s},\qquad 
  \zeta_{\pm}^{*}(w,s) = N^{s}\cdot  \sum_{n=1}^{\infty} Z^{*}(\pm n, w) n^{-s}. 
\]
The analytic properties of the zeta functions $\zeta_{\pm}(w, s), \zeta^*_{\pm}(w, s)$ and their twists by Dirichlet characters necessary for the application of the converse theorem are essentially obtained in \cite{Ueno}, and it is shown that, if we put   
\begin{gather*}
\alpha(n) = e^{\pi i (p-q)/4}|D|^{-1/2} Z(n,w), \quad 
\beta(n) = Z^*(n,w)  
\quad (n \in \Z,\ n\ne0),  \\
\alpha(0) = e^{\pi i (p-q)/4}|D|^{-1/2} Z(0,w), \quad 
\beta(0) = |D|^{-1} Z^*(0,w), \\
\alpha(\infty) =  \zeta(w-m+1), \quad 
\beta(\infty) = e^{-\pi i (p-q)/4}|D|^{-1/2} \zeta(w-m+1), 
\end{gather*}
then the functions $F_\alpha(z)$ and $G_\beta(z)$ in (III) are Maass forms of weight $\ell/2$ with $\ell=p-q+4k$ for $\widetilde\Gamma_0(N)$ (Theorem \ref{thm:Maass form for Ueno zeta}). Here $p$ and $q$ are respectively the numbers of  positive and negative eigenvalues of $A$, and $k$ is an arbitrary integer. 
Mizuno \cite[Theorem 1]{Mizuno} obtained the same Maass forms for positive definite $A$, and this generalizes the result of Mizuno to indefinite $A$. 
The simplest case of the two-variable zeta functions is the double Dirichlet series studied by Shintani \cite{Shintani}.
In this special case, by using their converse theorem for double Dirichlet series,  Diamantis and Goldfeld (\cite{DG} and \cite{DG2}) proved more precisely that the corresponding Maass forms are linear combinations of metaplectic Eisenstein series.

\subsubsection*{Notation}
We denote by $\Z, \Q,\R$, and $\C$ the ring of integers, the field of rational numbers,  the field of real numbers, and  the field of complex numbers, respectively. 
We write $\R^\times=\R\setminus\{0\}$. 
The set of positive integers and the set of non-negative integers are denoted by $\Z_{>0}$ and $\Z_{\geq0}$, respectively. 
The identity matrix of size $m$ is denoted by $1_m$. 
The real part and the imaginary part of a complex number $s$ are denoted by $\Re(s)$ and $\Im(s)$, respectively. 
For complex numbers $\alpha, z$ with $\alpha\ne0$, $\alpha^z$ always stands for the principal value, namely, $\alpha^z = \exp((\log|\alpha|+i\,\mathrm{arg}\, \alpha)z)$ with $-\pi<\mathrm{arg}\, \alpha\leq\pi$. 
We use $\mathbf{e}[x]$ to denote $\exp(2\pi i x)$. 
The quadratic residue symbol $\left(\frac{\ast}{\ast}\right)$ has the same meaning as in Shimura~\cite[p.~442]{Shimura73}. For a meromorphic function $f(s)$ with a pole at $s=\alpha$, we denote its residue at $s=\alpha$ by $\displaystyle \Res_{s=\alpha} f(s)$. 

\section{Preliminaries}

\subsection{A covering group of $SL_2(\R)$}

Let $G=SL_2(\R)=\left\{\left. g\in M_2(\R)\right| \det g=1\right\}$. 
The group $G$ acts on the Poincar\'{e} upper half plane
$\mathcal{H}:=\{z\in \C \, |\, \Im(z)>0\}$ by
\begin{equation}
 \label{form:LinearFractionalTransform}
 gz= \frac{az+b}{cz+d} \qquad
 \text{for} \quad
 g=\begin{pmatrix}
 a & b\\
 c & d
 \end{pmatrix}.
\end{equation}
We put $j(g, z)= cz+d$. 
We define a (non-connected) covering group
$\widetilde{G}$ of $G$ by
\begin{equation}
\label{form:DefOfMetaplecticGp}
 \widetilde{G}=\left\{
 (g, \varphi) \left|  
 \begin{array}{l}
 g\in G, \varphi\ \text{is a holomorphic function on} \; 
 \mathcal{H}\\
 \text{satisfying}\;  \varphi(z)^4 = j(g,z)^2
 \end{array}\right.
 \right\}.
\end{equation}
For $\tilde{g}_1=(g_1, \varphi_1), \tilde{g}_2=(g_2, \varphi_2)\in \widetilde{G}$, 
the product  $\tilde{g}_1\cdot \tilde{g}_2$ is defined by
\begin{equation}
 \label{form:DefOfProductInGtilde}
 \tilde{g}_1\cdot \tilde{g}_2=(g_1 g_2, \, \varphi_3), \qquad 
 \ \varphi_3(z):=\varphi_1(g_{2}z)\varphi_2(z).
\end{equation}
For $\tilde{g}=(g, \varphi)\in \widetilde{G}$, there exists a unique element
$\xi$ of the multiplicative group $\mu_4:=\{\pm 1, \pm i\}$ such that 
\[
 \varphi(z)=\xi \cdot j(g,z)^{1/2}= \xi \cdot (cz+d)^{1/2}.
\]
We put 
\begin{equation}
[g,\xi]=(g, \xi j(g,z)^{1/2}) \quad (g \in G,\ \xi \in \mu_4). 
 \label{form:TwoParis}
\end{equation}

Since $j(g_1g_2,z)=j(g_1,g_2z)j(g_2,z)$ for $g_1, g_2\in G$, there exists a $\sigma(g_1, g_2) \in \{\pm1\}$ such that
\begin{equation}
 \label{form:FactorSystemSigma}
 j(g_1, g_2 z)^{1/2} = \sigma(g_1, g_2)\cdot 
 \frac{j(g_1 g_2, z)^{1/2}}{j(g_2, z)^{1/2}}. 
\end{equation}
Using $\sigma(g_1,g_2)$, we can write the product in $\widetilde G$ as
\begin{equation}
 \label{form:ProductUsingSigma}
 [g_1, \xi_1]\cdot [g_2, \xi_2]= \left[g_1 g_2, \, \xi_1 \xi_2 \cdot \sigma(g_1, g_2)\right].
\end{equation}
The explicit values of $\sigma(g_1, g_2)$ are given by the following lemma due to Petersson (see Maass~\cite[Theorem~16]{MaassLN} for a proof).  

\begin{lemma}
\label{thm:FactorSystem}
For $g_1, g_2\in G$, we put $g_3=g_1 g_2$ and 
write
$g_i =
\begin{pmatrix}
a_i & b_i \\
c_i & d_i
\end{pmatrix} \; (i=1,2,3)$. 
We also write $\sigma(g_1, g_2)= \exp(\pi i \varpi(g_1, g_2))$. 
Then we have 
\[
 4 \cdot \varpi(g_1, g_2)=
 \begin{cases}
 \sgn c_1+\sgn c_2-\sgn c_3-\sgn(c_1 c_2 c_3),
 & \text{if} \; \; c_1 c_2 c_3\neq 0, \\[5pt]
 -(1-\sgn c_1)(1-\sgn c_2), & \text{if} \;\;
 c_1 c_2\neq 0, \, c_3=0, \\[5pt]
 (1-\sgn d_1)(1+\sgn c_2), & \text{if} \;\;
 c_2 c_3\neq 0, c_1=0,\\[5pt]
 (1+\sgn c_1)(1-\sgn a_2), & \text{if}\;\;
 c_1 c_3 \neq 0, c_2=0, \\[5pt]
 (1-\sgn d_1)(1-\sgn a_2), & \text{if}\;\;
 c_1=c_2=c_3=0.
 \end{cases}
\]
\end{lemma}

Let $p:\widetilde G \rightarrow G$ $(p([g,\xi])=g)$ be the standard projection. 
We define the subgroups $\widetilde G_1$, $\widetilde{M}$, $\widetilde{M}_0$ and $\widetilde{M}_1$ of $\widetilde G$ by 
\begin{gather*}
\widetilde G_1 := \left\{\left.[g,\xi] \in \widetilde G\;\right|\;\xi=\pm1\right\}, \quad 
\widetilde{M} := \left\{[\epsilon 1_2, \xi]\left|\; \epsilon=\pm1,\ \xi^4=1\right.\right\}, \\
\widetilde{M}_0 := \left\{[1_2, \xi]\left|\; \xi^4=1\right.\right\}, \quad
\widetilde{M}_1 := \widetilde{M} \cap \widetilde G_1 
              = \left\{[1_2,1], [1_2,-1], [-1_2,1], [-1_2,-1]\right\}. 
\end{gather*}
Then $\widetilde G_1$ is the identity component of $\widetilde G$, and $\widetilde G_1$ is the double covering group of $G$. 
The groups $\widetilde{M}$ and $\widetilde{M}_1$ are the centers of $\widetilde G$ and $\widetilde G_1$, respectively. 
The group $\widetilde{M}_0$ is the kernel of the projection $p:\widetilde G \rightarrow G$. 
The group $\widetilde G$ is isomorphic to the group $\widetilde{G}_1 \times \Z/2\Z$ via the isomorphism 
\[
\widetilde{G}_1 \times \Z/2\Z \longrightarrow \widetilde{G},   \qquad
([g,\xi],a\bmod2) \longmapsto [g,\xi]\cdot[-1_2,-i]^a.
\]
Note that $[-1_2,-i]$ is an element of order $2$ belonging to the center $\widetilde{M}$, and the canonical lift of $-1_2 \in \Gamma_0(4)$ to $\widetilde{G}$ (see \eqref{form:DefOfGammaStar} below). 

For real numbers $a, x$ with $a\ne0$ and a positive integer $N$, we use the standard notation
\[
 d(a)= 
 \begin{pmatrix}
 a & 0 \\
 0 &  a^{-1}
 \end{pmatrix}, \quad 
 n(x)=
 \begin{pmatrix}
 1 & x \\
 0 & 1
 \end{pmatrix}, \quad 
 \bar{n}(x)=
 \begin{pmatrix}
 1 & 0 \\
 x & 1
 \end{pmatrix}, \quad
 w_N=\begin{pmatrix}
 0 & -1/\sqrt{N} \\
 \sqrt{N} & 0
 \end{pmatrix},
\]
and lift these elements in $G$ to $\widetilde G$ by 
\begin{equation}
 \label{form:DefOfTildeNandBarN}
 \tilde{d}(a):=[d(a), \, 1], \quad
 \tilde{n}(x):=[n(x), \, 1], \quad
 \tilde{\bar{n}}(x):=[\bar{n}(x), \, 1], \quad
 \tilde w_N:=[w_N,1]. 
\end{equation}
If $N=1$, we simply write $w$ and $\tilde w$ instead of $w_1$ and $\tilde w_1$, respectively. 

Concrete calculations of group multiplication in $\widetilde G$ can be done by Lemma~\ref{thm:FactorSystem}. For example, we can easily obtain the following identities needed later: 
%
\begin{gather}
\tilde{\bar{n}}(-x) = \tilde w\tilde n(x) \tilde w^{-1}, 
\label{eqn:commutation of w and nx}\\
\tilde w_N  =  \tilde w \tilde d\left(\sqrt N\right), 
\label{eqn: wN equals d w} \\
\tilde w_N \tilde w^{-1}=\tilde d\left(\tfrac{1}{\sqrt N}\right), 
\label{eqn: wN w equals d} \\
\tilde w_N^{-1} \left[\begin{pmatrix} a & b \\ cN & d \end{pmatrix}, \xi\right] \tilde w_N
 = \left[\begin{pmatrix} d & -c \\ -bN & a \end{pmatrix}, \xi\sigma\right] 
\quad (ad-bcN=1,\ \xi^4=1,\ N>0). 
\label{eqn:commutation of wN and gamma0N}
\end{gather}
In \eqref{eqn:commutation of wN and gamma0N}, $\sigma$ is given by
\[
\sigma = \begin{cases}
             -1 & ([a<0,\ b>0,\ c\geq0] \vee [a<0,\ b\leq0,\ c<0]), \\
              1 & (\text{otherwise}).
\end{cases}
\]

\subsection{Principal series representations of $\widetilde{G}$}

In this subsection, we introduce the principal series representations of
$\widetilde{G}$. For details, we refer to Bruggeman-Lewis-Zagier~\cite{BLZ}, 
Suzuki~\cite{SuzukiWeil}, Yoshida~\cite{Yoshida}.

We fix a complex number $\lambda$ and an integer $\ell$. 
For a $C^\infty$-function $f$ on $\R$, we put
\begin{equation}
 \label{form:FinftyAndF}
 f_{\infty}(x)= (\mathrm{sgn}\,x)^{\ell/2}
  |x|^{-2\lambda}\cdot  f \left(-\dfrac{1}{x}\right). 
\end{equation}
Here $(\mathrm{sgn}\,x)^{\ell/2}$ denotes the principal value, namely, 
 $(\mathrm{sgn}\,x)^{\ell/2}=1$ or $i^{\ell}$ according as $x>0$ or $x<0$. 
We consider the space $\mathcal{V}_{\lambda, \ell}^{\infty}$ of $C^\infty$-functions $f$ on $\R$ such that the function $f_\infty$, which is originally defined for $x\ne0$, can be extended to a $C^\infty$-function on $\R$. 
Then 
we have
\begin{equation}
(f_\infty)_\infty(x)=i^\ell f(x),  
\label{eqn:square of infty-operation}
\end{equation}
and the mapping $f \mapsto f_\infty$ preserves $\mathcal{V}_{\lambda, \ell}^{\infty}$. 
When it is necessary to make the dependence of $f_\infty(x)$ on the parameters $\lambda, \ell$ explicit, we write $f_{\infty,\lambda,\ell}(x)$. 

We define semi-norms $\nu_{N,\lambda}$ $(N \in \Z_{\geq0})$ on $\mathcal{V}_{\lambda, \ell}^{\infty}$ by 
\[
 \nu_{N,\lambda}(f)=\sup_{x\in\R}\left\{(1+x^2)^{\Re(\lambda)}\cdot 
 \sum_{j=0}^{N} |({D_\lambda}^j f)(x)|\right\}, 
\]
where $D_\lambda$ is the differential operator given by
\[
 D_\lambda= (1+x^2)\frac{d}{dx}+ 2\lambda x.
\]
The space  $\mathcal{V}_{\lambda, \ell}^{\infty}$ becomes a Fr\'echet space with the topology defined by the family of semi-norms $\{\nu_{N,\lambda}\;|\; N  \in \Z_{\geq0}\}$. 

A continuous representation $\pi_{\lambda,\ell}$ of $\widetilde G$ on  $\mathcal{V}_{\lambda, \ell}^{\infty}$ (the (non-unitary) principal series representation) can be defined  as follows:
Let $\tilde{g}=[g,\xi]\in \widetilde{G}$, 
$g=\begin{pmatrix}
a & b \\
c & d 
\end{pmatrix}\in SL_2(\R)$ and $f\in \mathcal{V}_{\lambda, \ell}^{\infty}$.
Then, if $-cx+a\neq 0$, 
\begin{equation}
\label{form:DefOfActionOnVlambdaMu}
 \left(\pi_{\lambda, \ell}(\tilde{g})f\right)(x)
=
 |-cx+a|^{-2\lambda}\cdot 
 f\left(\frac{dx-b}{-cx+a}\right)\cdot
 \left\{
 \begin{array}{ll}
 \xi^{-\ell} &  (\text{if}\; \, -cx+a>0), \\[3pt]
 (i\xi)^{-\ell}&  (\text{if}\; \,  -cx+a<0, c\geq 0), \\[3pt]
 (-i\xi)^{-\ell} & (\text{if}\; \,  -cx+a<0, c< 0),
 \end{array}
 \right.
\end{equation}
and, if $-cx+a=0$, 
\begin{equation}
\label{form:DefOfActionOnVlambdaMuAtInfty}
 \left(\pi_{\lambda, \ell}(\tilde{g})f\right)(x)
= |c|^{2\lambda}\cdot f_{\infty}(0)\cdot
\left\{
\begin{array}{lc}
\xi^{-\ell} & (-cx+a=0, c<0), \\[3pt]
(i\xi)^{-\ell} & (-cx+a=0, c>0).
\end{array}
\right.
\end{equation}
Note that the representation $\pi_{\lambda,\ell}$ is determined completely to the restriction to the identity component $\widetilde{G}_1$, since $\widetilde{G}=\widetilde{G}_1 \cup [-1_2,-i]\cdot\widetilde{G}_1$ and 
$(\pi_{\lambda,\ell}([-1_2,-i])f)(x)=f(x)$ $(f \in \mathcal{V}^\infty_{\lambda,\ell})$. If $\ell$ is even, then $\pi_{\lambda,\ell}([1_2,\pm1])f(x)=f(x)$ $(f \in \mathcal{V}^\infty_{\lambda,\ell})$ and hence $\pi_{\lambda,\ell}$ is essentially a representation of $G \cong \widetilde{G}/\widetilde{M}_0$.

In terms of $\pi_{\lambda, \ell}$, 
the relation~\eqref{form:FinftyAndF} of $f$ and $f_{\infty}$ can be
rephrased as 
\begin{equation}
 \label{form:FInftyIsPiTau}
 f_{\infty}(x) = \left(\pi_{\lambda, \ell}(\tilde{w}^{-1})f\right)(x), \qquad 
 \tilde{w}^{-1}= [w^{-1},1] =
 \left[
 \begin{pmatrix}
 0 & 1\\
 -1 & 0
 \end{pmatrix},1\right].
\end{equation}

We call a continuous linear functional $T:\mathcal{V}_{\lambda,\ell}^{\infty}\to \C$ a {\it distribution} (on $\mathcal{V}_{\lambda,\ell}^{\infty}$), and 
denote by $\left(\mathcal{V}_{\lambda,\ell}^{\infty}\right)^*$ the space of distributions on $\mathcal{V}_{\lambda,\ell}^{\infty}$. 
We define the action $\pi^*_{\lambda, \ell}$ 
of $\widetilde{G}$ on $T\in \left(\mathcal{V}_{\lambda,\ell}^{\infty}\right)^*$
by
\begin{equation}
 \label{form:DefOfActionOnDualSpace}
  \left(\pi^*_{\lambda, \ell}(\tilde{g})T\right)(f)= T\left(\pi_{\lambda, \ell}(\tilde{g}^{-1})f\right)
\quad \text{for}\; \; \tilde{g}\in \widetilde{G}, f\in \mathcal{V}_{\lambda,\ell}^{\infty}.
\end{equation}

We prove a lemma needed in \S 2 in extending the Fourier transform and the local zeta functions to functions in $\mathcal{V}_{\lambda,\ell}^\infty$.

\begin{lemma}
\label{lem:diff of line model}
If $f$ is in $\mathcal{V}_{\lambda,\ell}^\infty$, then the derivative $f'$ is in 
$\mathcal{V}_{\lambda+1/2,\ell+2}^\infty$. More generally the $m$-th derivative $f^{(m)}$ is in $\mathcal{V}_{\lambda+m/2,\ell+2m}^\infty$, and there exists a positive constant $C$ independent of $f$ satisfying 
\begin{equation}
\nu_{0,\lambda+m/2}\left(f^{(m)}\right) = 
\sup_{x \in \R} \left\{(1+x^2)^{\Re(\lambda)+m/2} \left|f^{(m)}(x)\right|\right\}
 \leq C \nu_{m,\lambda}(f). 
\label{eqn:estimate of m-th derivative}
\end{equation}
\end{lemma}

\begin{proof}
By definition, the function 
\[
f_\infty(x)=(\mathrm{sgn}\,x)^{\ell/2} |x|^{-2\lambda}f\left(-\frac1x\right)
\]
is in $C^\infty(\R)$. 
Differentiating both sides of this identity, we have for $x \ne 0$
\[
 (f_\infty)'(x)  
 = -2\lambda (\mathrm{sgn}\,x)^{\ell/2} |x|^{-2\lambda} x^{-1} f\left(-\frac1x\right) + (\mathrm{sgn}\,x)^{\ell/2} |x|^{-2\lambda-2}f'\left(-\frac1x\right).
\]
Hence the function
\begin{eqnarray*}
(\mathrm{sgn}\,x)^{(\ell+2)/2} |x|^{-2\lambda-1}f'\left(-\frac1x\right)
 &=& x (f_\infty)'(x) + 2\lambda (\mathrm{sgn}\,x)^{\ell/2} |x|^{-2\lambda} f\left(-\frac1x\right) \\ 
 &=& x (f_\infty)'(x) + 2\lambda f_\infty(x)
\end{eqnarray*}
is  in $C^\infty(\R)$. 
This shows that $f'$ is in $\mathcal{V}_{\lambda+1/2,\ell+2}^\infty$ and 
\begin{equation}
\label{eqn:f' infty}
(f')_{\infty,\lambda+1/2,\ell+2}(x) = x (f_\infty)'(x) + 2\lambda f_\infty(x). 
\end{equation}
From this it is obvious that $f^{(m)} \in \mathcal{V}_{\lambda+m/2,\ell+2m}^\infty$. 
To show the inequality \eqref{eqn:estimate of m-th derivative}, we note that there exist polynomials $u_{m,k}(x)$ of degree at most $m-k$ satisfying
\[
(1+x^2)^m f^{(m)}(x) = \sum_{k=0}^m u_{m,k}(x) {D_\lambda}^kf(x) \quad (m \geq 0). 
\]
Hence we have
\begin{eqnarray*}
(1+x^2)^{\Re(\lambda)+m/2} \left|f^{(m)}(x)\right| 
 &\leq& (1+x^2)^{\Re(\lambda)} \sum_{k=0}^m \frac{\left|u_{m,k}(x)\right|}{ (1+x^2)^{m/2}}\cdot \left|{D_\lambda}^kf(x)\right| \\
 &\leq& C (1+x^2)^{\Re(\lambda)} \sum_{k=0}^m \left|{D_\lambda}^kf(x)\right|, 
\end{eqnarray*}
where 
\[
C = \max_{0 \leq k \leq m}\left[
      \sup_{x \in\R} \left\{\frac{\left|u_{m,k}(x)\right|}{ (1+x^2)^{m/2}}\right\}
      \right].
\]
Notice that $C$ is finite, since the degree of $u_{m,k}(x)$ is not greater than $m$.  
\end{proof}

\subsection{Poisson transform}

Let $\ell\in \Z$ and $F:\mathcal{H}\to \C$ be a complex-valued function on 
the Poincar\'{e} upper half plane $\mathcal{H}$. 
For $\tilde{g}=(g, \varphi)\in\widetilde{G}$, we define the slash operator $\big|_{\ell}$ by
\begin{equation}
 \label{form:DefOfActionOfTildeGonH}
\left(F\big|_{\ell}\, \tilde{g}\right)(z)= F(gz)\cdot \varphi(z)^{-\ell},
\end{equation}
where $gz$ denotes the linear fractional transformation 
\eqref{form:LinearFractionalTransform}.
By \eqref{form:DefOfProductInGtilde}, we have
\begin{equation}
 \label{form:CompositionOfSlash}
 F\big|_{\ell} \left(\tilde{g}_1 \cdot \tilde{g}_2\right) = \left(F\big|_{\ell}\, \tilde{g}_1\right)\big|_{\ell} \, \tilde{g}_2
\end{equation}
for $\tilde{g}_1, \tilde{g}_2\in \widetilde{G}$. 

Furthermore, 
we define the hyperbolic Laplacian 
$\Delta_{\ell/2}$ of weight $\ell/2$  on $\mathcal{H}$ by  
\begin{equation}
\label{form:Laplacian}
 \Delta_{\ell/2} = -y^2\left(\frac{\partial^2}{\partial x^2}+
 \frac{\partial^2}{\partial y^2}\right)+ \frac{i \ell y}{2}\left(\frac{\partial}{\partial x}
 +i\frac{\partial}{\partial y}\right) \quad (z=x+iy \in \mathcal{H}). 
\end{equation}
It is known that
for $C^{\infty}$-functions $F$ on $\mathcal{H}$ and $\tilde{g}\in \widetilde{G}$,
\begin{equation}
 \label{form:CommutativityOfDelta}
\left(\Delta_{\ell/2} F\right)\big|_\ell \tilde{g} 
 = \Delta_{\ell/2}  \left(F\big|_\ell \tilde{g}\right).
\end{equation}

For $\lambda\in \C, \ell\in \Z$, 
$u \in \R$, $z\in \mathcal{H}$, we define the {\it Poisson kernel\/} $p_{\lambda, \ell}(u, z)$ by 
\begin{equation}
 \label{form:DefOfPoissonKernel}
 p_{\lambda, \ell}(u, z)=
 \frac{y^{\lambda-(\ell/4)}}{|z-u|^{2\lambda-(\ell/2)}\cdot (z-u)^{\ell/2}}.
\end{equation}
When we fix $z\in \mathcal{H}$ (resp.\  $u \in \R$) and 
regard $p_{\lambda, \ell}(u, z)$ as a function of $u$ (resp.\ $z$), we denote it by
$p_{\lambda, \ell, z}(u)$ (resp.\ $q_{\lambda, \ell,u}(z)$).
Then it is easy to check the following lemma.

\begin{lemma}
 \label{lem:PoissonKerIntertwines}
 As a function of $u$, the Poisson kernel $p_{\lambda, \ell, z}(u)$
 is an element of $\mathcal{V}_{\lambda, \ell}^{\infty}$. We have
 \begin{equation}
 \label{form:PoissonKerIntertwines}
 \left(\pi_{\lambda, \ell}(\tilde{g})p_{\lambda, \ell, z}\right)(u)=
  \left(q_{\lambda, \ell, u}\big|_{\ell} \tilde{g}\right)(z) \qquad (\tilde{g}\in \widetilde{G}). 
 \end{equation}
 Furthermore, $q_{\lambda, \ell,u}(z)$ is an eigenfunction of the
Laplacian $\Delta_{\ell/2}$ defined by \eqref{form:Laplacian}. That is, 
\begin{equation}
\label{form:EigenfunctionYlambda}
\Delta_{\ell/2} q_{\lambda, \ell, u}(z) = \left(\lambda-\frac{\ell}{4}\right)
\left(1-\lambda-\frac{\ell}{4}\right) \cdot q_{\lambda, \ell, u}(z).
\end{equation}
\end{lemma}

\begin{definition}
The Poisson transform $\mathcal{P}_{\lambda,\ell}T$ of a distribution $T\in \left(\mathcal{V}_{\lambda, \ell}^{\infty}\right)^*$ is defined by
\[
 (\mathcal{P}_{\lambda, \ell}T)(z)= T(p_{\lambda, \ell, z}).
\]
\end{definition}

We need the following basic properties of the Poisson transforms.

\begin{lemma}
\label{lem:GrowthOfPtrans}
$(1)$ 
For $\tilde g \in \widetilde G$, we have 
\[
\left(\mathcal{P}_{\lambda, \ell}T\big|_{\ell} \, \tilde{g}\right)(z)
 = \mathcal{P}_{\lambda, \ell}\left(\pi^*_{\lambda, \ell}(\tilde{g}^{-1})T\right)(z).
\]

$(2)$ 
For $T\in \left(\mathcal{V}_{\lambda,\ell}^{\infty}\right)^*$, the Poisson transform $\mathcal{P}_{\lambda, \ell}T$ satisfies the differential equation
\[
 \Delta_{\ell/2} \cdot (\mathcal{P}_{\lambda, \ell}T)(z) 
  =  \left(\lambda-\frac{\ell}{4}\right)\left(1-\lambda-\frac{\ell}{4}\right) \cdot
 (\mathcal{P}_{\lambda, \ell}T)(z),
\]
and hence  $\mathcal{P}_{\lambda, \ell}T$  is a real analytic function on $\mathcal{H}$.

$(3)$ 
For $T\in \left(\mathcal{V}_{\lambda,\ell}^{\infty}\right)^*$, then there exist some
$C>0$ and $r\in\Z_{\geq 0}$ such that 
\[
 \left|\left(\mathcal{P}_{\lambda, \ell}T\right)(z)\right|\leq C
 \cdot \left(\frac{|z+i|^2}{y}\right)^r \qquad (z=x+iy\in \mathcal{H}).
\]
\end{lemma}

The first and the second assertions follow immediately from Lemma~\ref{lem:PoissonKerIntertwines}. 
For the third assertion, we refer to Lewis~\cite[Theorem~4.1]{Lewis78} and Oshima~\cite[Theorem~3.5]{Oshima}.
Furthermore, the Poisson transform is a bijection of  $\left(\mathcal{V}_{\lambda,\ell}^{\infty}\right)^*$ onto the space of functions on $\mathcal{H}$ with the properties in (2) and (3) above for  generic $\lambda$.
 (see  Oshima-Sekiguchi~\cite[Theorem~3.15]{OS}  and Wallach \cite[\S 11.9]{Wallach} for more general treatment).

\subsection{Maass forms for $\widetilde{\Gamma}_0(N)$}

For a positive integer $N$, let
\[
 \Gamma_0(N)=\left\{ \left.
 \begin{pmatrix}
 a & b \\
 c & d 
 \end{pmatrix}\in SL_2(\Z) \right| 
 c\equiv 0 \pmod{N}
 \right\}.
\]
The subgroups $\widetilde{\Gamma}_0(N)$ 
of $\widetilde{G}$ corresponding to $\Gamma_0(N)$
are defined differently, in accordance with the parity of $\ell$. 

\medskip

\noindent
(1) In the case of even $\ell$, the subgroup 
$\widetilde{\Gamma}_0(N)$ of $\widetilde{G}$ is defined by
\begin{equation}
 \label{form:TildeGammaForEven}
 \widetilde{\Gamma}_0(N)= \left\{
 \bar{\gamma}_+,\ \bar{\gamma}_-\, |\, \gamma\in \Gamma_0(N) \right\}, 
\quad \bar{\gamma}_\pm=[\gamma,\pm1]=(\gamma,\pm j(\gamma, z)^{1/2}).
\end{equation}
If we write $\ell=2\kappa\, (\kappa\in \Z)$, 
then the action  \eqref{form:DefOfActionOfTildeGonH} 
of $\bar{\gamma}_{\pm}$ becomes
\begin{equation}
 \label{form:ActionInEvenWeight}
 \left(F\big|_{\ell}\, \bar{\gamma}_{\pm}\right)(z)= F(\gamma z)\cdot j(\gamma, z)^{-\kappa},  
\end{equation}
and coincides with the standard action of $\gamma\in \Gamma_0(N)$ with
integral weight $\kappa=\ell/2$.

\medskip

\noindent
(2) In the case of odd $\ell$, we use the ``theta multiplier''. 
We put 
\[
\theta(z)=\sum_{n=-\infty}^{\infty} \exp(2\pi i n^2 z), \qquad
J(\gamma, z)=\frac{\theta(\gamma z)}{\theta(z)}.
\]
Then it is well-known that
\[
J(\gamma, z)= \varepsilon_{d}^{-1} \cdot \left(\frac{c}{d}\right)\cdot 
 (cz+d)^{1/2} \quad \text{for}\; \; \gamma=
 \begin{pmatrix}
 a & b \\
 c & d 
 \end{pmatrix}\in \Gamma_0(4), 
\]
where 
\begin{equation}
\label{eqn:def of epsilond}
\varepsilon_d = \begin{cases}
 1 & (d \equiv 1 \pmod 4), \\
 i & (d \equiv 3 \pmod 4).
 \end{cases}
\end{equation}
Note that 
\[
 J(\gamma, z)^{2}=\left(\frac{-1}{d}\right)\cdot j(\gamma, z).
\]
Thus, if we put 
\begin{equation}
 \label{form:DefOfGammaStar} 
 \gamma^*= \left(\gamma, \ J(\gamma, z)\right)   
 = \left[\gamma, \ \varepsilon_d^{-1}\cdot \left(\frac{c}{d}\right)\right]
\qquad 
 \text{for}\;\; \gamma\in \Gamma_0(4), 
\end{equation}
then the map $\gamma\mapsto \gamma^{*}$ defines an injective homomorphism of
$\Gamma_0(4)$ into $\widetilde{G}$. 
In the case of odd $\ell$, for a positive integer $N$ with $4|N$, 
the group $\widetilde{\Gamma}_0(N)$ is defined to be the image of 
$\Gamma_0(N)$ under this injective homomorphism. 
That is, 
\begin{equation}
 \label{form:TildeGammaForOdd}
 \widetilde{\Gamma}_0(N) 
 = \left\{
 \left(\gamma, \ J(\gamma, z)\right)\, |\, 
 \gamma\in \Gamma_0(N)
 \right\}.
\end{equation}

Let $\chi$ be a Dirichlet character mod ${N}$. Then we use 
the same symbol $\chi$ to denote the character of $\widetilde{\Gamma}_0(N)$ 
defined by
\begin{equation}
 \label{form:ChiInducesCharOnGamma0N}
 \chi(\tilde\gamma)=\chi(d) \qquad \text{for 
 $\tilde\gamma=[\gamma,\xi]$ with 
 $\gamma=
 \begin{pmatrix}
 a & b \\
 c & d 
 \end{pmatrix}\in \Gamma_0(N)$}.
\end{equation}

\begin{definition}[Maass forms]
\label{defn:DefOfMaassForms}
Let $\ell \in \Z$, and $N$ be a positive integer, with $4|N$ when $\ell$ is odd.  
A complex-valued $C^{\infty}$-function $F(z)$ on $\mathcal{H}$
is called {\it a Maass form for} $\widetilde{\Gamma}_0(N)$
of weight $\ell/2$ with character $\chi$, if the following three conditions 
are satisfied;
\begin{enumerate}
\def\labelenumi{(\roman{enumi})}
\item 
$F\big|_{\ell}\, \tilde{\gamma}= \chi(\tilde\gamma) \cdot F$ for 
 every $\tilde{\gamma}\in \widetilde{\Gamma}_0(N)$, 
\item 
$\Delta_{\ell/2} F= \Lambda \cdot F$ with some $\Lambda\in \C$,
\item  
$F$ is of moderate growth at every cusp, namely,  for every $A \in SL_2(\Z)$, there exist positive constants $C$, $K$ and $\nu$ depending on $F$ and $A$ such that 
\[
|(F\big|_{\ell} [A,1])(z)|< C y^{\nu} \quad \text{if}\;\;   y= \Im(z)>K. 
\]
\end{enumerate}
We call $\Lambda$ the {\it eigenvalue} of $F$. 
\end{definition}

\begin{remark}
\label{remark:CondOnCharacter}
In the condition (i) above, let 
\[
\tilde{\gamma} 
 = \begin{cases}
      [-1_2, 1]=(-1_2,i) & (\ell \equiv 0 \pmod 2), \\
      (-1_2)^*=[-1_2, -i]=(-1_2,1) & (\ell \equiv 1 \pmod 2).
    \end{cases}
\]
Then, we have
\[ 
  \left(F\big|_{\ell}\, \tilde{\gamma}\right)(z) = F(z) \cdot
\begin{cases}
      i^{\ell} & (\ell \equiv 0 \pmod 2), \\
      1 & (\ell \equiv 1 \pmod 2).
\end{cases}
\]
Hence there exist no non-trivial Maass forms for $\widetilde{\Gamma}_0(N)$
of weight $\ell/2$ with character $\chi$, unless $\chi(-1) = i^\ell$ or $1$ according as $\ell \equiv 0 \pmod 2$ or $\ell \equiv 1 \pmod 2$. 
\end{remark}

\begin{definition}
\label{defn:AutomDistriGamma0}
Let $\widetilde\Gamma$ be a discrete subgroup of $\widetilde G$ and $\chi$ be a character of $\widetilde\Gamma$. 
We say that a distribution $T\in \left(\mathcal{V}_{\lambda, \ell}^{\infty}\right)^*$
is {\it automorphic} for $\widetilde{\Gamma}$ 
with character $\chi$ if
\begin{equation}
 \left(\pi^*_{\lambda, \ell}(\tilde{\gamma})T\right)(f) 
 = \chi({\tilde\gamma}) T(f) 
  \qquad \text{for}\;\; \tilde{\gamma}\in 
\widetilde{\Gamma}, f\in \mathcal{V}_{\lambda, \ell}^{\infty}.
\end{equation}
When the character $\chi$ is trivial, we simply say that $T$ is automorphic for $\widetilde\Gamma$. 
\end{definition}

In Lemma~\ref{lem:GrowthOfPtrans}, we saw that 
the Poisson transform is
an intertwining operator between $\left(\mathcal{V}_{\lambda, \ell}^{\infty}\right)^*$ 
and the space of eigenfunctions on $\mathcal{H}$ with respect to $\Delta_{\ell/2}$
of moderate growth, and therefore, 
the image of an automorphic distribution through the Poisson transform
is a Maass form.

\begin{theorem}
\label{thm:MaassFormFromAutomDist}
Let $T\in \left(\mathcal{V}_{\lambda, \ell}^{\infty}\right)^*$ be an automorphic distribution 
on $\mathcal{V}_{\lambda, \ell}^{\infty}$ for $\widetilde{\Gamma}_{0}(N)$
with character $\chi$. Then the Poisson transform 
$(\mathcal{P}_{\lambda, \ell}T)(z)$ of $T$ is
a Maass form for $\widetilde{\Gamma}_0(N)$ of weight $\ell/2$ with
character $\chi^{-1}$ and eigenvalue $(\lambda-\ell/4)(1-\lambda-\ell/4)$.
\end{theorem}

\begin{proof}
Let $T\in \left(\mathcal{V}_{\lambda, \ell}^{\infty}\right)^*$ is an automorphic distribution
for $\widetilde{\Gamma}_0(N)$ with character $\chi$. 
Then, by Lemma \ref{lem:GrowthOfPtrans} $(1)$ and $(2)$,  the Poisson transform of $T$ satisfies
\begin{gather*}
 \left(\mathcal{P}_{\lambda, \ell}T\big|_{\ell}\, \tilde{\gamma}\right)(z) 
  = \mathcal{P}_{\lambda, \ell} (\pi^*_{\lambda, \ell}(\tilde{\gamma}^{-1})T)(z)
  = \chi(\tilde\gamma)^{-1}\cdot  (\mathcal{P}_{\lambda, \ell}T)(z) 
 \quad (\tilde{\gamma}\in \widetilde{\Gamma}_0(N)), \\
 \Delta_{\ell/2} \cdot (\mathcal{P}_{\lambda, \ell}T)(z) 
  =  \left(\lambda-\frac{\ell}{4}\right)\left(1-\lambda-\frac{\ell}{4}\right) \cdot
 (\mathcal{P}_{\lambda, \ell}T)(z). 
\end{gather*}
Let $A \in SL_2(\Z)$. 
Then Lemma~\ref{lem:GrowthOfPtrans} $(1)$ and $(3)$ applied to $\pi^*_{\lambda,\ell}([A,1]^{-1})T$  imply that there exists a positive constant
$\nu$ such that 
\[
 (\mathcal{P}_{\lambda,\ell}T\big|_\ell[A,1])(z) 
   =  (\mathcal{P}_{\lambda,\ell}(\pi^*_{\lambda,\ell}([A,1]^{-1})T))(z) 
   = O\left(y^{\nu}\right) \quad (y \to \infty)
\]
uniformly with respect to $x$. 
This shows the moderate growth condition (iii).
\end{proof}

Theorem \ref{thm:MaassFormFromAutomDist} was proved in Unterberger \cite[Theorem 3.2]{U} in the case of the full modular group $SL_2(\Z)$ and $\ell=0$. 
Concrete examples of the correspondence between automorphic distributions and Maass forms are studied in Kato \cite{Kato} for Maass forms attached to zeta functions of real quadaratic fields with Gr\"o\ss encharacters, and in Shimeno \cite{Shimeno} and Unterberger \cite[Sections 2 and 3]{U} for the Eisenstein series of $SL_2(\Z)$. 

\section{Local functional equations on $\mathcal{V}_{\lambda,\ell}^{\infty}$}

In this section we consider the local zeta functions $\int |x|_\pm^{s-1} f(x)\,dx$ and  $\int |t|_\pm^{s-1} \mathcal{F}f(t)\,dt$ for $f \in \mathcal{V}_{\lambda,\ell}^\infty$, and  prove the local functional equation which relate these local zeta fucntions. 
Here $\mathcal{F}f$ denotes the Fourier transform of $f$. 
If $f$ is a rapidly decreasing function, then the local zeta functions and the local  functional equation  are well-studied (\cite{GS}, \cite{Igusa}, \cite{PVBook}).  
Our purpose here is to extend them for arbitrary $f \in \mathcal{V}_{\lambda,\ell}^\infty$.   

If $\Re(\lambda) \leq \frac 12$, then a function $f$ in $\mathcal{V}_{\lambda,\ell}^\infty$ is not necessarily absolutely integrable and the integral $\int_\R f(x) \exp(2\pi i xt)\,dx$ defining the Fourier transform may not converge. 
So we start to extend the definition of the Fourier transformation.

\subsection{Fourier transforms of functions in $\mathcal{V}_{\lambda,\ell}^{\infty}$}

For simplicity, we put $\mathbf{e}[x]=\exp(2\pi i x)$. 
Let $f(x)$ be a function in  $\mathcal{V}_{\lambda,\ell}^{\infty}$. 
Then $f(x)=O(\abs{x}^{-2\Re(\lambda)})$ $(\abs{x}\rightarrow\infty)$. 
Hence, if $\Re(\lambda)>\frac12$, then $f(x)$ is absolutely integrable and the Fourier transform  
\[
\widehat{f}(t) = \int_\R f(x)\mathbf{e}[xt]\,dx
\]
is well-defined. 
To define the Fourier transformation for functions in $\mathcal{V}_{\lambda,\ell}^{\infty}$ with $\Re(\lambda) \leq \frac12$, we use the following lemma. 

\begin{lemma}
\label{lem:FT0}
If $\Re(\lambda)>\frac12$ and $t \ne 0$, then we have 
\[
\widehat{f}(t) = \left(\frac{-1}{2\pi i t}\right)^m \widehat{f^{(m)}}(t)
\]
for $m \geq 0$. 
\end{lemma}

\begin{definition}
\label{dfn:FT}
For $t\ne0$, we define the Fourier transform $\mathcal{F}f(t)$ of $f \in \mathcal{V}_{\lambda,\ell}^\infty$ by 
\[
\mathcal{F}f(t) = \left(\frac{-1}{2\pi i t}\right)^m \widehat{f^{(m)}}(t),  
\]
where $m$ is a positive integer chosen so that $\Re(\lambda)+\frac m2 > \frac 12$. 
Then, by Lemma \ref{lem:diff of line model}, $f^{(m)}$ is in $\mathcal{V}_{\lambda+m/2,\ell+2m}^\infty$ and $\widehat{f^{(m)}}(t)$ is defined.  By Lemma \ref{lem:FT0}, $\mathcal{F}f(t)$ is independent of the choice of $m$. 
\end{definition}

Note that we use the symbol $\widehat{f}$ to denote the Fourier transform of an integrable function $f$, and $\mathcal{F}f$ to denote the Fourier transform in the generalized sense in Definition \ref{dfn:FT}. 
Of course we have $\mathcal{F}f=\widehat{f}$ for an integrable function $f$. 

\begin{remark}
When $\Re(\lambda)\leq\frac12$, the limit $\lim_{t\rightarrow0}\mathcal{F}f(t)$ does not exist in general. 
We shall define later a regularized value $\mathcal{F}f(0)$ for $\lambda\not\in\frac12-\frac12\Z_{\geq0}$. 
\end{remark}

\begin{lemma}
\label{lem:FT01}
For $f \in \mathcal{V}_{\lambda,\ell}^\infty$, the Fourier transform $\mathcal{F}f$ is a $C^\infty$-function on $\R^\times$ and satisfies 
\begin{equation}
(\mathcal{F}f)^{(k)}(t) = (2\pi i)^k \mathcal{F}(x^kf)(t) \quad (t \ne 0)
\label{eqn:derivative of Ff}
\end{equation}遯ｶ�｢
for any non-negative integer $k$. 
\end{lemma}

\begin{proof}
It is enough to prove that $\mathcal{F}f$ is differentiable and the identity \eqref{eqn:derivative of Ff} holds for $k=1$. 
First note that $xf(x) \in \mathcal{V}_{\lambda-1/2,\ell+2}^\infty$ and $xf(x)=O(|x|^{-2\Re(\lambda)+1})$ $(|x| \rightarrow \infty)$. 
Hence, if $\Re(\lambda)>1$, then the function 
\[
\frac{\partial f(x)\mathbf{e}[xt]}{\partial t}
 = 2\pi i  x f(x)\mathbf{e}[xt]
\]
is an integrable function of $x$ and we have 
\[
(\mathcal{F}f)'(t)= (\widehat{f})'(t)
  = 2\pi i \widehat{(xf)}(t)
  = 2\pi i \mathcal{F}(xf)(t).
\]

Now let $\lambda$ be arbitrary. 
Choose a positive integer $m$ larger than $2-2\Re(\lambda)$. Then we have 
\[
\mathcal{F}f(t) = \left(\frac{-1}{2\pi i t}\right)^m \widehat{f^{(m)}}(t).
\]
By the choice of $m$ and what we have seen above, the function $\widehat{f^{(m)}}(t)$ is differentiable and 
\[
\frac{d}{dt} \widehat{f^{(m)}}(t) = 2\pi i \widehat{(xf^{(m)})}(t).
\]
Hence, 
\begin{eqnarray*}
\frac{d}{dt} \mathcal{F}f(t) 
  &=& 2\pi i \left(m \left(\frac{-1}{2\pi i t}\right)^{m+1} \widehat{f^{(m)}}(t) 
         +  \left(\frac{-1}{2\pi i t}\right)^m \widehat{xf^{(m)}}(t) \right) \\
  &=& 2\pi i \left(\frac{-m}{2\pi i t} \mathcal{F}{f}(t) 
         +  \left(\frac{-1}{2\pi i t}\right)^m \widehat{xf^{(m)}}(t) \right).
\end{eqnarray*}遯ｶ�｢
On the other hand, we have 
\[
\frac{d^m}{dx^m} (xf(x)) = m f^{(m-1)}(x) + xf^{(m)}(x)
\]
and therefore 
\begin{eqnarray*}
\mathcal{F}(xf)(t) 
 &=& \left(\frac{-1}{2\pi i t}\right)^m 
          \left(m \widehat{f^{(m-1)}}(t) + \widehat{xf^{(m)}}(t)\right) \\
 &=& \frac{-m}{2\pi i t}\mathcal{F}f(t) + \left(\frac{-1}{2\pi i t}\right)^m \widehat{xf^{(m)}}(t).
\end{eqnarray*}
This shows that $ \mathcal{F}f$ is differentiable and the identity \eqref{eqn:derivative of Ff} holds for $k=1$. 
\end{proof}

\begin{lemma}
\label{lem:FT1}
For every non-negative integer $m$ and every sufficiently large $N$, there exists a constant $C_{m,N}$ such that  the inequality  
\[
\left|(\mathcal{F}f)^{(m)}(t)\right| \leq C_{m,N}\cdot \nu_{N,\lambda-m/2}(x^m f) |t|^{-N} 
\]
holds for $t\ne0$. 
Hence, for a fixed $t\ne0$, the evaluation map $f \mapsto \mathcal{F}f(t)$ is  continuous and defines a distribution on $\mathcal{V}_{\lambda,\ell}^\infty$. 
\end{lemma}

\begin{proof}
By Lemma \ref{lem:FT01} it is enough to prove the case where $m=0$. 
For any positive integer $N > 1-2\Re(\lambda)$, it follows from  the inequality \eqref{eqn:estimate of m-th derivative} that  
\[
\left|\mathcal{F}f(t)\right| 
 \leq  \frac{1}{(2\pi|t|)^N} \int_\R \left|f^{(N)}(x)\right|\,dx 
 \leq  \frac{C \nu_{N,\lambda}(f) }{(2\pi|t|)^N} \int_\R \frac{1}{(1+x^2)^{\Re(\lambda)+N/2}}\,dx 
 =  C_N\cdot \nu_{N,\lambda}(f) |t|^{-N},
\]
where $C_N=\frac{C}{(2\pi)^N} \int_\R \frac{1}{(1+x^2)^{\Re(\lambda)+N/2}}\,dx < \infty$.
\end{proof}

\subsection{Local zeta functions on $\mathcal{V}_{\lambda,\ell}^{\infty}$}

We denote by $|x|_\pm^s$  the functions given by 
\[
|x|_+^s = \begin{cases}
             x^s & (x>0), \\
             0 & (x \leq 0),
\end{cases}
\quad 
|x|_-^s = \begin{cases}
             0 & (x \geq 0), \\
             |x|^s & (x<0).
\end{cases}
\]
For $f \in \mathcal{V}_{\lambda,\ell}^\infty$, we consider the local zeta functions defined by 
\[
\Phi_\pm(f;s)=\int_{\R} |x|_\pm^{s-1}f(x)\,dx=\int_0^\infty x^{s-1}f(\pm x)\,dx.
\]
Since $f$ is in $C^\infty(\R)$ and $f(x)=O(|x|^{-2\Re(\lambda)})$ $(|x| \to \infty)$, the integrals defininig $\Phi_\pm(f;s)$ converge absolutely only when $0<\Re(s)$ and $\Re(s)<2\Re(\lambda)$. 
Therefore the integrals have no meaning unless $\Re(\lambda)>0$. 
To define the local zeta functions $\Phi_\pm(f;s)$ for arbitrary $\lambda$, we use 
the identity 
\begin{equation}
\Phi_\pm (f;s) = \frac{(\mp1)^{m}}{s(s+1)\cdots(s+m-1)}\cdot 
                      \Phi_\pm(f^{(m)};s+m) \quad (m \geq 1),
\label{eqn:anal cont of Phi}
\end{equation}
which is valid for $\Re(\lambda)>0$ and $0<\Re(s)<2\Re(\lambda)$. 
This identity is obtained by applying the relation
\[
\Phi_\pm (f;s) = \left[\frac{x^sf(\pm x)}{s}\right]_0^\infty \mp \frac1s\int_0^\infty x^{s} f'(\pm x)\,dx 
 = \mp \frac1s \cdot \Phi_\pm(f';s+1) 
\]
repeatedly. 
Note that, by Lemma \ref{lem:diff of line model}, we have $f^{(m)}(x)=O(|x|^{-2\Re(\lambda)-m})$  $(|x| \to \infty)$, and hence the integral on the right hand side of \eqref{eqn:anal cont of Phi} converges absolutely for  
$-m<\Re(s)<2\Re(\lambda)$.

\begin{definition}
\label{dfn:LZ}
For $\lambda \in \C$ (not necessarily with positive real part), take a positive integer $m$ such that $-m<2\Re(\lambda)$, and we define the local zeta functions $\Phi_\pm (f;s)$ for $f \in \mathcal{V}_{\lambda,\ell}^\infty$ by the expression on the right hand side of  \eqref{eqn:anal cont of Phi}. 
Then, since $\Phi_\pm (f;s)$ is determined independently of the choice of $m$, and $m$ can be arbitrary large, $\Phi_\pm (f;s)$ are defined for $\Re(s)<2\Re(\lambda)$ as meromorphic functions of $s$.  
Note that, with this definition, the identity  \eqref{eqn:anal cont of Phi} holds without any restriction of $\Re(s)$ and $\Re(\lambda)$. 
\end{definition}

\begin{theorem}
\label{thm:loc zeta}
For  $f \in \mathcal{V}_{\lambda,\ell}^\infty$, the local zeta functions $\Phi_\pm (f;s)$ have analytic continuations to meromorphic functions of $s$ in $\C$ with poles only at $s=-k, 2\lambda+k$ $(k \in \Z_{\geq0})$. 
The orders of the poles are at most $1$ and the residues are given by 
\begin{eqnarray}
\mathop{\mathrm{Res}}_{s=-k} \Phi_\pm (f;s) 
 &=&  \frac{(\pm1)^{k}}{k!} f^{(k)}(0), 
\label{eqn:residue of Phi at -k}\\ 
\mathop{\mathrm{Res}}_{s=2\lambda+k} \Phi_\pm (f;s) 
 &=&  -(\mp 1)^{-\ell/2} \frac{ (\mp1)^{k} (f_\infty)^{(k)}(0)}{k!}.
\label{eqn:residue of Phi at 2lambda+k}
\end{eqnarray}
unless $\lambda = -\frac q2 \in -\frac12\Z_{\geq0}$ and $0 \leq k \leq q$. 
In the exceptional case the poles at $s=-(q-k)$ and $s=2\lambda+k$ $(0 \leq k \leq q)$ coincide  and the residue there is given by 
\begin{equation}
\mathop{\mathrm{Res}}_{s=-(q-k)} \Phi_\pm (f;s) 
 = \frac{(\pm1)^{q-k}}{(q-k)!} f^{(q-k)}(0) - (\mp 1)^{-\ell/2} \frac{ (\mp1)^{k} (f_\infty)^{(k)}(0)}{k!}.
\label{eqn:residue of Phi exceptional case}
\end{equation}
\end{theorem}

\begin{proof}
We consider the integrals 
\[
\Phi_{\pm,0}(f;s) = \int_0^1 x^{s-1}f(\pm x)\,dx. \quad 
\Phi_{\pm,\infty}(f;s) = \int_1^\infty x^{s-1}f(\pm x)\,dx \quad (f \in \mathcal{V}_{\lambda,\ell}^\infty). 
\]
The integrals $\Phi_{\pm,0}(f;s)$ and $\Phi_{\pm,\infty}(f;s)$ converge absolutely for $\Re(s)>0$ and $\Re(s)<2\Re(\lambda)$, respectively.  
By the definition of $f_\infty$, we have 
\begin{equation}
\Phi_{\pm,\infty}(f;s)=(\mp1)^{-\ell/2}\Phi_{\mp,0}(f_\infty;2\lambda-s)
\label{eqn:Phi0toInf}
\end{equation}
for $\Re(s)<2\Re(\lambda)$. 
For a positive integer $n$, we write
\[
f(\pm x)=\sum_{k=0}^{n-1} \frac{(\pm1)^kf^{(k)}(0)}{k!} x^k + f_n(\pm x).
\]
Here $f_n(x)=O(x^n)$ $(x \rightarrow 0)$. 
Then, for $\Re(s)>0$, 
\begin{eqnarray*}
\Phi_{\pm,0}(f;s) 
 &=& \sum_{k=0}^{n-1} \frac{(\pm1)^kf^{(k)}(0)}{k!} \int_0^1 x^{s+k-1}\,dx
     +\int_0^1 x^{s-1} f_n(x)\,dx \\
 &=& \sum_{k=0}^{n-1} \frac{(\pm1)^kf^{(k)}(0)}{k!}\cdot\frac{1}{s+k}
     +\int_0^1 x^{s-1} f_n(x)\,dx.
\end{eqnarray*}
Since the integral involving $f_n$ converges absolutely for $\Re(s)>-n$, and $n$ can be arbitrary large, $\Phi_{\pm,0}(f;s)$ has an analytic continuation to a meromorphic function of $s$ in $\C$ with poles of order at most $1$ only at $s=-k$ $(k \in \Z_{\geq0})$. 
The residue at $s=-k$ is equal to $(\pm1)^kf^{(k)}(0)/(k!)$.   
The same argument yields that $\Phi_{\pm,\infty}(f;s)=i^{-\ell}\Phi_{\mp,0}(f_\infty;2\lambda-s)$ has an analytic continuation to a meromorphic function of $s$ in $\C$ with poles of order at most $1$ only at $s=2\lambda+k$ $(k \in \Z_{\geq0})$. 
The residue at $s=2\lambda+k$ is equal to $-(\mp1)^{-\ell/2}(\mp1)^k(f_\infty)^{(k)}(0)/(k!)$. 

For any positive integer $m$, we have 
\[
\Phi_{\pm,0}(f;s) 
 = \sum_{k=0}^{m-1} \frac{(\mp1)^k f^{(k)}(\pm1)}{s(s+1)\cdots(s+k)} 
       + \frac{(\mp1)^m}{s(s+1)\cdots(s+m)} \int_0^1 x^{s+m-1}f^{(m)}(\pm x)\,dx
\]
for $\Re(s)>0$, and 
\[ 
\Phi_{\pm,\infty}(f;s) 
 = -\sum_{k=0}^{m-1} \frac{(\mp1)^k f^{(k)}(\pm1)}{s(s+1)\cdots(s+k)} 
       + \frac{(\mp1)^m}{s(s+1)\cdots(s+m)} \int_1^\infty x^{s+m-1}f^{(m)}(\pm x)\,dx
\]
for $\Re(s)<2\Re(\lambda)$. 
These are  analogues of the identity  \eqref{eqn:anal cont of Phi} for $\Phi_{\pm,0}$ and $\Phi_{\pm,\infty}$, and proved by repeated use of integration by parts.
The integral on the right hand sides of the upper (resp.\ lower) identity 
converges absolutely for $\Re(s)>-m$ (resp.\  $2\Re(\lambda)> \Re(s)$). 
This shows that the identity
\[
\Phi_{\pm,0}(f;s) + \Phi_{\pm,\infty}(f;s) 
 = \frac{(\mp1)^m}{s(s+1)\cdots(s+m)} \int_0^\infty x^{s+m-1}f^{(m)}(\pm x)\,dx
\]
holds for $2\Re(\lambda)> \Re(s)>-m$. 
This proves that 
\begin{equation}
\Phi_\pm(f;s) = \Phi_{\pm,0}(f;s) + \Phi_{\pm,\infty}(f;s),  
\label{eqn:Phi= 0+inf}
\end{equation}遯ｶ�｢
and the theorem follows from what we have proved for  $\Phi_{\pm,0}(f;s)$ 
and $\Phi_{\pm,\infty}(f;s)$.
\end{proof}

\begin{corollary}
For $f \in \mathcal{V}_{\lambda,\ell}^{\infty}$, we have 
\begin{equation}
\label{form:PhiFinfty}
\Phi_\pm(f;s)=(\mp1)^{-\ell/2}\Phi_{\mp}(f_\infty;2\lambda-s).
\end{equation}
\end{corollary}

\begin{proof}
By \eqref{eqn:square of infty-operation}, \eqref{eqn:Phi0toInf}
and \eqref{eqn:Phi= 0+inf}, we have 
\begin{eqnarray*}
\Phi_\pm(f;s) &=& \Phi_{\pm,0}(f;s) + \Phi_{\pm,\infty}(f;s) \\
     &=&   (\mp1)^{-\ell/2}\Phi_{\mp,\infty}(f_\infty;2\lambda-s) 
       +  (\mp1)^{-\ell/2}\Phi_{\mp,0}(f_\infty;2\lambda-s) \\
     &=& (\mp1)^{-\ell/2}\Phi_{\mp}(f_\infty;2\lambda-s).
\end{eqnarray*}
\end{proof}

\begin{lemma}
For a fixed $s \ne -k, 2\lambda+k$ $(k  \in \Z_{\geq0})$, the linear functional 
on $\mathcal{V}_{\lambda,\ell}^\infty$ defined by $f \mapsto \Phi_\pm(f;s)$ is continuous and defines a distribution on $\mathcal{V}_{\lambda,\ell}^\infty$. 
\label{lem:continuity of LZ}
\end{lemma}

\begin{proof}
First assume that $\Re(s)<2\Re(\lambda)$. 
Choose a positive integer $m$ such that  $\Re(s) > -m$. 
Then, by definition, we have
\begin{equation}
\Phi_\pm(f;s) = \frac{(\mp1)^m}{s(s+1)\cdots(s+m-1)}\Phi_\pm(f^{(m)};s+m).  
\label{eqn:AC neg}
\end{equation}
By the estimate \eqref{eqn:estimate of m-th derivative}, 
there exists a positive constant $C$ such that 
\begin{eqnarray*}
\left|\Phi_\pm(f^{(m)};s+m)\right| 
 &\leq & \int_0^\infty \left|f^{(m)}(\pm x)\right| |x|^{\Re(s)+m-1}\,dx \\
 &\leq & C \nu_{m,\lambda}(f) \int_0^\infty \frac{|x|^{\Re(s)+m-1}}{(1+x^2)^{\Re(\lambda)+m/2}}\,dx. 
\end{eqnarray*}
By the choice of $m$, the integral converges and the mapping $f \mapsto \Phi_\pm(f^{(m)};s+m)$ is continuous, and hence the mapping  $f \mapsto \Phi_\pm(f;s)$ is continuous.

Since the analytic continuation of $\Phi_\pm(f;s)$ preserves the continuity (the discussion in \cite[Chapter 1, Appendix 2 and Appendix A]{GS} and \cite[\S 5.2]{Igusa} can be modified in a form applicable to the present situation), 
the mapping  $f \mapsto \Phi_\pm(f;s)$ is continuous for any $\lambda$ and $s$ different from the poles. 
\end{proof}

\subsection{Regularized values of $(\mathcal{F}f)^{(k)}$ at  $0$}
\label{subsect:regularized value}

In \S 2.2, we have defined the Fourier transform $\mathcal{F}f(t)$ $(f \in \mathcal{V}_{\lambda,\ell}^\infty)$ for $t \ne 0$. If $\Re(\lambda) \leq \frac 12$, then $\mathcal{F}f(t)$ is not necessarily  bounded near $t=0$. 
However, it is convenient to introduce the (regularized) value of the $k$-th derivative of $\mathcal{F}f(t)$ at $t=0$ by
\begin{equation}
 \label{form:DefOfFourierF0}
 (\mathcal{F}f)^{(k)}(0):= (2\pi i)^k\left(\Phi_{+}(f;k+1)+(-1)^k \Phi_{-}(f;k+1)\right).
\end{equation}
If $\Re(\lambda) > (k+1)/2$, then, $(\mathcal{F}f)^{(k)}(0)$ coincides with the value at $t=0$ of the $k$-th derivative of $\widehat f(t)$.
Since $\Phi_\pm(f;s)$ are holomorphic at $s=k+1$ if $\lambda \not\in \frac{k+1}2-\frac12\Z_{\geq0}$, the value $(\mathcal{F}f)^{(k)}(0)$ is well-defined for any such $\lambda$. 
Moreover, by Lemma \ref{lem:continuity of LZ}, the mapping $f \mapsto (\mathcal{F}f)^{(k)}(0)$ is continuous. 

\subsection{Analytic continuation of $\mathcal{F}f$ with respect to $\lambda$}

To calculate the Fourier transform $\mathcal{F}f(t)$ for $\lambda$ with small real part, it is convenient to use the method of analytic continuation. 

Let $U$ be a connected open subset of $\C$. 
For $r \in \R$, we put $U_r=\{\lambda \in U\; | \;\Re(\lambda)>r\}$ and assume that $U_{1/2}$ is not empty. 

\begin{proposition}
\label{prop:AC of F}
Let $f_{\lambda}(x)$ be a function on $U \times \R$ satisfying the conditions
\begin{enumerate}
\def\labelenumi{(\alph{enumi})}
\item
For each $\lambda$, $f_{\lambda}(x)$ belongs to $\mathcal{V}_{\lambda,\ell}^\infty$ as a function of $x$. 
\item
For each $x$, $f_{\lambda}(x)$ is a holomorphic function of $\lambda$ in $U$. 
\end{enumerate}
We assume that,  for any compact subset $K$ in $U$ and any sufficiently large positive integer $m$,   
there exists a constant $C_{K,m}$ such that 
\begin{equation}
 \abs{\frac{\partial^m f_\lambda}{\partial x^m}(x) }, 
 \abs{\frac{\partial^m (f_\lambda)_\infty}{\partial x^m}(x) } 
    \leq C_{K,m} (1+x^2)^{-\Re(\lambda)-m/2}
\quad (\forall\lambda \in K,\ x \in \R).
\label{eqn:loc bdd for m-the der}
\end{equation}
Then, if $t \ne 0$, the Fourier transform $\mathcal{F} f_\lambda(t)$ as a function of $\lambda$ is holomorphic in $U$. If $t=0$, the Fourier transform $\mathcal{F} f_\lambda(0)$ as a function of $\lambda$ is holomorphic in $U\setminus\left(\frac12-\frac12\Z_{\geq0}\right)$. 
\end{proposition}

\begin{proof}
First we consider the case $t \ne 0$.  
Then, by Definition \ref{dfn:FT}, $\mathcal{F}f_\lambda(t)$ is given by the integral 
\[
\left(\frac{-1}{2\pi i t}\right)^m \int_\R \frac{\partial^m f_\lambda}{\partial x^m}(x) 
 \mathbf{e}[xt]\,dt
\]
for sufficiently large $m$. 
Hence, by the assumption \eqref{eqn:loc bdd for m-the der} for $f_\lambda$, the integral above is convergent uniformly on $K \cap U_{(1-m)/2}$.  
This shows that $\mathcal{F}f_\lambda(t)$ is holomorphic on $U_{(1-m)/2}$. 
Since $m$ can be arbitrary large, $\mathcal{F}f_\lambda(t)$ is holomorphic on $U$. 

Next we consider the case $t=0$. 
Again by \eqref{eqn:loc bdd for m-the der} for $(f_\lambda)_\infty$,  we see that the convergence of  the local zeta function 
\[
\Phi_\mp\left(\frac{\partial^m (f_\lambda)_\infty}{\partial x^m};2\lambda-s+m\right)
\]
is  locally uniform in $2\Re(\lambda)+m>\Re(s)>0$. 
Hence  
\begin{eqnarray*}
\Phi_\pm(f_\lambda;s) 
 &=& (\mp1)^{-\ell/2}\Phi_\mp((f_\lambda)_\infty;2\lambda-s) \\
 &=& \frac{(\mp1)^{-\ell/2}(\pm1)^m}{(2\lambda-s)(2\lambda-s+1)\cdots(2\lambda-s+m-1)}\Phi_\mp\left(\frac{\partial^m (f_\lambda)_\infty}{\partial x^m};2\lambda-s+m\right)
\end{eqnarray*}
is a meromorphic function of $(\lambda,s)$ in  $\{(\lambda,s) \in U\times \C\; | \; 2\Re(\lambda)+m>\Re(s)>0\}$ with possible poles at $s=2\lambda, 2\lambda+1,\ldots,2\lambda+m-1$. Since $m$ can be arbitrary large,  $\Phi_\pm(f_\lambda;s)$ can be extended meromorphically to  $\{(\lambda,s) \in U\times \C\; | \; \Re(s)>0\}$. 
Hence, we see that $(\mathcal{F}f_\lambda)(0)=\Phi_{+}(f_\lambda;1)+ \Phi_{-}(f_\lambda;1)$ is a holomorphic function of $\lambda$ in $U\setminus\left(\frac12-\frac12\Z_{\geq0}\right)$. 
\end{proof}

By the proposition above, if one can calculate the Fourier transform $\int_\R f_\lambda(x)\mathbf{e}[xt]\,dt$ explicitly for $\lambda$ with sufficiently  large real part, then the explicit expression obtained is valid for any $\lambda \in U$ by the principle of analytic continuation.

Now we give examples of applications of Proposition \ref{prop:AC of F}. 

\begin{example}
\label{ex:Lambda-family}
Let $f$ be a function in $\mathcal{V}_{\lambda_0,\ell}^\infty$. 
Put $f_{\lambda}(x)=(1+x^2)^{-\lambda+\lambda_0}f(x)$. 
Then $f_{\lambda}(x)$ satisfies the conditions (a) and (b) in Proposition \ref{prop:AC of F}. 
Since 
\[
\sup_{x\in\R}\left\{(1+x^2)^{\Re(\lambda)+m/2}\abs{\frac{\partial^m f_\lambda}{\partial x^m}(x) } \right\} 
 \leq \nu_{m,\lambda}(f_\lambda) = \nu_{m,\lambda_0}(f), 
\]
and $\nu_{m,\lambda}(f_\lambda)=\nu_{m,\lambda}((f_\lambda)_\infty)$, the famlily $f_{\lambda}(x)$ satisfies also the condition \eqref{eqn:loc bdd for m-the der}, and Proposition \ref{prop:AC of F} can be applied to $f_\lambda$. 
Hence $\mathcal{F}f$ is obtained as the value at $\lambda=\lambda_0$ of the analytic continuation of $\mathcal{F}f_\lambda$. 
For a fixed $u \in \R$, we put $g_\lambda(x)=f_\lambda(x-u)$. 
Then Proposition \ref{prop:AC of F} can be applied also to $g_\lambda$. 
If $\Re(\lambda)>\frac12$, then 
\[
\mathcal{F}(g_\lambda)(t) 
 = \widehat{g_\lambda}(t) = e^{2\pi i ut} \widehat{f_\lambda}(t) 
 = e^{2\pi i ut} \mathcal{F}(f_\lambda)(t).
\]
By  Proposition \ref{prop:AC of F} and the principle of analytic continuation, this identity holds for any $\lambda$, and in parcular, the elementary formula 
\begin{equation}
\mathcal{F}g(t) = e^{2\pi i ut} \mathcal{F}f(t), \quad 
g(x)=f(x-u)=\pi_{\lambda,\ell}(\tilde{n}(u))f(x)\quad (u \in \R)
\label{eqn:translate for FT}
\end{equation}
holds for $f \in \mathcal{V}_{\lambda_0,\ell}^\infty$. 
Similarly, we can prove 
\begin{equation}
(\mathcal{F}f)(ty) = t^{-1}\left(\mathcal{F}(f_{t^{-1}})\right)(y)
\quad (t \ne 0)
\label{eqn:t-mult}
\end{equation}
for $f \in \mathcal{V}_{\lambda_0,\ell}^\infty$ and $\lambda\not\in\frac12-\frac12\Z_{\geq0}$. 

\end{example}

\begin{example}
\label{example:Fourier transform of Poisson kernel}
As the second example, let us consider the Poisson kernel 
\[
p_{\lambda,\ell,z}(u)=\frac{\Im(z)^{\lambda-(\ell/4)}}{|z-u|^{2\lambda-(\ell/2)}(z-u)^{\ell/2}} \quad (\Im(z)>0) 
\]
defined in \S 1.4. 
Then the Fourier transform of $p_{\lambda,\ell,z}$ is given by 
\begin{equation}
(\mathcal{F}p_{\lambda, \ell, z})(t) 
 = \begin{cases}
\dfrac{i^{-\ell/2} \pi^{\lambda} \cdot |t|^{\lambda-1}}
{\Gamma\left(\lambda+\frac{\sgn(t)\ell}{4}\right)} \cdot 
 y^{-\ell/4}\, W_{\frac{\sgn(t)l}{4}, \lambda-\frac{1}{2}}\left(4\pi|t|y\right)\cdot 
 e^{2\pi i tx}
& (t\ne0), \\[4ex]
y^{1-\lambda-(\ell/4)}\cdot i^{-\ell/2} \cdot 
 \dfrac{(2\pi) 2^{1-2\lambda} \Gamma(2\lambda-1)}
 {\Gamma\left(\lambda+\frac{\ell}{4}\right)
 \Gamma\left(\lambda-\frac{\ell}{4}\right)}
 & (t=0),
\end{cases}
\label{eqn:Fourier transform of Poisson kernel}
\end{equation}
where $W_{\mu, \nu}(x)$ is the Whittaker function. 
This can be proved also by using Proposition \ref{prop:AC of F}.  
Indeed, if $\Re(\lambda)>\frac12$, then we have 
\begin{align*}
 (\mathcal{F}p_{\lambda, \ell, z})(t) 
 &= y^{\lambda-(\ell/4)} \int_{-\infty}^{\infty}
 |z-u|^{-2\lambda+(\ell/2)}\cdot (z-u)^{-\ell/2} \cdot e^{2\pi i tu} du \\
 &= y^{1-\lambda-(\ell/4)}  \cdot e^{2\pi i tx} \int_{-\infty}^{\infty} 
  (1+u^2)^{-\lambda+(\ell/4)} \cdot (-u+i)^{-\ell/2} \cdot e^{2\pi i tyu}\, du \\
 &= i^{-\ell/2} y^{1-\lambda-(\ell/4)}  \cdot e^{2\pi i tx} \int_{-\infty}^{\infty} 
  (1+u^2)^{-\lambda} \cdot \left(\frac{1-iu}{1+iu}\right)^{\ell/4} \cdot e^{2\pi i tyu}\, du.
\end{align*}
Here we have chosen the principal branch of the logarithm to define the power function in the integrand.  
Now the identity \eqref{eqn:Fourier transform of Poisson kernel} for $\Re(\lambda)>\frac12$ follows from the integral formula 
\[
 \int_{-\infty}^{\infty} \frac{e^{ixu/2}}
 {(u^2 +1)^{\nu+1/2}} \cdot 
 \left(\frac{1-iu}{1+iu}\right)^{\mu} du
 = \begin{cases}
 \dfrac{\pi(x/4)^{\nu-1/2}}{\Gamma(\mu+\nu+\frac{1}{2})}\cdot W_{\mu, \nu}(x)
   & (x>0), \\[3ex]
 \dfrac{2\pi \Gamma(2\nu) 2^{-2\nu}}{\Gamma(\mu+\nu+\frac{1}{2})
 \Gamma(-\mu+\nu+\frac{1}{2})} & (x=0), 
 \end{cases}
\]
which is derived from \cite[p.~119 (12)]{Erdelyi} for $x>0$, and from \cite[p.~12 (30)]{Erdelyi-HTF} for $x=0$. 
Note that the integral on the left hand side is absolutely convergent for $\Re(\nu)>0$. 
Since the Whittaker function $W_{\mu,\nu}(x)$ is an entire function of $\nu$, the Fourier transform $(\mathcal{F}p_{\lambda, \ell, z})(t)$ is an entire function of $\lambda$  for $t\ne0$, and a meromorphic function of $\lambda$ with poles at $\lambda=\frac{1-k}2$ $(k \in \Z_{\geq0})$ for $t=0$.  
Therefore, once we prove that  Proposition \ref{prop:AC of F} is applicable to $p_{\lambda,\ell,z}$, then the identity \eqref{eqn:Fourier transform of Poisson kernel} holds for any $\lambda$. 
To see this, it is enough to show the estimate \eqref{eqn:loc bdd for m-the der} for $\frac{\partial^m}{\partial u^m}p_{\lambda,\ell,z}(u)$, since 
 $(p_{\lambda,\ell,z})_\infty(u)=z^{-\ell/2}p_{\lambda,\ell,-1/z}(u)$. 
By an elementary calculation, we have
\[
\frac{\partial^m }{\partial u^m}p_{\lambda,\ell,z}(u)
 = \Im(z)^{\lambda-(\ell/4)} 
    \sum_{k=0}^m \frac{P_k(\lambda,\Re(z),u)}{|z-u|^{2\lambda-(\ell/2)+2k}(z-u)^{\ell/2+m+k}},
\]
where $P_k(\lambda,\Re(z),u)$ are polynomials of $\lambda$, $\Re(z)$ and $u$, and  of degree at most $2k$ with respect to $u$. 
Therefore 
\[
\frac{\partial^m }{\partial u^m}p_{\lambda,\ell,z}(u)
 = O\left(\abs{u}^{-2\Re(\lambda)-m}\right) \quad (\abs{u}\rightarrow \infty)
\]
and we can take the $O$-constant not depending on $\lambda$ when $\lambda$ is restricted to a compact subset in $\C$. 
Thus we may apply Proposition \ref{prop:AC of F} to $p_{\lambda,\ell,z}$ and the proof of \eqref{eqn:Fourier transform of Poisson kernel} is completed. 
\end{example}

\begin{remark}
The explicit formula for the Fourier transform of the Poisson kernel \eqref{eqn:Fourier transform of Poisson kernel} was given previously in \cite[(1.5), (4.10)]{SuzukiWeil}.
\end{remark}

\subsection{Dual local zeta functions and local functional equations}
\label{subsect:dual LZ and LFE}

For $f \in \mathcal{V}_{\lambda,\ell}^\infty$, we define the dual local zeta functions by 
\[
\Phi_\pm(\mathcal{F}f;s)=\int_{\R} |x|_\pm^{s-1}\mathcal{F}f(x)\,dx=\int_0^\infty x^{s-1}\mathcal{F}f(\pm x)\,dx.
\]
It follows from Definition \ref{dfn:FT} that the estimate $\mathcal{F}f(t)= O(|t|^{-m})$ $(\abs{t}\rightarrow 0)$ holds for any $m$ with $2\Re(\lambda)+m>1$, 
and by Lemma \ref{lem:FT1}, the Fourier transform $\mathcal{F}f(t)$ is rapidly decreasing at infinity. 
Hence, $\Phi_\pm(\mathcal{F}f;s)$ converge absolutely for $\Re(s)>m_0$, where $m_0$ is the smallest non-negative integer satisfying $m_0>1-2\Re(\lambda)$. 

\begin{theorem}
\label{thm:LFE}
$(1)$ For $f \in \mathcal{V}_{\lambda,\ell}^\infty$, the dual local zeta functions $\Phi_\pm(\mathcal{F}f;s)$ have analytic continuations to meromorphic functions of $s$ in $\C$. 
The poles are located only at $s=-k, 1-2\lambda-k$ $(k \in \Z_{\geq0})$. 

$(2)$ The following local functional equation holds:
\begin{equation}
\label{form:LFE}
\begin{pmatrix}
\Phi_{+}(\mathcal{F}f; s) \\[4pt]
\Phi_{-}(\mathcal{F}f; s) 
\end{pmatrix}
 = (2\pi)^{-s} \Gamma(s)
    \begin{pmatrix}
    e^{{\pi is}/{2}} & e^{-{\pi is}/{2}} \\[4pt]
    e^{-{\pi is}/{2}} & e^{{\pi is}/{2}} 
    \end{pmatrix}
    \begin{pmatrix}
    \Phi_{+}(f; 1-s) \\[4pt]
    \Phi_{-}(f; 1-s) 
    \end{pmatrix}.
\end{equation}

$(3)$ If $\lambda \not\in \frac12\Z$, then all the poles of $\Phi_\pm(\mathcal{F}f;s)$ are of order at most $1$. 
If $\lambda=\frac q2$ $(q\in \Z)$, then the poles at $s=\min\{0,-(q-1)\}-k$ $(k \in \Z_{\geq 0})$ are of order at most $2$, and the other poles are of order at most $1$. 
Except for the (possible double) poles at $s=\min\{0,-(q-1)\}-k$ $(k \in \Z_{\geq 0})$ for $\lambda = \frac q2$ $(q \in \Z)$,  
the residues at the poles are given by 
\begin{eqnarray*}
\mathop{\mathrm{Res}}_{s=-k} \Phi_\pm(\mathcal{F}f;s)
 &=& \frac{(\pm1)^k(\mathcal{F}f)^{(k)}(0)}{k!}, \\
\mathop{\mathrm{Res}}_{s=1-2\lambda-k} \Phi_\pm(\mathcal{F}f;s)
 &=&
 (2\pi)^{2\lambda+k-1} \Gamma(1-2\lambda-k) \cdot \frac{(f_\infty)^{(k)}(0)}{k!} \\
 & & {} \times (\pm i)^{k-1}  \left(e^{\pm \lambda \pi i}  - i^{-\ell} e^{\mp \lambda\pi i}  \right).  
\end{eqnarray*}
\end{theorem}

\begin{proof}
First assume that $\Re(\lambda)>\frac12$ and  $0<\Re(s)<1$. 
Then both sides of the functional equation \eqref{form:LFE} converge absolutely, and  $\mathcal{F}f(t)$ is given by the convergent integral $\widehat f(t)=\int_\R f(x)\mathbf{e}[xt]\,dx$. 
Hence, the dominated convergence theorem implies that
\begin{align*}
\Phi_{\pm}(\mathcal{F}f;s)
&=\int_0^{\infty} t^{s-1} \left(
\int_{\R} f(x) e^{\pm 2\pi i xt}\,dx\right)dt \\
&=\lim_{y\to+0} \int_0^{\infty} t^{s-1} e^{-ty}\left(
\int_{\R} f(x)  e^{\pm 2\pi i xt}\,dx\right)dt.
\end{align*}
Further, by Fubini's theorem, we have
\begin{align*}
\Phi_{\pm}(\mathcal{F}f;s)
 &=\lim_{y\to +0} \int_{\R} f(x)\left(
 \int_0^{\infty} t^{s-1} e^{-t(y \mp 2\pi i x)}dt\right)dx \\
 &=\Gamma(s)\cdot \lim_{y\to +0}\int_{\R} (y \mp 2\pi i x)^{-s} f(x) dx.
\end{align*} 
Note that if $0<\Re(s)<1$, then $|x|^{-s}f(x)$ is absolutely integrable, 
and we can apply the dominated convergence theorem again. 
Since
\[
 \lim_{y\to +0} (y\mp 2\pi i x)^{-s} 
  = \left(e^{\mp \mathrm{sgn}(x) \pi i/2} |2\pi x|\right)^{-s},  
\]
we have
\begin{align*}
\Phi_{\pm}(\mathcal{F}f;s)
 &= (2\pi)^{-s} \Gamma(s) \int_\R 
 \left(e^{\pm \pi is/2} |x|_+^{-s} + e^{\mp \pi is/2} |x|_-^{-s}\right) f(x)\,dx \\
 &=(2\pi)^{-s}\cdot \Gamma(s)\cdot 
 \left(e^{\pm \pi is/2} \Phi_{+}(f;1-s)+
 e^{\mp \pi is/2} \Phi_{-}(f;1-s)\right).
\end{align*}
Thus the functional equation  \eqref{form:LFE} has been proved for $\Re(\lambda)>\frac12$ and $1>\Re(s)>0$. 
Since, by Theorem \ref{thm:loc zeta}, $\Phi_{\pm}(f;1-s)$ can be extended to $\C$ meromorphically, the functional equation implies that, if $\Re(\lambda)>\frac12$, then  $\Phi_{\pm}(\mathcal{F}f;s)$ can also be extended to $\C$ meromorphically. 

Let $\lambda$ be arbitrary. 
Then, by taking a positive integer $m$ with $\Re(\lambda)+\frac m2 > \frac12$, the dual local zeta function is given by
\[
\Phi_\pm(\mathcal{F}f;s)
 = \int_\R |t|_\pm^{s-1} \left(\frac{-1}{2\pi i t}\right)^m \mathcal{F}(f^{(m)})(t)\,dt
 = \left(\frac{\mp1}{2\pi i}\right)^m \Phi_\pm(\mathcal{F}(f^{(m)});s-m).
\]
Since $f^{(m)}$ is in $\mathcal{V}_{\lambda+m/2,\ell+2m}^\infty$ and $\Re(\lambda)+\frac m2>\frac12$, the local functional equation  \eqref{form:LFE} holds for $\Phi_\pm(\mathcal{F}(f^{(m)});s-m)$, and so we have
\begin{eqnarray*}
\Phi_\pm(\mathcal{F}f;s)
 &=& \left(\frac{\mp1}{2\pi i}\right)^m (2\pi)^{m-s} \Gamma(s-m) \\
 & & {} \times
       \left(
       e^{\pm \pi i (s-m)/2} \Phi_+(f^{(m)};1+m-s) 
       + e^{\mp \pi i (s-m)/2} \Phi_-(f^{(m)};1+m-s) 
       \right).
\end{eqnarray*} 
Using \eqref{eqn:anal cont of Phi} (valid without any restriction of $\Re(\lambda), \Re(s)$), we rewrite the right hand side to obtain 
\begin{eqnarray*}
\Phi_\pm(\mathcal{F}f;s)
 &=& (\pm i)^m (2\pi)^{-s} \Gamma(s-m) \\
 & & {} \times \frac{\Gamma(m+1-s)}{\Gamma(1-s)}
       \left(
       (-1)^m e^{\pm \pi i (s-m)/2} \Phi_+(f;1-s) 
       +  e^{\mp \pi i (s-m)/2} \Phi_-(f;1-s) 
       \right) \\
 &=& (\pm i)^m (2\pi)^{-s} (-1)^m \Gamma(s) \\
 & & {} \times 
       \left(
      (-i)^{\mp m} e^{\pm \pi i s/2} \Phi_+(f;1-s) 
       +  i^{\pm m} e^{\mp \pi i s/2} \Phi_-(f;1-s) 
       \right) \\
 &=& (2\pi)^{-s} \Gamma(s) 
       \left(
      e^{\pm \pi i s/2} \Phi_+(f;1-s) 
       + e^{\mp \pi i s/2} \Phi_-(f;1-s) 
       \right).
\end{eqnarray*} 
This shows that the local functional equation holds without any restriction on $\lambda$, and the dual local zeta functions $\Phi_{\pm}(\mathcal{F}f;s)$ can be extended to $\C$ meromorphically.  

We now examine the location of poles. 
Possible poles of the right hand side of the local functional equation \eqref{form:LFE} are the poles of $\Gamma(s)$ and $\Phi_\pm(f;1-s)$. 
Hence, by Theorem \ref{thm:loc zeta},  the poles of $\Phi_{\pm}(\mathcal{F}f;s)$ are contained in 
\[
\{-k, 1-2\lambda -k, k+1\; | \; k \in \Z_{\geq0}\}.
\]
The possible poles at $s=k+1$ $(k \in \Z_{\geq0})$ are not actual poles unless $\lambda=-\frac q2 \in -\frac12\Z_{\geq0}$ and $0 \leq k \leq q$, since from Theorem \ref{thm:loc zeta} we have
\begin{eqnarray*}
\mathop{\mathrm{Res}}_{s=k+1} \Phi_{\pm}(\mathcal{F}f;s) 
 &=& (2\pi)^{-(k+1)}\Gamma(k+1) \\
 & & {} \times  \left(e^{\pm (k+1)\pi i/2} \mathop{\mathrm{Res}}_{s=k+1} \Phi_+(f;1-s) + e^{\mp (k+1)\pi i/2} \mathop{\mathrm{Res}}_{s=k+1} \Phi_-(f;1-s)\right) \\
 &=& -(2\pi)^{-(k+1)}\Gamma(k+1)\frac{f^{(k)}(0)}{k!} \left(e^{\pm (k+1)\pi i/2}  + (-1)^ke^{\mp (k+1)\pi i/2} \right)\\
 &=& 0.
\end{eqnarray*}
If $\lambda=-\frac q2$ $(q \in \Z_{\geq0})$, then $\{k+1\;|\;0\leq k \leq q\}$ is included in $\{1-2\lambda-k\;|\;k  \in \Z_{\geq0}\}$.  
Thus the poles of $\Phi_\pm(\mathcal{F}f;s)$ are contained in $\{-k\;|\;k \in \Z_{\geq0}\} \cup \{1-2\lambda-k\;|\;k \in \Z_{\geq0}\}$, and this proves the first assertion in the theorem. 

Since all the poles of $\Gamma(s)$ and $\Phi_\pm(f;1-s)$ are of order $1$, the order of a pole of $\Phi_\pm(\mathcal{F}f;s)$ at $s=s_0$ can be equal to $2$, only when $s_0$ belongs to the set  
\[
A := \{-k\;|\;k \in \Z_{\geq0}\} \cap \{1-2\lambda-k\;|\;k \in \Z_{\geq0}\}
    = \begin{cases}
      \min\{0,-(q-1)\}-\Z_{\geq0} & (\lambda=\frac q2,\ q \in \Z), \\
      \emptyset & (\lambda \not\in \frac12\Z).
      \end{cases} 
\] 
The poles of  $\Phi_\pm(\mathcal{F}f;s)$ not belonging to $A$ are of order at most $1$.

If $-k \not\in A$, then $s=-k$ is a pole of  $\Phi_\pm(\mathcal{F}f;s)$ coming from a pole of $\Gamma(s)$ and the residue is given by 
\begin{eqnarray*}
\mathop{\mathrm{Res}}_{s=-k} \Phi_{\pm}(\mathcal{F}f;s) 
 &=& (2\pi)^{k} \cdot \mathop{\mathrm{Res}}_{s=-k} \Gamma(s) 
      \left(e^{\mp k\pi i/2} \Phi_+(f;k+1) + e^{\pm k\pi i/2}  \Phi_-(f;k+1)\right) \\
 &=& (2\pi)^{k} \cdot \frac{(-1)^k}{k!} \cdot e^{\mp k\pi i/2}  
  \left(\Phi_+(f;k+1) + (-1)^{k}  \Phi_-(f;k+1)\right) \\
 &=& \frac{(\pm1)^k(\mathcal{F}f)^{(k)}(0)}{k!}.
\end{eqnarray*}
If $1-2\lambda-k \not\in A$, then $s=1-2\lambda-k$ is a pole of  $\Phi_\pm(\mathcal{F}f;s)$ coming from a pole of $\Phi_\pm(f;1-s)$ and the residue is given by 
\begin{eqnarray*}
\mathop{\mathrm{Res}}_{s=1-2\lambda-k} \Phi_{\pm}(\mathcal{F}f;s) 
 &=& (2\pi)^{2\lambda+k-1} \Gamma(1-2\lambda-k) 
 \left(e^{\pm (1-2\lambda-k)\pi i/2} \mathop{\mathrm{Res}}_{s=1-2\lambda-k} \Phi_+(f;1-s) \right. \\
 & & {} + \left. e^{\mp(1-2\lambda-k)\pi i/2} \mathop{\mathrm{Res}}_{s=1-2\lambda-k}  \Phi_-(f;1-s)\right) \\
 &=& -(2\pi)^{2\lambda+k-1} \Gamma(1-2\lambda-k) 
 \left(e^{\pm (1-2\lambda-k)\pi i/2} \mathop{\mathrm{Res}}_{s=2\lambda+k} \Phi_+(f;s) \right. \\
 & & {} + \left. e^{\mp(1-2\lambda-k)\pi i/2} \mathop{\mathrm{Res}}_{s=2\lambda+k}  \Phi_-(f;s)\right) \\
 &=& -(2\pi)^{2\lambda+k-1} \Gamma(1-2\lambda-k) \frac{(f_\infty)^{(k)}(0)}{k!} \\
 & & {} \times  \left(e^{\pm (1-2\lambda-k)\pi i/2} i^{-\ell}(-1)^{k+1}
      - e^{\mp(1-2\lambda-k)\pi i/2} \right) \\
 &=& (2\pi)^{2\lambda+k-1} \Gamma(1-2\lambda-k) \frac{(f_\infty)^{(k)}(0)}{k!} \\
 & & {} \times (\pm i)^{k-1}  \left(e^{\pm \lambda \pi i}  - i^{-\ell} e^{\mp \lambda\pi i}  \right). 
\end{eqnarray*}
This proves the theorem.
\end{proof}

\begin{remark}
Analytic continuations 
 of the local zeta functions $\Phi_{\pm}(f; s)$ and $\Phi_{\pm}(\mathcal{F}f; s)$ for $f \in \mathcal{V}_{\lambda,\ell}^\infty$ and the functional equation \eqref{form:LFE}  are included in a more general result due to 
Lee~\cite{Lee16}, in which Lee discussed the local zeta functions associated with prehomogeneous vector spaces of commutative parabolic type. 
However, our results on the singularities seems to be new. 
\end{remark}

\subsection{Order of magnitude of local zeta functions}

\begin{lemma}
\label{lem:EstimatesOfGamma}
For $a_1,a_2, \epsilon \in \R$ with $a_1<a_2$ and $\epsilon>0$, put 
\[
D=D_{a_1,a_2,\epsilon}=\{s\in\C\; | \;a_1\leq \Re(s) \leq a_2,\ |\Im(s)|\geq 2|\Im(\lambda)|+\epsilon\}.
\]
Let $f \in \mathcal{V}^\infty_{\lambda,\ell}$. 
Then, for any positive integer $N$,  there exists a positive constant $C_N$ for which the inequalities 
\[
|\Phi_\pm(f;s)| < C_N|\Im(s)|^{-N}, \quad
|\Phi_\pm(\mathcal{F}f;s)| < C_N|\Im(s)|^{-N}
\]
hold for any $s \in D$.
\end{lemma}

\begin{proof}
First note that, for positive integers $m$ and $n$, 
\begin{eqnarray*}
\Phi_\pm(f;s) 
 &=& \frac{(\mp1)^{m}}{s(s+1)\cdots(s+m-1)}\cdot \Phi_\pm(f^{(m)};s+m) \\
 &=& \frac{(\mp1)^{-\ell/2}}{s(s+1)\cdots(s+m-1)}\cdot \Phi_\mp((f^{(m)})_\infty;2\lambda-m-s) \\
 &=& \frac{(\mp1)^{-\ell/2}(\pm1)^n}{s(s+1)\cdots(s+m-1)(2\lambda-m-s)(2\lambda-m-s+1)\cdots(2\lambda-m-s+n)}\\
 & & \quad {}\times \Phi_\mp(\left((f^{(m)})_\infty\right)^{(n)};2\lambda-m+n-s). 
\end{eqnarray*}
The rightmost expression gives an analytic continuation of $\Phi_\pm(f;s)$ in the domain $0 < 2\Re(\lambda)-m+n-\Re(s) < 2\Re(\lambda)+m+n$, equivalently  
$-2m < \Re(s) < 2\Re(\lambda)-m+n$. 
We take $m,n$ so large that $-2m<a_1<a_2<   2\Re(\lambda)-m+n$. 
Then, for any $s$ with $a_1 \leq \Re(s) \leq a_2$, we have
\begin{eqnarray*}
\lefteqn{
\left|\Phi_\mp(\left((f^{(m)})_\infty\right)^{(n)};2\lambda-m+n-s) \right| 
} \\
 & & \leq \int_0^\infty x^{2\Re(\lambda)-\Re(s)-m+n-1} \left|\left((f^{(m)})_\infty\right)^{(n)}(\mp x)\right|\,dx \\
 & & \leq \int_0^1 x^{2\Re(\lambda)-a_2-m+n-1} \left|\left((f^{(m)})_\infty\right)^{(n)}(\mp x)\right|\,dx \\
 & & \quad {} + \int_1^\infty x^{2\Re(\lambda)-a_1-m+n-1} \left|\left((f^{(m)})_\infty\right)^{(n)}(\mp x)\right|\,dx.  
\end{eqnarray*}
Since the integrals $\int_0^1$ and $\int_1^\infty$ are convergent,  
we get a constant $C_{a_1,a_2}$ independent of $s$ such that 
\[
\left|\Phi_\mp(\left((f^{(m)})_\infty\right)^{(n)};2\lambda-m+n-s) \right|
   < C_{a_1,a_2} \quad (a_1  \leq \Re(s) \leq a_2).
\]
Hence, we have 
\begin{eqnarray*}
\left|\Phi_\pm(f;s) \right| 
 &\leq& C_{a_1,a_2} |\Im(s)|^{-m}|\Im(s)-2\Im(\lambda)|^{-n} \\
 &\leq& C_{a_1,a_2} \left(1-\frac{2|\Im(\lambda)|}{2|\Im(\lambda)|+\epsilon}\right)^{-n} |\Im(s)|^{-(m+n)}
\end{eqnarray*}
for any $s \in D$. 
Since $m,n$ can be arbitrary large, the estimate in the lemma is proved for $\Phi_\pm(f;s)$. 
If we notice the estimate  
\[
\Gamma(s) e^{\pm \pi i s/2} = O\left(|\Im(s)|^{a_2-1/2}\right) \quad (a_1 \leq \Re(s) \leq a_2)
\]
as $|\Im(s)| \rightarrow \infty$, which is an immediate consequence of Stirling's estimate for $\Gamma(s)$,  
the estimate for $\Phi_\pm(\mathcal{F}f;s)$ follows from the functional equation 
\eqref{form:LFE} and the estimate for $\Phi_\pm(f;s)$. 
\end{proof}

\section{Summation formulas and $L$-functions}

In this section, we always assume that  $\lambda\not\in \frac 12 - \frac 12 \Z_{\geq0}$. 
This assumption allows us to define $\mathcal{F}f(0)$ for $f \in \mathcal{V}_{\lambda,\ell}^\infty$ (see \S \ref{subsect:regularized value}). 

\subsection{$L$-functions associated with Ferrar-Suzuki summation formulas}

For complex sequences  
$\{\alpha(n)\}, \{\beta(n)\} \, (n \in \Z)$  
of polynomial growth,  complex numbers $\alpha(\infty), \beta(\infty)$, 
and a positive integer $N$,  
we consider the Ferrar-Suzuki summation formula of level $N$
\begin{eqnarray}
\lefteqn{
\alpha(\infty)f(\infty)+\sum_{n=-\infty}^{\infty} \alpha(n)(\mathcal{F}f)(n)
}  \nonumber \\ 
 & & =
\beta(\infty) f_{\infty}(\infty)+
 \sum_{n=-\infty}^{\infty} \beta(n)(\mathcal{F}f_{\infty})\left(\frac{n}{N}\right) 
\quad (f\in \mathcal{V}_{\lambda,\ell}^{\infty}),
\label{form:GivenSumFormula}
\end{eqnarray}
where we put
\begin{equation}
\label{form:DefOfFInfty}
f(\infty):= f_{\infty}(0).  
\end{equation}
Note that, by \eqref{eqn:square of infty-operation}, we have $f_{\infty}(\infty)= i^{\ell}\cdot f(0)$.

First we give an interpretation of the summation formula as  an  automorphic property of distributions. 
Define two linear functionals $T$ and $T_\infty$ on $\mathcal{V}^\infty_{\lambda,\ell}$ by 
\begin{eqnarray*}
T(f) 
  &=& \alpha(\infty)f(\infty)+\sum_{n=-\infty}^{\infty} \alpha(n)(\mathcal{F}f)(n), \\
T_\infty(f) 
  &=& \beta(\infty)f(\infty)+\sum_{n=-\infty}^{\infty} \beta(n)(\mathcal{F}f)\left(\frac{n}{N}\right). 
\end{eqnarray*}
The summation formula can be rewritten as $T(f)=T_\infty(f_\infty)$. 
Hence, by \eqref{form:FInftyIsPiTau} and \eqref{form:DefOfActionOnDualSpace}, we have $T_\infty=\pi^*_{\lambda,\ell}(\tilde{w}^{-1})T$.

\begin{lemma}  
\label{lemma:distribution interpretation of sum formula}
Let $\widetilde{\Delta}(N)$ be the subgroup of $\widetilde{\Gamma}_0(N)$ generated by $\tilde{n}(1)$ and $\tilde{\bar n}(N)$. 
Then the linear functionals $T$ and $T_\infty$ are distributions on $\mathcal{V}^\infty_{\lambda,\ell}$ automorphic for $\widetilde\Delta(N)$ and $\tilde{w}^{-1}\widetilde\Delta(N)\tilde{w}$, respectively. 
\end{lemma}

We call $T$  the {\it automorphic distribution for $\widetilde\Delta(N)$ associated with the summation formula}\/  \eqref{form:GivenSumFormula}. 

\begin{proof}
Since $T_\infty=\pi^*_{\lambda,\ell}(\tilde{w}^{-1})T$ and $\tilde{w}\tilde{n}(-N)\tilde{w}^{-1}=\tilde{\bar n}(N)$ (see \eqref{eqn:commutation of w and nx}), it is enough to prove that 
\begin{enumerate}
\def\labelenumi{$(\arabic{enumi})$}
\item $T$ is a distribution and $\pi^*_{\lambda,\ell}(\tilde{n}(-1))T=T$, and 
\item $T_\infty$ is a distribution and $\pi^*_{\lambda,\ell}(\tilde{n}(-N))T_\infty=T_\infty$.
\end{enumerate}
We give a proof only for (1). 
Since $\{\alpha(n)\}$ is assumed to be of polynomial growth, we can find a positive constant $C$ and a non-negative integer $k$ such that $|\alpha(n)|<C(1+|n|)^k$ for all $n \in \Z$. 
By Lemma~\ref{lem:FT1}, there exists a positive constant
$C_{0, k+2}$ such that 
\[
 |(\mathcal{F}f)(n)|\leq \frac{C_{0, k+2}}{|n|^{k+2}}\cdot \nu_{k+2,\lambda}(f)\qquad
 \text{for \; $n\in \Z\setminus\{0\}$ \; and \; $f\in \mathcal{V}_{\lambda, \ell}^{\infty}$}.
\]
From this estimate we have 
\begin{align*}
 \left|\sum
\begin{Sb}
n=-\infty \\
n\neq 0 
\end{Sb}^{\infty}
 \alpha(n) (\mathcal{F}f)(n)\right|
&\leq 
\sum
\begin{Sb}
n=-\infty \\
n\neq 0 
\end{Sb}^{\infty}
|\alpha(n)|\cdot |(\mathcal{F}f)(n)| \\
&\leq CC_{0, k+2}\cdot \nu_{k+2,\lambda}(f)
\sum
\begin{Sb}
n=-\infty \\
n\neq 0 
\end{Sb}^{\infty}
\frac{(1+|n|)^k}{|n|^{k+2}}<+\infty.
\end{align*}
This shows that $\sum_{n\ne0} \alpha(n) \mathcal{F}f(n)$ is absolutely convergent and the mapping 
\[
f \longmapsto \sum_{n\ne0} \alpha(n) \mathcal{F}f(n)
\] 
defines a continuous linear functional on $\mathcal{V}^\infty_{\lambda,\ell}$. 
Moreover, as stated in \S \ref{subsect:regularized value}, if $\lambda\not\in \frac 12 - \frac12\Z_{\geq0}$, 
the mapping $\mathcal{V}_{\lambda, \ell}^{\infty}\ni f\longmapsto
(\mathcal{F}f)(0) \in \C$ is continuous. 
Since 
\[
 \abs{f(\infty)}=\abs{f_{\infty}(0)} \leq \nu_{0,\lambda}(f_\infty) = \nu_{0,\lambda}(f), 
\]
the mapping $\mathcal{V}_{\lambda, \ell}^{\infty}\ni f \longmapsto f(\infty) \in \C$ 
is also continuous. 
Thus $T$ defines a distribution in $(\mathcal{V}_{\lambda,\ell}^{\infty})^*$. 
By \eqref{eqn:translate for FT}, we have 
\[
\mathcal{F}(\pi_{\lambda,\ell}(\tilde{n}(1))f)(n) = \mathcal{F}f\left(n\right).
\]
for any $n \in \Z$. 
Since
\begin{equation}
(\pi_{\lambda,\ell}(\tilde{n}(1))f)_\infty(x) = (\mathrm{sgn}\,(1+x))^{\ell/2} |1+x|^{-2\lambda} f_\infty\left(\frac{x}{1+x}\right),
\label{eqn:fN-infty}
\end{equation} 
it is obvious that $(\pi_{\lambda,\ell}(\tilde{n}(1))f)(\infty)=(\pi_{\lambda,\ell}(\tilde{n}(1))f)_\infty(0)=f_\infty(0)=f(\infty)$. 
Thus we conclude that $(\pi_{\lambda,\ell}^*(\tilde{n}(-1))T)(f)=T(\pi_{\lambda,\ell}(\tilde{n}(1))f)=T(f)$. 
\end{proof}

Now, for complex sequences  $\{\alpha(n)\}$ and $\{\beta(n)\}$ of polynomial growth,   we define the $L$-functions $\xi_{\pm}(\alpha;s)$, $\xi_{\pm}(\beta;s)$  by
\begin{equation}
\label{eqn:def of xi}
 \xi_{\pm}(\alpha;s)= \sum_{n=1}^{\infty} \frac{\alpha(\pm n)}{n^s}, \quad
 \xi_{\pm}(\beta;s) = \sum_{n=1}^{\infty} \frac{\beta(\pm n)}{n^s}.
\end{equation}
Since $\alpha(n)$ and $\beta(n)$ are assumed to be of polynomial growth, the $L$-functions converge absolutely if $\Re(s)$ are sufficiently large. 
We also define the zeta integrals with test function $f \in \mathcal{V}^\infty_{\lambda,\ell}$ by 
\begin{eqnarray}
\label{eqn:def of Z}
 Z(\alpha, \mathcal{F}f;s) &=& 
 \int_{0}^{\infty} t^{s-1}
 \sum\begin{Sb}
  n=-\infty\\
  n\neq 0
  \end{Sb}^{\infty}
  \alpha(n) \cdot \mathcal{F}f\left(tn\right) dt, \\
 Z^{(N)}(\beta, \mathcal{F}f_\infty;s) &=& 
 \int_{0}^{\infty} t^{s-1}
 \sum\begin{Sb}
  n=-\infty\\
  n\neq 0
  \end{Sb}^{\infty}
  \beta(n) \cdot (\mathcal{F}f_\infty)\left(\frac{tn}{N}\right) dt. 
\label{eqn:def of ZN}
\end{eqnarray}

\begin{lemma}
\label{lem:Int rpr of zeta}
For any $f \in \mathcal{V}_{\lambda,\ell}^\infty$, the zeta integrals $Z(\alpha,\mathcal{F}f;s)$ and $Z^{(N)}(\beta, \mathcal{F}f_\infty;s)$ are absolutely convergent if $\Re(s)$ is sufficiently large, and the following identities hold:
\begin{eqnarray*}
Z(\alpha, \mathcal{F}f;s) 
 &=& \xi_{+}(\alpha;s)\Phi_{+}(\mathcal{F}f;s)
   + \xi_{-}(\alpha;s)\Phi_{-}(\mathcal{F}f;s), \\
Z^{(N)}(\beta, \mathcal{F}f_\infty;s) 
 &=& N^s\cdot \xi_{+}(\beta;s)\Phi_{+}(\mathcal{F}f_\infty;s)
   + N^s\cdot \xi_{-}(\beta;s)\Phi_{-}(\mathcal{F}f_\infty;s),  
\end{eqnarray*}
where $\Phi_\pm(\mathcal{F}f;s)$ and  $\Phi_\pm(\mathcal{F}f_\infty;s)$ are the dual local zeta functions introduced in \S\ref{subsect:dual LZ and LFE}. 
\end{lemma}

\begin{proof}
If $f\in \mathcal{V}_{\lambda, \ell}^{\infty}$, then  
$\Phi_{\pm}(\mathcal{F}f;s)$ converge absolutely for $\Re(s)> m_0$, where $m_0$ is the smallest non-negative integer satisfying $m_0 > 1-2\Re(\lambda)$, as we have seen in \S \ref{subsect:dual LZ and LFE}. 
Hence the following standard (formal) calculation is justified by Fubini's theorem:
\begin{eqnarray*}
\lefteqn{
 Z(\alpha, \mathcal{F}f;s) 
} \\
 & & = 
 \int_{0}^{\infty} t^{s-1} \sum
 \begin{Sb}
  n=-\infty\\
  n\neq 0
  \end{Sb}^{\infty} \alpha(n) (\mathcal{F}f)\left(tn\right)dt \\
 & & = 
\sum_{n=1}^{\infty} \int_0^{\infty} t^{s-1}
 \cdot \alpha(n)(\mathcal{F}f)\left(tn\right) dt + \sum_{n=1}^{\infty} 
 \int_0^{\infty} t^{s-1} \cdot \alpha(-n)(\mathcal{F}f)\left(-tn\right)
dt \\
 & & = 
 \left(\sum_{n=1}^{\infty} \frac{\alpha(n)}{n^{s}}\right)\cdot 
 \int_0^{\infty} t^{s-1} (\mathcal{F}f)(t)dt +
 \left(\sum_{n=1}^{\infty} \frac{\alpha(-n)}{n^s}\right)\cdot 
 \int_{-\infty}^{0} |t|^{s-1} (\mathcal{F}f)(t)dt \\[5pt]
 & & = 
 \xi_{+}(\alpha;s)\Phi_{+}(\mathcal{F}f;s) 
     + \xi_{-}(\alpha;s)\Phi_{-}(\mathcal{F}f;s).
\end{eqnarray*}
This also shows that the zeta integral $Z(\alpha,\mathcal{F}f;s)$ converges absolutely if $\Re(s)$ is sufficiently large. The proof for $Z^{(N)}(\beta, \mathcal{F}f_\infty;s)$ is quite similar.
\end{proof}

Now, assuming the summation formula \eqref{form:GivenSumFormula},  we prove the functional equation of the $L$-functions following the method of Suzuki \cite{SuzukiAD} and Tamura \cite{Tamura}.

For $t>0$ and $f\in \mathcal{V}_{\lambda,\ell}^{\infty}$, we put
\begin{equation}
 f_{t}(x)= f(tx).
\end{equation}
Then, by \eqref{eqn:t-mult} and  
\begin{align*}
 \left(f_{t^{-1}}\right)_{\infty}(x) &
 =(\sgn x)^{\ell/2} |x|^{-2\lambda} \cdot f
 \left(-\frac{1}{tx}\right) 
 = t^{2\lambda}\cdot f_{\infty}(tx)= t^{2\lambda}\cdot (f_{\infty})_{t}(x),
\end{align*}
we have
\[
 \left(
 \mathcal{F}\left(f_{t^{-1}}\right)_{\infty}\right)(y)=
 t^{2\lambda}\cdot \left(\mathcal{F}((f_{\infty})_{t})\right)(y) =
 t^{2\lambda-1} \cdot (\mathcal{F}f_{\infty})(t^{-1}y).
\]
Further, it follows from the definition \eqref{form:DefOfFInfty} that
\begin{align*}
\left(f_{t^{-1}}\right)(\infty) &= \left(f_{t^{-1}}\right)_{\infty}(0) =
t^{2\lambda}\cdot (f_{\infty})_{t}(0) =
t^{2\lambda}\cdot f_{\infty}(t\cdot 0)=
t^{2\lambda}\cdot f(\infty), \\
\left(f_{t^{-1}}\right)_{\infty}(\infty) &=i^{\ell}\cdot f_{t^{-1}}(0) = 
i^{\ell} \cdot f(t^{-1}\cdot 0)=f_{\infty}(\infty). 
\end{align*}
We substitute $f_{t^{-1}}$ for $f$ in
\eqref{form:GivenSumFormula},
and multiply the result by $t^{-1}$ to obtain
\begin{eqnarray}
\nonumber
\lefteqn{
 t^{2\lambda-1}\cdot \alpha(\infty)f(\infty)+ \sum_{n=-\infty}^{\infty}
 \alpha(n) (\mathcal{F}f)(tn)
} \\
 & & =
 t^{-1}\cdot \beta(\infty)f_{\infty}(\infty)+ t^{2\lambda-2}
 \sum_{n=-\infty}^{\infty} \beta(n) \left(\mathcal{F}f_{\infty}\right)
 \left(\frac{t^{-1}n}{N}\right).
 \label{form:SumFormWithT>0}
\end{eqnarray}

\begin{lemma}
 \label{thm:FEofZetaIntegral1}
 Assume that $\lambda\not\in \frac12-\frac12\Z_{\geq0}$ and  the summation formula \eqref{form:GivenSumFormula} holds for any $f\in \mathcal{V}_{\lambda,\ell}^{\infty}$. 
Then the zeta integrals 
 $Z(\alpha,\mathcal{F}f; s)$ and $Z^{(N)}(\beta,\mathcal{F}f_{\infty};s)$
 have  analytic continuations to meromorphic functions in the whole $s$-plane, 
 and further, they satisfy the functional equation
 \[
  Z(\alpha,\mathcal{F}f;s)=Z^{(N)}(\beta,\mathcal{F}f_{\infty};2-2\lambda-s).
 \]
The poles of $Z(\alpha,\mathcal{F}f;s)$ and $Z^{(N)}(\beta,\mathcal{F}f_\infty;s)$ are of order at most $1$
and located only at $s=0, 1, 2-2\lambda, 1-2\lambda$,  
The residues at the poles are given by
\[
 \Res_{s=1}Z(\alpha,\mathcal{F}f;s) = \beta(\infty)f_{\infty}(\infty), 
 \quad 
 \Res_{s=1-2\lambda}Z(\alpha,\mathcal{F}f;s) = -\alpha(\infty)f(\infty), 
\]
and, if $\lambda\ne1$, then
\[
 \Res_{s=0}Z(\alpha,\mathcal{F}f;s) = -\alpha(0)(\mathcal{F}f)(0), \quad 
 \Res_{s=2-2\lambda}Z(\alpha,\mathcal{F}f;s) = \beta(0)(\mathcal{F}f_{\infty})(0). 
\]
If $\lambda=1$, then the poles at $s=0$ and $2-2\lambda$ coincide and the residue is given by 
\[
 \Res_{s=0} Z(\alpha,\mathcal{F}f;s)  
 = \beta(0)(\mathcal{F}f_{\infty})(0) - \alpha(0)(\mathcal{F}f)(0).
\]
\end{lemma}

\begin{proof}
We apply the standard Riemann trick. We put
\begin{align*}
Z_{+}(\alpha, \mathcal{F}f;s) &=
\int_{1}^{\infty} t^{s-1} \sum
  \begin{Sb}
  n=-\infty\\
  n\neq 0
  \end{Sb}^{\infty}
  \alpha(n) \cdot (\mathcal{F}f)\left(tn\right) dt, \\
Z_{-}(\alpha, \mathcal{F}f;s) &=
\int_{0}^{1} t^{s-1} \sum
  \begin{Sb}
  n=-\infty\\
  n\neq 0
  \end{Sb}^{\infty}
  \alpha(n) \cdot (\mathcal{F}f)\left(tn\right) dt, \\
Z^{(N)}_{+}(\beta, \mathcal{F}f_{\infty};s) &=
\int_{1}^{\infty} t^{s-1} \sum
  \begin{Sb}
  n=-\infty\\
  n\neq 0
  \end{Sb}^{\infty}
  \beta(n) \cdot (\mathcal{F}f_{\infty})\left(\frac{tn}{N}\right) dt, \\ 
Z^{(N)}_{-}(\beta, \mathcal{F}f_{\infty};s) &=
\int_{0}^{1} t^{s-1} \sum
  \begin{Sb}
  n=-\infty\\
  n\neq 0
  \end{Sb}^{\infty}
  \beta(n) \cdot (\mathcal{F}f_{\infty})\left(\frac{tn}{N}\right) dt.
\end{align*}
By Lemma \ref{lem:Int rpr of zeta}, the integrals $Z_{+}(\alpha, \mathcal{F}f;s)$ and 
$Z^{(N)}_{+}(\beta, \mathcal{F}f_{\infty};s)$ are absolutely convergent for 
any $s\in\C$, and define entire functions of $s$, 
while the integrals 
$Z_{-}(\alpha, \mathcal{F}f;s)$ and 
$Z^{(N)}_{-}(\beta,\mathcal{F}f_{\infty};s)$ converge absolutely 
if $\Re(s)$ is sufficiently large. 
It follows from the summation formula~\eqref{form:SumFormWithT>0}
that if $\Re(s)$ is sufficiently large,
\begin{align*}
Z_{-}(\alpha, \mathcal{F}f;s) 
 &=\int_0^1 t^{s+2\lambda-3} \sum
 \begin{Sb}
 n=-\infty \\
 n\neq 0
 \end{Sb}^{\infty}
 \beta(n) \left(\mathcal{F}f_{\infty}\right)
 \left(\frac{t^{-1}n}{N}\right)dt
 +\frac{\beta(0)(\mathcal{F}f_{\infty})(0)}{s+2\lambda-2}+
 \frac{\beta(\infty)f_{\infty}(\infty)}{s-1} \\[5pt]
 &\qquad \quad -
 \frac{\alpha(0)(\mathcal{F}f)(0)}{s}-
 \frac{\alpha(\infty)f(\infty)}{s+2\lambda-1}.
\end{align*}
By the change of the variable $t\mapsto t^{-1}$, we have
\begin{align*}
\int_0^1 t^{s+2\lambda-3} \sum
 \begin{Sb}
 n=-\infty \\
 n\neq 0
 \end{Sb}^{\infty}
 \beta(n) \left(\mathcal{F}f_{\infty}\right)
 \left(\frac{t^{-1}n}{N}\right)dt 
 &=\int_1^{\infty} t^{(2-2\lambda-s)-1}\sum
 \begin{Sb}
 n=-\infty \\
 n\neq 0
 \end{Sb}^{\infty}
 \beta(n) \left(\mathcal{F}f_{\infty}\right)
 \left(\frac{tn}{N}\right)dt \\
 &=Z^{(N)}_{+}(\beta, \mathcal{F}f_{\infty};2-2\lambda-s),
\end{align*}
and hence 
\begin{align}
\nonumber
Z(\alpha,\mathcal{F}f;s) &=
Z_{+}(\alpha,\mathcal{F}f;s)+Z_{-}(\alpha,\mathcal{F}f;s)\\
\nonumber
&=Z_{+}(\alpha,\mathcal{F}f;s)+
Z_{+}^{(N)}(\beta, \mathcal{F}f_{\infty};2-2\lambda-s)
+\frac{\beta(0)(\mathcal{F}f_{\infty})(0)}{s+2\lambda-2}+
 \frac{\beta(\infty)f_{\infty}(\infty)}{s-1} \\[5pt]
\label{form:AnalyConiZetaIntegral}
 &\qquad \quad -
 \frac{\alpha(0)(\mathcal{F}f)(0)}{s}-
 \frac{\alpha(\infty)f(\infty)}{s+2\lambda-1}.
\end{align}
Since the first two terms on the right hand side of the second identity 
are entire functions, the zeta integral $Z(\alpha, \mathcal{F}f;s)$ 
has an analytic continuation to a meromorphic function of $s$ 
on the whole $\C$ with possible poles (of order at most 1)
only at $s=0, 1, 2-2\lambda, 1-2\lambda$. 
By a similar calculation, we can prove that 
$Z(\beta, \mathcal{F}f_{\infty};2-2\lambda-s)$ is also equal to 
the last expression in \eqref{form:AnalyConiZetaIntegral}.
We therefore obtain the functional equation of the zeta integrals 
\[
Z(\alpha,\mathcal{F}f;s)=Z^{(N)}(\beta,\mathcal{F}f_{\infty};2-2\lambda-s).
\]
The residue formula follows immediately from the 
identity~\eqref{form:AnalyConiZetaIntegral}.
\end{proof}

\begin{theorem}
 \label{thm:FEofZeta1}
 Assume that $\lambda\not\in \frac12-\frac12\Z_{\geq0}$ and   the summation formula \eqref{form:GivenSumFormula} holds for any $f\in \mathcal{V}_{\lambda,\ell}^{\infty}$.  

$(1)$ The $L$-functions $\xi_{\pm}(\alpha;s)$ and 
 $\xi_{\pm}(\beta;s)$ have analytic continuations to meromorphic 
 functions on the whole $s$-plane 
 with  poles only at $s=2-2\lambda$ and $s=1$. 
The orders of the poles are at most $1$ and the residues are given by
\begin{gather*}
\mathop{\mathrm{Res}}_{s=1} \xi_\pm(\alpha;s) = i^\ell \beta(\infty), \\
\mathop{\mathrm{Res}}_{s=2-2\lambda} \xi_\pm(\alpha;s)  
 = \pm i (2\pi)^{-(2\lambda-1)} \Gamma(2\lambda-1)(i^\ell e^{\mp \lambda \pi  i} 
   - e^{\pm \lambda \pi  i})  \beta(0), \\
\mathop{\mathrm{Res}}_{s=1} \xi_\pm(\beta;s) = \frac{\alpha(\infty)}{N}, \\
\mathop{\mathrm{Res}}_{s=2-2\lambda} \xi_\pm(\beta;s)
 =  \mp i N^{2\lambda-2} (2\pi)^{-(2\lambda-1)}\Gamma(2\lambda-1)
      (i^{-\ell}e^{\pm\pi i \lambda} - e^{\mp \pi i \lambda})\alpha(0). 
\end{gather*}

$(2)$  Put
 \[
  \Xi_{\pm}(\alpha;s)=(2\pi)^{-s} \Gamma(s) \xi_{\pm}(\alpha;s), \qquad 
   \Xi_{\pm}(\beta;s)=(2\pi)^{-s} \Gamma(s) \xi_{\pm}(\beta;s)
 \]
 and 
 \begin{equation}
  \label{form:DefOfGammaAndSigma}
  \gamma(s) = 
  \begin{pmatrix}
  e^{\pi s i/2}  & e^{-\pi s i/2} \\
  e^{-\pi s i/2}  & e^{\pi s i/2} 
 \end{pmatrix}, \qquad
 \Sigma(\ell) =
 \begin{pmatrix}
 0 & i^{\ell} \\
 1 & 0
 \end{pmatrix}.
 \end{equation}
 Then the completed $L$-functions 
 $\Xi_{\pm}(\alpha;s)$ and $\Xi_{\pm}(\beta;s)$ satisfy the 
 functional equation
\begin{equation}
\label{eqn:FE of xi}
 \gamma(s)\left(
 \begin{array}{c}
 \Xi_{+}(\alpha;s) \\[2pt] \Xi_{-}(\alpha;s)
 \end{array}
 \right)\\
 = N^{2-2\lambda-s}\cdot 
 \Sigma(\ell)\cdot
 \gamma(2-2\lambda-s)\left(
 \begin{array}{c}
 \Xi_{+}(\beta;2-2\lambda-s) \\[2pt] \Xi_{-}(\beta;2-2\lambda-s)
 \end{array}
 \right).
\end{equation}

$(3)$  The functions
 $(s-1)(s+2\lambda-2)\xi_{\pm}(\alpha, s)$ and 
 $(s-1)(s+2\lambda-2)\xi_{\pm}(\beta, s)$
 are entire functions of order at most $1$.
\end{theorem}

\begin{remark}
The functional equation above is the simplest case of the result of Suzuki \cite[Theorem 1]{SuzukiAD} (see also Tamura \cite[Theorem 2]{Tamura}). 
However,  in these earlier works,  the residues at the poles of $L$-functions are not determined, since the set of test functions for the summation formula  is restricted to $C^\infty_0(\R^\times)$. 
\end{remark}

\begin{proof}
(1) Let $\varphi_0$ (resp.\ $\varphi_1$) be a function in $C^\infty_0(\R)$ whose support is contained in the set of positive (resp.\ negative) real numbers. 
We denote by $f_0$ (resp.\ $f_1$) a rapidly decreasing function on $\R$ such that $\mathcal{F}f_0=\varphi_0$ (resp,\ $\mathcal{F}f_1=\varphi_1$).  
Notice that  $f_j(\infty)=0$ and $\mathcal{F}f_j(0)=\varphi_j(0)=0$. 
Then we have 
\begin{equation}
Z(\alpha,\mathcal{F} f_0;s) = \xi_+(\alpha;s) \Phi_+(\mathcal{F} f_0;s), \quad 
Z(\alpha,\mathcal{F} f_1;s) = \xi_-(\alpha;s) \Phi_-(\mathcal{F} f_1;s).
\label{eqn:f-pm}
\end{equation}遯ｶ�｢
Since $f_j$ $(j=0,1)$ belongs to $\mathcal{V}_{\lambda,\ell}^\infty$,  
it follows from Lemma~\ref{thm:FEofZetaIntegral1} that the zeta integrals 
$Z(\alpha,\mathcal{F} f_j;s)$ has an analytic continuation to a meromorphic function of $s$ in $\C$. 
Since $f_j(\infty)=\mathcal{F}f_j(0)=0$, $Z(\alpha,\mathcal{F} f_j;s)$ has poles of order at most $1$ only at $s=1,2-2\lambda$. 
For any $s_0 \in \C$, we can choose $\varphi_j$ such that $\Phi_+(\mathcal{F} f_0;s)$ and $\Phi_-(\mathcal{F}f_1;s)$ are entire functions of $s$ which do not vanish at $s=s_0$. Hence $\xi_\pm(\alpha;s)$ can be continued meromorphically to the whole complex plane $\C$ with poles of order at most $1$ only at $s=1,2-2\lambda$.  
Let us calculate the residues of $\xi_\pm(\alpha;s)$ at $s=1, 2-2\lambda$. 
By Lemma~\ref{thm:FEofZetaIntegral1}, we have
\[
\mathop{\mathrm{Res}}_{s=1} Z(\alpha,\mathcal{F}f_j;s) 
 = \beta(\infty)(f_j)_\infty(\infty)=i^\ell \beta(\infty)f_j(0).
\]
On the other hand, since the support of $\mathcal{F}f_j(x)$ is contained in the set of real numbers with a fixed sign,  we have 
\[
\Phi_+(\mathcal{F} f_0;1)=f_0(0), \quad \Phi_-(\mathcal{F} f_1;1)=f_1(0). 
\]
Hence, by taking the residue at $s=1$ in $(\ref{eqn:f-pm})$, we get 
\begin{gather*}
\mathop{\mathrm{Res}}_{s=1} Z(\alpha,\mathcal{F}f_0;s) 
 = \Phi_+(\mathcal{F} f_0;1)\cdot \mathop{\mathrm{Res}}_{s=1} \xi_+(\alpha;s)
        = f_0(0)\cdot \mathop{\mathrm{Res}}_{s=1} \xi_+(\alpha;s), \\
\mathop{\mathrm{Res}}_{s=1} Z(\alpha,\mathcal{F}f_1;s) 
 = \Phi_+(\mathcal{F} f_1;1)\cdot \mathop{\mathrm{Res}}_{s=1} \xi_-(\alpha;s)
        = f_1(0)\cdot \mathop{\mathrm{Res}}_{s=1} \xi_-(\alpha;s).
\end{gather*}
Thus we see that 
\[
\mathop{\mathrm{Res}}_{s=1} \xi_+(\alpha;s) 
 = \mathop{\mathrm{Res}}_{s=1} \xi_-(\alpha;s) = i^\ell \beta(\infty).
\]
We now consider the residues at $s=2-2\lambda$. 
By Lemma~\ref{thm:FEofZetaIntegral1} we have
\[
\mathop{\mathrm{Res}}_{s=2-2\lambda} Z(\alpha,\mathcal{F}f_j;s) 
 = \beta(0)\mathcal{F}\left((f_j)_\infty\right)(0) 
     - \delta_{\lambda,1}\alpha(0)(\mathcal{F}f_j)(0), 
\]
where $\delta_{\lambda,1}=1$ or $0$ according as $\lambda=1$ or $\lambda\ne1$. 
Since $(\mathcal{F}f_j)(0)=0$,  
the identity implies that 
\begin{eqnarray*}
\mathop{\mathrm{Res}}_{s=2-2\lambda} Z(\alpha,\mathcal{F}f_j;s) 
 &=& \beta(0)\mathcal{F}\left((f_j)_\infty\right)(0) \\
 &=& \beta(0) \left(\Phi_+((f_j)_\infty;1) 
     +  \Phi_-((f_j)_\infty;1) \right) \\
 &=& \beta(0) \left(i^\ell \Phi_+(f_j;2\lambda-1) +  \Phi_-(f_j;2\lambda-1) \right). 
\end{eqnarray*}
On the other hand, by  $(\ref{eqn:f-pm})$, we have 
\begin{gather*}
\mathop{\mathrm{Res}}_{s=2-2\lambda} Z(\alpha,\mathcal{F}f_0;s) 
 = \Phi_+(\mathcal{F} f_0;2-2\lambda)\cdot \mathop{\mathrm{Res}}_{s=2-2\lambda} \xi_+(\alpha;s), \\
\mathop{\mathrm{Res}}_{s=2-2\lambda} Z(\alpha,\mathcal{F}f_1;s) 
 = \Phi_-(\mathcal{F} f_1;2-2\lambda)\cdot \mathop{\mathrm{Res}}_{s=2-2\lambda} \xi_-(\alpha;s).
\end{gather*}
From the local functional equation (Theorem \ref{thm:LFE}), it follows that
\[
\begin{pmatrix}
\Phi_+(f_j;2\lambda-1) \\
\Phi_-(f_j;2\lambda-1)
\end{pmatrix}
 = i (2\pi)^{1-2\lambda}\Gamma(2\lambda-1) 
     \begin{pmatrix}
       e^{-\lambda \pi  i} &  -e^{\lambda \pi  i} \\
       -e^{\lambda \pi  i} &  e^{-\lambda \pi  i} 
     \end{pmatrix}
\begin{pmatrix}
\Phi_+(\mathcal{F} f_j;2-2\lambda) \\
\Phi_-(\mathcal{F} f_j;2-2\lambda) 
\end{pmatrix}. 
\]
Since $\Phi_-(\mathcal{F} f_0;2-2\lambda) = \Phi_+(\mathcal{F} f_1;2-2\lambda)=0$, we have
\begin{eqnarray*}
\lefteqn{
i^\ell \Phi_+(f_0;2\lambda-1) +  \Phi_-(f_0;2\lambda-1)
} \\
 & &=  i \Phi_+(\mathcal{F} f_0;2-2\lambda)\cdot 
     (i^\ell e^{-\lambda \pi  i} - e^{\lambda \pi  i})
                (2\pi)^{-(2\lambda-1)}\Gamma(2\lambda-1), \\
\lefteqn{
i^\ell \Phi_+(f_1;2\lambda-1) +  \Phi_-(f_1;2\lambda-1)
} \\
 & &=  i \Phi_-(\mathcal{F} f_1;2-2\lambda)\cdot 
     (-i^\ell e^{\lambda \pi  i} + e^{-\lambda \pi  i})
                (2\pi)^{-(2\lambda-1)}\Gamma(2\lambda-1). 
\end{eqnarray*}
Hence we obtain  
\[
\mathop{\mathrm{Res}}_{s=2-2\lambda} \xi_\pm(\alpha;s) 
 = \pm i \beta(0) \cdot (i^\ell e^{\mp \lambda \pi  i} - e^{\pm \lambda \pi  i})
                  (2\pi)^{-(2\lambda-1)} \Gamma(2\lambda-1). 
\]
The assertion for $\xi_\pm(\beta;s)$ can be proved similarly.

(2)
 The functional equation of the zeta integrals in 
 Theorem~\ref{thm:FEofZetaIntegral1} can be written as
 \begin{align*}
  & \quad \left(\Phi_{+}(\mathcal{F}f;s), \Phi_{-}(\mathcal{F}f;s) \right)
 \left(
 \begin{array}{c}
 \xi_{+}(\alpha;s) \\[2pt]
 \xi_{-}(\alpha;s)
 \end{array}
 \right) \\
 &=N^{2-2\lambda-s}\cdot 
 \left(\Phi_{+}(\mathcal{F}f_{\infty};2-2\lambda-s), 
\Phi_{-}(\mathcal{F}f_{\infty};2-2\lambda-s) \right)
 \left(
 \begin{array}{c}
 \xi_{+}(\beta;2-2\lambda-s) \\[2pt]
 \xi_{-}(\beta;2-2\lambda-s)
 \end{array}
 \right).
 \end{align*}
 The local functional equation~\eqref{form:LFE} yields
 \begin{equation}
  \left(\Phi_{+}(\mathcal{F}f;s), \Phi_{-}(\mathcal{F}f;s) \right) =
 (2\pi)^{-s}\Gamma(s) 
 \left(\Phi_{+}(f;1-s), \Phi_{-}(f;1-s)\right)\gamma(s),
 \end{equation}
 and 
\begin{eqnarray*}
\lefteqn{
\left(\Phi_{+}(\mathcal{F}f_{\infty};2-2\lambda-s), 
 \Phi_{-}(\mathcal{F}f_{\infty};2-2\lambda-s) \right) 
} \\
 & & =
 (2\pi)^{s+2\lambda-2}\cdot \Gamma(2-2\lambda-s) 
 \left(\Phi_{+}(f_{\infty};s+2\lambda-1), 
 \Phi_{-}(f_{\infty};s+2\lambda-1)\right)\gamma(2-2\lambda-s).
\end{eqnarray*}
 Further, by \eqref{form:PhiFinfty}, 
 we have
 \begin{equation}
 \label{form:LocalZetaTransform}
 \left(\Phi_{+}(f_{\infty};s+2\lambda-1), 
 \Phi_{-}(f_{\infty};s+2\lambda-1)\right) = 
 \left(\Phi_{+}(f;1-s), \Phi_{-}(f;1-s)\right)\Sigma(\ell).
 \end{equation}
 Combining these formulas, we have
\begin{eqnarray*}
\lefteqn{
 \left(\Phi_{+}(f;1-s), \Phi_{-}(f;1-s)\right)
 \gamma(s)\left(
 \begin{array}{c}
 \Xi_{+}(\alpha;s) \\[2pt] \Xi_{-}(\alpha;s)
 \end{array}
 \right)
} \\
 & & = N^{2-2\lambda-s}\cdot 
 \left(\Phi_{+}(f;1-s), \Phi_{-}(f;1-s)\right)
 \Sigma(\ell)\cdot
 \gamma(2-2\lambda-s)\left(
 \begin{array}{c}
 \Xi_{+}(\beta;2-2\lambda-s) \\[2pt] \Xi_{-}(\beta;2-2\lambda-s)
 \end{array}
 \right).
\end{eqnarray*} 
Since there exist $f_0, f_1 \in C_0^{\infty}(\R)$ 
 such that 
 \[
 \left\{
 \begin{array}{l}
 \Phi_{+}(f_0;1-s)\neq 0, \\[2pt]
 \Phi_{-}(f_0;1-s)=0,
 \end{array} 
 \right. \qquad 
 \left\{
 \begin{array}{l}
 \Phi_{+}(f_1;1-s)= 0, \\[2pt]
 \Phi_{-}(f_1;1-s)\neq 0,
 \end{array} 
 \right.
 \]
 we obtain the functional equation
 \[
 \gamma(s)\left(
 \begin{array}{c}
 \Xi_{+}(\alpha;s) \\[2pt] \Xi_{-}(\alpha;s)
 \end{array}
 \right)\\
 = N^{2-2\lambda-s}\cdot 
 \Sigma(\ell)\cdot
 \gamma(2-2\lambda-s)\left(
 \begin{array}{c}
 \Xi_{+}(\beta;2-2\lambda-s) \\[2pt] 
\Xi_{-}(\beta;2-2\lambda-s)
 \end{array}
 \right).
 \]

(3) 
We consider the
 magnitude of $\xi_{\pm}(\alpha;s)$ and 
 $\xi_{\pm}(\beta;s)$.
 Take $f\in \mathcal{S}(\R)$ and fix a sufficiently large number 
 $\sigma_0>0$, and let $s=\sigma+i\tau$. 
We may assume that $\xi_\pm(\alpha;s)$, $\xi_\pm(\beta;s)$, $\Phi_\pm(\mathcal{F}f;s)$ and $Z(\alpha,f;s)$ are absolutely convergent for $\Re(s) \geq \sigma_0$. 
If $\sigma\leq \sigma_0$, then 
 \begin{align*}
 |Z_{+}(\alpha, \mathcal{F}f;s)| 
 &\leq \int_1^{\infty} t^{\sigma_0-1}
 \sum_{n\neq 0}
 |\alpha(n)| \cdot |\mathcal{F}f(nt)| dt =:C_{1, \sigma_0}<+\infty, 
 \end{align*}
 and 
 \begin{align*}
 |Z_{+}^{(N)}(\beta, \mathcal{F}f_{\infty};s)| 
 &\leq \int_1^{\infty} t^{\sigma_0-1}
 \sum_{n\neq 0}
 |\beta(n)| \cdot |\mathcal{F}f_{\infty}\left(\frac{nt}{N}\right)| dt 
 =:C_{2, \sigma_0}<+\infty.
 \end{align*}
 Hence it follows from \eqref{form:AnalyConiZetaIntegral} that 
 in the vertical strip 
 $2-2\Re(\lambda)-\sigma_0\leq \sigma\leq \sigma_0$, we have
 \begin{align*}
 & |s(s-1)(s+2\lambda-2)(s+2\lambda-1)Z(\alpha,\mathcal{F}f;s)| \\
 &\qquad \leq
 (C_{1, \sigma_0}+C_{2,\sigma_0})|s(s-1)(s+2\lambda-2)(s+2\lambda-1)| \\
 &\qquad \quad 
 + |\beta(0)(\mathcal{F}f_{\infty})(0)||s(s-1)(s+2\lambda-1)| + 
 |\beta(\infty)f_{\infty}(\infty)| |s(s+2\lambda-2)(s+2\lambda-1)| \\
 &\qquad \quad 
  + |\alpha(0)(\mathcal{F}f)(0)||(s-1)(s+2\lambda-2)(s+2\lambda-1)| + 
 |\alpha(\infty)f(\infty)| |s(s-1)(s+2\lambda-2)|.
 \end{align*}
 That is to say,  in the vertical strip 
 $2-2\Re(\lambda)-\sigma_0\leq \sigma\leq \sigma_0$, 
 an entire function
 $s(s-1)(s+2\lambda-2)(s+2\lambda-1)Z(\alpha,\mathcal{F}f;s)$
 satisfies 
 \[
 |s(s-1)(s+2\lambda-2)(s+2\lambda-1)Z(\alpha,\mathcal{F}f;s)|=
 O\left(|\tau|^4\right) \qquad (\text{as}\; \;
 |\tau|=|\Im(s)|\to \infty)
 \]
 uniformly with respect to $\sigma=\Re(s)$. 
 Now let $\varphi_1, \varphi_2$ be rapidly decreasing functions on 
 $\R$ with
 $\mathcal{F}\varphi_1(x) = e^{-\pi x^2}, 
 \mathcal{F}\varphi_2(x) = x e^{-\pi x^2}$. Then we have
 \[
 \Phi_{\pm}(\mathcal{F}\varphi_1;s) = \frac{1}{2}\pi^{-s/2} \Gamma
 \left(\frac{s}{2}\right), \qquad
  \Phi_{\pm}(\mathcal{F}\varphi_2;s) = \pm\frac{1}{2}\pi^{-(s+1)/2} \Gamma
 \left(\frac{s+1}{2}\right), 
 \]
 and thus 
 \begin{eqnarray*}
\lefteqn{
 s(s-1)(s+2\lambda-2)(s+2\lambda-1)\xi_{\pm}(\alpha;s)
}\\
 & & = s(s-1)(s+2\lambda-2)(s+2\lambda-1)\left(\pi^{s/2}\Gamma
 \left(\frac{s}{2}\right)^{-1} Z(\alpha, \mathcal{F}\varphi_1;s) 
 \vphantom{\Gamma \left(\frac{s+1}{2}\right)^{-1}}\right.\\ 
 & & \qquad {} \pm \left. 
 \pi^{(s+1)/2}\Gamma
 \left(\frac{s+1}{2}\right)^{-1} Z(\alpha, \mathcal{F}\varphi_2;s)\right).
 \end{eqnarray*}
Recall Stirling's estimate:
 \begin{equation}
  \label{form:StirlingEstimate}
 |\Gamma(\sigma+i\tau)|\sim\sqrt{2\pi} |\tau|^{\sigma-(1/2)} e^{-\pi|\tau|/2}
 \qquad (|\tau|\to \infty)
 \end{equation}
 uniformly in any vertical strip $\sigma_{1}\leq \sigma \leq \sigma_{2}$.
 Hence, we have
\begin{equation}
 |s(s-1)(s+2\lambda-2)(s+2\lambda-1)\xi_{\pm}(\alpha;s)|
 =O\left(|\tau|^{4+\Re(\lambda)+(\sigma_0/2)} e^{\pi |\tau|/4}
 \right) \qquad (|\tau|\to \infty)
\label{eqn:estimate on the strip}
\end{equation}
uniformly in the vertical strip  $2-2\Re(\lambda)-\sigma_0\leq \sigma\leq \sigma_0$.  
On the right half plane $\Re(s) \geq \sigma_0$, it follows from the absolute convergence of $\xi_\pm(\alpha;s)$ that $\xi_\pm(\alpha;s)$ is bounded. 
Hence 
\begin{equation}
(s-1)(s+2\lambda-2)\xi_{\pm}(\alpha;s)=O(|s|^2) \quad (\sigma_0 \leq \Re(s) \rightarrow \infty)
\label{eqn:estimate on the right half plane}
\end{equation}
with an $O$-constant idependent of $\Im(s)$. 
To  estimate $\Xi_\pm(\alpha;s)$ on the left half plane $\Re(s) \leq 2-2\Re(\lambda)-\sigma_0$, we use the functional equation \eqref{eqn:FE of xi} in the following form:
\begin{eqnarray*}
\begin{pmatrix}
\xi_+(\alpha;s) \\
\xi_-(\alpha;s) \\
\end{pmatrix} 
 &=& \frac{(2\pi)^s}{\Gamma(s)}\cdot \gamma(s)^{-1}\\
  & & \times 
   \left(\frac{N}{2\pi}\right)^{2-2\lambda-s}\Gamma(2-2\lambda-s)\Sigma(\ell)\gamma(2-2\lambda-s) 
\begin{pmatrix}
\xi_+(\beta;2-2\lambda-s) \\
\xi_-(\beta;2-2\lambda-s) \\
\end{pmatrix}. 
\end{eqnarray*}
Since 
\[
\frac{1}{\Gamma(s)}\gamma(s)^{-1} 
 = \frac{1}{\Gamma(s)\cdot 2i\sin(\pi s)}\gamma^*(s)
 = \frac{\Gamma(1-s)}{2\pi i}\gamma^*(s), 
\quad 
 \gamma^*(s) =
    \begin{pmatrix}
     e^{\pi i s/2} & -e^{-\pi i s/2} \\
     -e^{-\pi i s/2} & e^{\pi i s/2}
   \end{pmatrix},
\]
we have 
\begin{eqnarray*}
\begin{pmatrix}
\xi_+(\alpha;s) \\
\xi_-(\alpha;s) \\
\end{pmatrix} 
 &=& \frac{(2\pi)^{2s-2\lambda+2}\cdot N^{2-2\lambda-s}}{2\pi i}\cdot
           \Gamma(1-s) \Gamma(2-2\lambda-s) \\
  & & \times \gamma^*(s)  \Sigma(\ell)\gamma(2-2\lambda-s) 
\begin{pmatrix}
\xi_+(\beta;2-2\lambda-s) \\
\xi_-(\beta;2-2\lambda-s) \\
\end{pmatrix}. 
\end{eqnarray*}
The first factor $\frac{(2\pi)^{2s-2\lambda+2}\cdot N^{2-2\lambda-s}}{2\pi i} $ is majorized by 
$c_1 (4\pi^2/N)^{\Re(s)}$ with some positive constant $c_1$. 
Since 
\[
|\Gamma(s)| \leq \int_0^\infty x^{\Re(s)-1}e^{-x}\,dx = \Gamma(\Re(s)) 
 = O(e^{|\Re(s)|\log(|\Re(s)|)}) \quad (\epsilon \leq \Re(s) \rightarrow \infty)
\]
for any positive $\epsilon$, we have 
\[
\Gamma(1-s) \Gamma(2-2\lambda-s) 
 = O(e^{2|\Re(s)|\log(|\Re(s)|)}) \quad (2-2\Re(\lambda)-\sigma_0 \geq \Re(s) \rightarrow -\infty)
\]
uniformly with respect to $\Im(s)$. 
We see furthermore that every entry of  the matrix 
$\gamma^*(s)\Sigma(\ell)\gamma(2-2\lambda-s) $ is $O(e^{\pi|\Im(s)|})$. 
Finally $|\xi_\pm(\beta;2-2\lambda-s)|$ are bounded on $\Re(s) \leq 2-2\Re(\lambda)-\sigma_0$. 
Hence there exist positive constants $A,B$ such that 
\begin{equation}
|\xi_\pm(\alpha;s)| \leq A e^{B |s|\log(|s|)}
\quad (\Re(s) \leq 2-2\Re(\lambda)-\sigma_0).
\label{eqn:estimate on the left half plane} 
\end{equation}
By \eqref{eqn:estimate on the strip}, \eqref{eqn:estimate on the right half plane} and \eqref{eqn:estimate on the left half plane}, we see that $(s-1)(s+2\lambda-2)\xi_{\pm}(\alpha;s)$ is an entire function of order at most $1$. 
\end{proof}

\begin{corollary}
\label{cor:trivial zeros}
$(1)$ We have 
\[
\Res_{s=1} \xi_+(\alpha;s) = \Res_{s=1} \xi_-(\alpha;s), 
\quad \Res_{s=1} \xi_+(\beta;s) = \Res_{s=1} \xi_-(\beta;s).
\]

$(2)$  If $\lambda=\frac q2$ $(q \in \Z,\ q \geq 2)$, then we have
\[
\Res_{s=2-2\lambda} \xi_+(\alpha;s) + (-1)^q \Res_{s=2-2\lambda} \xi_-(\alpha;s)
 = \Res_{s=2-2\lambda} \xi_+(\beta;s) + (-1)^q \Res_{s=2-2\lambda} \xi_-(\beta;s)
 = 0.
\]
If $\lambda=\frac q2$ further satisfies that $2q \equiv \ell \pmod 4$, then we have
\[
\Res_{s=2-2\lambda} \xi_+(\alpha;s) = \Res_{s=2-2\lambda} \xi_-(\alpha;s)
 = \Res_{s=2-2\lambda} \xi_+(\beta;s) = \Res_{s=2-2\lambda} \xi_-(\beta;s)
 = 0.
\]

$(3)$ Unless $\lambda=\frac q2$ $(q \in \Z,\ q \geq 2)$ and $k = q-2$ or $q-1$, then we have 
\[
\xi_+(\alpha;-k)+(-1)^k\xi_-(\alpha;-k)=\xi_+(\beta;-k)+(-1)^k\xi_-(\beta;-k)=0
\quad (k \in \Z_{>0}). 
\]
\end{corollary}

\begin{proof}
The first and the second assertions are immediate consequences of the residue formulas in Theorem \ref{thm:FEofZeta1} (1). 
Let us prove the third assertion. 
By Lemma \ref{thm:FEofZetaIntegral1} and Theorem \ref{thm:FEofZeta1} (1), the zeta integral $Z(\alpha,\mathcal{F}f;s)$ and the $L$-functions are holomorphic at $s=-k$ $(k \in \Z_{>0})$ unless $\lambda=\frac q2$ $(q \in \Z,\ q \geq 2)$ and $k =  q-2, q-1$. Moreover, from the local functional equation \eqref{form:LFE} and Lemma \ref{lem:Int rpr of zeta},  we have 
\begin{eqnarray*}
Z(\alpha,\mathcal{F}f;s) 
 &=& (2\pi)^{-s} \Gamma(s) \left\{
   \Phi_+(f;1-s)  \left(
      e^{\pi i s/2}\xi_+(\alpha;s) 
      + e^{-\pi i s/2}\xi_-(\alpha;s) \right) \right. \\
 & & + \left.
   \Phi_-(f;1-s) \left(
      e^{-\pi i s/2}\xi_+(\alpha;s) 
      + e^{\pi i s/2}\xi_-(\alpha;s) \right)\right\}. 
\end{eqnarray*}
Take an $f \in C^\infty_0(\R)$ that is not identically zero, with non-negative values, and compactly supported in the set of positive real numbers.  
Then, $\Phi_-(f;1-s) = 0$, $\Phi_+(f;1-(-k)) \ne 0$, and  
\[
Z(\alpha,\mathcal{F}f;s) 
 = (2\pi)^{-s}  e^{\pi i s/2} \Gamma(s) \Phi_+(f;1-s)  
     \left(\xi_+(\alpha;s)  + e^{-\pi i s}\xi_-(\alpha;s) \right). 
\]
By the assumption on $\lambda$ and $k$, $Z(\alpha,\mathcal{F}f;s)$ is holomorphic at $s=-k$. 
On the other hand, by the choice of $f$, $(2\pi)^{-s}  e^{\pi i s/2} \Phi_+(f;1-s)$ does not vanish at $s=-k$, and $\Gamma(s)$ has a simple pole at $s=-k$. 
Hence $\xi_+(\alpha;s)  + e^{-\pi i s}\xi_-(\alpha;s)$ must have a zero at $s=-k$.    
This shows that
\[
\xi_+(\alpha;-k)+(-1)^k\xi_-(\alpha;-k)=0.
\]
The proof of the assertions for $\xi_\pm(\beta;s)$ is quite similar. 
\end{proof}

\subsection{A converse theorem for summation formulas}
\label{subsection:3.3}

Let us prove the converse
of Theorem~\ref{thm:FEofZeta1}.
Let $\lambda$ be a complex number with $\lambda \not\in \frac12-\frac12\Z_{\geq0}$. 
Let $\ell$ be an integer and $N$ a positive integer.

Suppose that two complex sequences 
$\alpha=\{\alpha(n)\}, \beta=\{\beta(n)\} \, (n  \in \Z\setminus\{0\})$ 
of polynomial growth are given. 
Then, 
we can define $L$-functions 
$\xi_{\pm}(\alpha;s)$ and $\xi_{\pm}(\beta;s)$ by \eqref{eqn:def of xi}
 and the zeta integrals $Z(\alpha,\mathcal{F}f;s)$ and $Z^{(N)}(\beta,\mathcal{F}f_\infty;s)$ $(f \in \mathcal{V}_{\lambda,\ell}^\infty)$ by \eqref{eqn:def of Z} and 
\eqref{eqn:def of ZN}. 
By Lemma \ref{lem:Int rpr of zeta}, if $\Re(s)$ is sufficiently large, the $L$-functions and the zeta integrals are convergent absolutely, and satisfy
 \begin{align}
\label{form:UnfoldZetaTilde1}
 Z(\alpha,\mathcal{F}f;s) &= \xi_{+}(\alpha;s)\Phi_{+}(\mathcal{F}f;s)+
\xi_{-}(\alpha;s)\Phi_{-}(\mathcal{F}f;s),\\
\label{form:UnfoldZetaTilde2}
Z^{(N)}(\beta, \mathcal{F}f_{\infty};s) &= N^s \cdot \xi_{+}(\beta;s)
\Phi_{+}(\mathcal{F}f_{\infty};s)+
N^s\cdot \xi_{-}(\beta;s)\Phi_{-}(\mathcal{F}f_{\infty};s).
\end{align} 
In the following, for simplicity, we put 
\begin{gather*}
\xi_e(\alpha;s) = \xi_+(\alpha;s) + \xi_-(\alpha;s), \quad
\xi_o(\alpha;s) = \xi_+(\alpha;s) - \xi_-(\alpha;s), \\
\xi_e(\beta;s) = \xi_+(\beta;s) + \xi_-(\beta;s), \quad
\xi_o(\beta;s) = \xi_+(\beta;s) - \xi_-(\beta;s).
\end{gather*}

Now we assume the following conditions [A1] -- [A4]:

\begin{description}
 \item[{\bf [A1]}] The $L$-functions $\xi_{\pm}(\alpha;s), \xi_{\pm}(\beta;s)$  have meromorphic continuations to the whole $s$-plane, and $(s-1)(s-2+2\lambda)\xi_\pm(\alpha;s)$ and $(s-1)(s-2+2\lambda)\xi_\pm(\alpha;s)$ are entire functions, which are of finite order in any vertical strip. 
\end{description}
Here a function $f(s)$ on a vertical strip $\sigma_1 \leq \Re(s) \leq \sigma_2$ $(\sigma_1,\sigma_2 \in \R, \sigma_1<\sigma_2)$ is said to be {\it of finite order\/} on the strip if there exist some positive constants $A,B,\rho$ such that 
\[
|f(s)| < A e^{B|\Im(s)|^\rho}, \quad \sigma_1 \leq \Re(s) \leq \sigma_2.
\]
\begin{description}
 \item[{\bf [A2]}]
The residues of  $\xi_{\pm}(\alpha;s)$ and $\xi_{\pm}(\beta;s)$ at $s=1$ satisfy
\[
\mathop{\mathrm{Res}}_{s=1} \xi_+(\alpha;s) = \mathop{\mathrm{Res}}_{s=1} \xi_-(\alpha;s), \quad 
\mathop{\mathrm{Res}}_{s=1} \xi_+(\beta;s) = \mathop{\mathrm{Res}}_{s=1} \xi_-(\beta;s). 
\] 
 \item[{\bf [A3]}] The following functional equation holds:
 \[
  \gamma(s)\left(
 \begin{array}{c}
 \Xi_{+}(\alpha;s) \\[2pt] \Xi_{-}(\alpha;s)
 \end{array}
 \right)\\
 = N^{2-2\lambda-s}\cdot 
 \Sigma(\ell)\cdot
 \gamma(2-2\lambda-s)\left(
 \begin{array}{c}
 \Xi_{+}(\beta;2-2\lambda-s) \\[2pt] \Xi_{-}(\beta;2-2\lambda-s)
 \end{array}
 \right),
 \]
 where $\Xi_\pm(\alpha;s)$,  $\Xi_\pm(\beta;s)$,  $\gamma(s)$ and $\Sigma(\ell)$ are  as defined  in Theorem \ref{thm:FEofZeta1}.
\item[{[A4]}]
\begin{itemize}
\item
If $\lambda = \frac q2$ $(q \in \Z_{\geq0},\ q \geq 4)$, then
\[
\xi_+(\alpha;-k)+(-1)^k\xi_-(\alpha;-k)=0 \quad (k = 1,2,\ldots,q-3).
\]
\item
If $\lambda = 1$, then
\[
\mathop{\mathrm{Res}}_{s=2-2\lambda} \xi_e(\alpha;s) 
 = \mathop{\mathrm{Res}}_{s=2-2\lambda} \xi_e(\beta;s) = 0.
\]
\item 
If $\ell \equiv 0 \pmod 4$ as well as $\lambda = 1$, then 
\begin{quote}
$\xi_\pm(\alpha;s), \xi_\pm(\beta;s)$ are holomorphic at $s=2-2\lambda=0$.
\end{quote}
\end{itemize}
\end{description}

As is proved in Theorem \ref{thm:FEofZeta1} and Corollary \ref{cor:trivial zeros}, 
if $\{\alpha(n)\}$ and $\{\beta(n)\}$ are the coefficients of a Ferrar-Suzuki summation formula of level $N$, then the $L$-functions $\xi_\pm(\alpha;s)$ and $\xi_\pm(\beta;s)$ satisfy the conditions [A1] -- [A4]. 
We prove that, conversely,  the conditions [A1] -- [A4] imply a summation formula. 

Under the assumptions [A1] -- [A4], we define $\alpha(0)$, $\beta(0)$, $\alpha(\infty)$, $\beta(\infty)$ by 
\begin{eqnarray}
 \label{form:DefOfA0AndAInfty}
 \alpha(0) &=&  
 \begin{cases}
-\xi_{e}(\alpha;0) & (\lambda \ne 1),  \\
-\frac12\xi_{e}(\alpha;0) + c & (\lambda=1,\ \ell \equiv0\pmod 4), \\
\frac{2\pi i}{1-i^{-\ell}}\cdot \Res_{s=0} \xi_+(\beta;s) 
  & (\lambda=1,\ \ell \not\equiv0\pmod 4),
\end{cases} \\
 \label{form:DefOfA0AndAInfty2}
\alpha(\infty) &=&
 \frac{N}{2} \Res_{s=1}\xi_{e}(\beta;s), \\[4pt]
 \label{form:DefOfB0AndBInfty}
 \beta(0) &=&
 \begin{cases}
 -\xi_{e}(\beta;0)   & (\lambda \ne 1),  \\
-\frac12\xi_{e}(\beta;0) + c & (\lambda=1,\ \ell \equiv0\pmod 4), \\
\frac{2\pi i}{i^{\ell}-1}\cdot \Res_{s=0} \xi_+(\alpha;s) 
  & (\lambda=1,\ \ell \not\equiv0\pmod 4),
\end{cases} \\
 \label{form:DefOfB0AndBInfty2}
 \beta(\infty) &=&
 \frac{i^{-\ell}}{2} \Res_{s=1}\xi_{e}(\alpha;s),
\end{eqnarray}
where $c$ appearing in $\alpha(0)$ and $\beta(0)$ in the case $\lambda=1$ and $\ell \equiv 0 \pmod 4$ is an arbitrary constant.

\begin{remark}
If $\lambda=1$ and $\ell \equiv 0 \pmod 4$, then we have $\mathcal{F}f_{\infty}(0)=\mathcal{F}f(0)$,  as we prove in Lemma \ref{lem:Ff0} (2) below. 
Hence, in the summation formula \eqref{form:SumFormulaByZeta}, we may replace $\alpha(0)$ and $\beta(0)$, respectively, by $\alpha(0)+c$ and $\beta(0)+c$ with an arbitrary constant $c$. This is the reason why an undetermined  constant $c$ appears in the definition of $\alpha(0)$ and $\beta(0)$. 
\end{remark}

\begin{theorem}
 \label{thm:ConverseTheoremForSummationFormula}
Assume that $\lambda \not\in \frac12- \frac12\Z_{\geq0}$.  
 If the $L$-functions $\xi_\pm(\alpha;s)$ and $\xi_\pm(\beta;s)$ satisfy the assumptions {\rm [A1]--[A4]}, then, for any $f\in {\mathcal{V}}_{\lambda,\ell}^{\infty}$, the following summation formula holds:
 \begin{equation}
 \label{form:SumFormulaByZeta}
 \alpha(\infty) f(\infty)+
 \sum_{n=-\infty}^{\infty} \alpha(n) (\mathcal{F}f)(n)
 = \beta(\infty) f_{\infty}(\infty) + 
\sum_{n=-\infty}^{\infty} \beta(n) (\mathcal{F}f_{\infty})
 \left(\frac{n}{N}\right).
 \end{equation}
\end{theorem}

First we prove the following 
\begin{lemma}
 \label{lem:AnaProperTildeZeta}
Under the same assumptions as in Theorem \ref{thm:ConverseTheoremForSummationFormula}, the following hold for any  $f\in {\mathcal{V}}_{\lambda,\ell}^{\infty}$. 

$(1)$ 
The zeta integrals $Z(\alpha, \mathcal{F}f;s)$ and $Z^{(N)}(
 \beta,\mathcal{F}f_{\infty};s)$ have analytic continuations to meromorphic functions of $s$ in $\C$ with poles only at $s=0, 1, 2-2\lambda, 1-2\lambda$.

$(2)$
The poles are of order at most $1$ and the residues are given as follows: 
 \begin{align*}
 \Res_{s=0}Z(\alpha, \mathcal{F}f;s)&= -\alpha(0) (\mathcal{F}f)(0), & 
 \Res_{s=1}Z(\alpha, \mathcal{F}f;s)&= \beta(\infty) f_{\infty}(\infty), \\ 
 \Res_{s=2-2\lambda}Z(\alpha, \mathcal{F}f;s)&= \beta(0) (\mathcal{F}f_{\infty})(0), &
 \Res_{s=1-2\lambda}Z(\alpha, \mathcal{F}f;s)&= -\alpha(\infty) f(\infty). &
 \end{align*}
In case $\lambda=1$, then $2-2\lambda=0$ and we understand that 
\[
\Res_{s=0}Z(\alpha, \mathcal{F}f;s)
 =  \Res_{s=2-2\lambda}Z(\alpha, \mathcal{F}f;s) 
 = \beta(0) (\mathcal{F}f_{\infty})(0)  - \alpha(0) (\mathcal{F}f)(0).
\]
$(3)$
They satisfy the same functional equation
 \begin{equation}
 \label{form:FEofTildeZetaInt}
 Z(\alpha, \mathcal{F}f;s)= Z^{(N)}(\beta, 
 \mathcal{F}f_{\infty};2-2\lambda-s)
 \end{equation}
 as the one in Lemma~\ref{thm:FEofZetaIntegral1}.
\end{lemma}

For the proof of Lemma \ref{lem:AnaProperTildeZeta}, we need the following lemma. 

\begin{lemma}
 \label{lem:Ff0}
Put
\[
\Phi_{e}(\mathcal{F}f;s) = \Phi_+(\mathcal{F}f;s)+\Phi_-(\mathcal{F}f;s), \quad
\Phi_{o}(\mathcal{F}f;s) = \Phi_+(\mathcal{F}f;s)-\Phi_-(\mathcal{F}f;s).
\]

$(1)$ If $\lambda \not\in -\frac12\Z_{\geq0}$, then $\Phi_{e}(\mathcal{F}f;s)$ is holomorphic at $s=1$ and the identity 
\[
\Phi_e(\mathcal{F}f;1)= f(0) 
\]
holds for any $f \in \mathcal{V}^\infty_{\lambda,\ell}$.

$(2)$ If $\lambda=1$, then $\Phi_{o}(\mathcal{F}f;s)$ is holomorphic at $s=0$ and the identities 
\begin{eqnarray*}
\mathcal{F}f(0) 
   &=& \frac{1+i^{-\ell}}2 \mathcal{F}f_\infty(0) + \frac{1-i^{-\ell}}{2\pi i} \Phi_o(\mathcal{F}f_\infty;0), \\
\mathcal{F}f_\infty(0) 
 &=& \frac{1+i^\ell}2 \mathcal{F}f(0) - \frac{1-i^\ell}{2\pi i} \Phi_o(\mathcal{F}f;0)
\end{eqnarray*}
hold for any $f \in \mathcal{V}^\infty_{\lambda,\ell}$. 
If $\ell \equiv 0 \pmod 4$ as well as $\lambda=1$, then $\mathcal{F}f(0)=\mathcal{F}f_\infty(0)$. 
\end{lemma}

\begin{proof}
(1) From the local functional equation $(\ref{form:LFE})$, it follows that
\[
\Phi_e(\mathcal{F}f;s)
 = 2(2\pi)^{-s}\Gamma(s)\cos\left(\frac{\pi s}2\right)
   (\Phi_+(f;1-s)+\Phi_-(f;1-s)).
\]
Hence we have
\begin{eqnarray*}
\Phi_e(\mathcal{F}f;1)
 &=& \frac{1}{\pi} \lim_{s\rightarrow 1}  
          \left\{\cos\left(\frac{\pi s}2\right)
                                    (\Phi_+(f;1-s)+\Phi_-(f;1-s))\right\} \\
 &=& \frac{1}{\pi} \cdot \left.\left(\frac{d}{ds}\cos\left(\frac{\pi s}2\right)\right)\right|_{s=1}
                 \cdot    (\Res_{s=1}\Phi_+(f;1-s)+\Res_{s=1}\Phi_-(f;1-s)) \\
 &=& \frac{1}{\pi} \cdot \left(-\frac{\pi}{2}\right)\cdot  (-2f(0)) = f(0).
\end{eqnarray*}

$(2)$ If  $\lambda = 1$, then we have 
\[
\mathcal{F}f(0)
 = \Phi_+(f;1)+\Phi_-(f;1) 
   = \Phi_+(f_\infty;1)+i^{-\ell}\Phi_-(f_\infty;1).
\]
From the local functional equation $(\ref{form:LFE})$, it follows that
\[
\begin{pmatrix}
\Phi_+(f;s) \\ \Phi_-(f;s) 
\end{pmatrix}
 = (2\pi)^{-s}\Gamma(s) 
    \begin{pmatrix} 1 & -i \\ 1 & i \end{pmatrix}
    \begin{pmatrix}
      \cos\left(\frac{\pi s}2\right) \Phi_e(\mathcal{F}f;1-s) \\
      \sin\left(\frac{\pi s}2\right) \Phi_o(\mathcal{F}f;1-s) 
\end{pmatrix}.
\]
Recall that $\Phi_e(\mathcal{F}f;s)$ has a pole of order at most $1$ at $s=0$ with residue $2\mathcal{F}f(0)$, and $\Phi_o(\mathcal{F}f;s)$ is holomorphic at $s=0$. 
Hence we get
\[
\begin{pmatrix}
\Phi_+(f_\infty;1) \\ \Phi_-(f_\infty;1) 
\end{pmatrix}
 =   \begin{pmatrix} 1 & -i \\ 1 & i \end{pmatrix}
    \begin{pmatrix}
      \frac12 \mathcal{F}f_\infty(0) \\
      \frac1{2\pi} \Phi_o(\mathcal{F}f_\infty;0) 
\end{pmatrix}
\]
and 
\begin{eqnarray*}
\mathcal{F}f(0) 
   &=& \left(\frac12 \mathcal{F}f_\infty(0) + \frac1{2\pi i} \Phi_o(\mathcal{F}f_\infty;0) \right)
     + i^{-\ell} \left(\frac12 \mathcal{F}f_\infty(0) - \frac1{2\pi i} \Phi_o(\mathcal{F}f_\infty;0) \right) \\
   &=& \frac{1+i^{-\ell}}2 \mathcal{F}f_\infty(0) + \frac{1-i^{-\ell}}{2\pi i} \Phi_o(\mathcal{F}f_\infty;0).
\end{eqnarray*}
Since $(f_\infty)_\infty(x)=i^\ell f(x)$, we also have
\[
\mathcal{F}f_\infty(0) 
 = \frac{1+i^\ell}2 \mathcal{F}f(0) - \frac{1-i^\ell}{2\pi i} \Phi_o(\mathcal{F}f;0).
\]
If $\ell \equiv 0 \pmod 4$, then this implies that $\mathcal{F}f_\infty(0) 
 = \mathcal{F}f(0)$.
\end{proof}

\begin{proof}[Proof of Lemma \ref{lem:AnaProperTildeZeta}]
 As we have seen in \S \ref{subsect:dual LZ and LFE}, 
 $\Phi_{\pm}(\mathcal{F}f;s)$ and $\Phi_{\pm}(\mathcal{F}f_\infty;s)$ converge absolutely for sufficiently large  $\Re(s)$, and have analytic continuations to  meromorphic functions of $s$ in $\C$. 
Hence,  by the identities \eqref{form:UnfoldZetaTilde1}, \eqref{form:UnfoldZetaTilde2}, and the assumption~[A1], 
$Z(\alpha, \mathcal{F}f;s)$ and $Z^{(N)}(
 \beta,\mathcal{F}f_{\infty};s)$ have analytic continuations to meromorphic functions of $s$ in $\C$. 

Now the functional equation of the zeta integrals can be proved as follows: 
 \begin{align*}
 Z(\alpha,\mathcal{F}f;s) &
 \overset{\eqref{form:UnfoldZetaTilde1}}{=}
 \vec{\Phi}(\mathcal{F}f;s)
 \vec{\xi}(\alpha;s) 
 \overset{\eqref{form:LFE}}{=}
 \vec{\Phi}(f;1-s) \gamma(s)
 \vec{\Xi}(\alpha;s) \\
 &\overset{\text{[A3]}}{=}
 N^{2-2\lambda-s}\vec{\Phi}(f;1-s)  
 \Sigma(\ell) \cdot \gamma(2-2\lambda-s)
 \vec{\Xi}(\beta;2-2\lambda-s) \\
 &\overset{\eqref{form:LocalZetaTransform}}{=}
 N^{2-2\lambda-s}\vec{\Phi}(f_{\infty};s+2\lambda-1) \gamma(2-2\lambda-s)
 \vec{\Xi}(\beta;2-2\lambda-s) \\
 &\overset{\eqref{form:LFE}}{=}
 N^{2-2\lambda-s}\vec{\Phi}(\mathcal{F}f_{\infty};2-2\lambda-s)
 \vec{\Xi}(\beta;2-2\lambda-s) \\
 &\overset{\eqref{form:UnfoldZetaTilde2}}{=} 
 Z^{(N)}(\beta, \mathcal{F}f_{\infty};2-2\lambda-s).
 \end{align*}
Here we put 
\[
\vec\Phi=(\Phi_+, \Phi_-), \quad
\vec\xi=\begin{pmatrix} \xi_+ \\ \xi_-\end{pmatrix}, \quad
\vec\Xi=\begin{pmatrix} \Xi_+ \\ \Xi_-\end{pmatrix}.
\]

Let us determine the poles of the zeta integrals. 
The poles of $\xi_\pm(\alpha;s)$ and  $\xi_\pm(\beta;s)$ are located at $s=1,2-2\lambda$ and,  by Theorem~\ref{thm:LFE},  the poles of $\Phi_{\pm}(\mathcal{F}f;s)$ and $\Phi_{\pm}(\mathcal{F}f_\infty;s)$ are located at $s=-n, 1-2\lambda-n$ $(n  \in \Z_{\geq0})$. 
Hence the poles of $ Z(\alpha, \mathcal{F}f;s)$ are at $s=1-n, 2-2\lambda-n$ $(n \in \Z_{\geq0})$, and the poles of $ Z^{(N)}(\beta, \mathcal{F}f_\infty;2-2\lambda-s)$ are at $s=n, 1-2\lambda+n$ $(n \in \Z_{\geq0})$. 
By the functional equation  $ Z(\alpha, \mathcal{F}f;s) = Z^{(N)}(\beta, \mathcal{F}f_\infty;2-2\lambda-s)$, the actual poles are included in 
\[
\{1-n, 2-2\lambda-n\; | \; n \in \Z_{\geq 0}\} 
\cap \{n, 1-2\lambda+n\; | \; n \in \Z_{\geq 0}\}.
\]
The assumption that $\lambda \not\in \frac12-\frac12\Z_{\geq0}$ implies that the intersection is 
\[
\begin{cases}
\{0,1,1-2\lambda, 2-2\lambda\} \cup \{-(q-3),\ldots,-2,-1\} & (\lambda=\frac q2 \in \frac12\Z,\ q \geq 4), \\
\{0,1,1-2\lambda, 2-2\lambda\} & (\text{otherwise}).
\end{cases}
\]
In the case where $\lambda=\frac q2 \in \frac12\Z$ with $q \geq 4$ and $1\leq k \leq q-3$, then $\xi_\pm(\alpha;s)$ are holomorphic at $s=-k$ and $\Phi_\pm(\mathcal{F}f;s)$ has a pole  of order at most $1$ at $s=-k$ with residue $\frac{(\pm1)^k(\mathcal{F}f)^{(k)}(0)}{k!}$. Hence, by [A4], we see that
\begin{eqnarray*}
\Res_{s=-k} Z(\alpha,\mathcal{F}f;s) 
 &=&  \xi_+(\alpha;-k) \Res_{s=-k} \Phi_+(\mathcal{F}f;s)
    + \xi_-(\alpha;-k) \Res_{s=-k} \Phi_-(\mathcal{F}f;s) \\
 &=& ( \xi_+(\alpha;-k) + (-1)^k\xi_-(\alpha;-k)) \frac{(\mathcal{F}f)^{(k)}(0)}{k!}
 = 0.
\end{eqnarray*}
This shows that $s=-k$ $(1 \leq k \leq q-3)$ is not an actual pole of $Z(\alpha,\mathcal{F}f;s)$. 
Therefore the poles of   $Z(\alpha,\mathcal{F}f;s)$ and also of  $Z^{(N)}(\beta,\mathcal{F}f_\infty;s)$ are located at $s=0,1,1-2\lambda,2-2\lambda$. 

By Theorem \ref{thm:LFE} and the assumption [A2], the orders of the poles at $s=0,1$ of $\xi_\pm(\alpha;s)$ and  $\Phi_\pm(\mathcal{F}f;s)$ are given by the following table.
\[
\begin{array}{c|cc}
 & s=0 & s=1 \\
\hline
\xi_\pm(\alpha;s) & \begin{cases} 0 & (\lambda\ne1) \\  {} \leq 1  & (\lambda=1) \end{cases}
 &  {} \leq 1  \\
\Phi_\pm(\mathcal{F}f;s) & {} \leq 1 & 0
\end{array}
\]
Therefore $s=1$ is a pole of  $Z(\alpha,\mathcal{F}f;s)$ of order at most 1. 
By Lemma \ref{lem:Ff0} and the assumption [A2], the residue there is calculated as follows:
\begin{eqnarray*}
\Res_{s=1} Z(\alpha,\mathcal{F}f;s) 
 &=& \Phi_+(\mathcal{F}f;1)\cdot \Res_{s=1} \xi_+(\alpha;s)   
        + \Phi_-(\mathcal{F}f;1)\cdot \Res_{s=1} \xi_-(\alpha;s)   \\
 &=& \Res_{s=1} \xi_+(\alpha;s)   
        (\Phi_+(\mathcal{F}f;1)+\Phi_-(\mathcal{F}f;1)) \\
 &=& f(0) \cdot \Res_{s=1} \xi_+(\alpha;s).   
\end{eqnarray*}
Similarly we can prove that $Z^{(N)}(\beta,\mathcal{F}f_\infty;s)$ has a pole of order at most $1$ at $s=1$ with residue $N\cdot f_\infty(0) \cdot \Res_{s=1} \xi_+(\beta;s)$.  
From the functional equation $Z(\alpha,\mathcal{F}f;s) = Z^{(N)}(\beta,\mathcal{F}f_\infty;2-2\lambda-s)$, it follows that $Z(\alpha,\mathcal{F}f;s)$ has a pole of order at most $1$ at $s=1-2\lambda$ and the residue is given by
\begin{eqnarray*}
\Res_{s=1-2\lambda} Z(\alpha,\mathcal{F}f;s) 
 &=& \Res_{s=1-2\lambda}  Z^{(N)}(\beta,\mathcal{F}f_\infty;2-2\lambda-s) \\
 &=& -\Res_{s=1}  Z^{(N)}(\beta,\mathcal{F}f_\infty;s) \\
 &=& -N\cdot f_\infty(0) \cdot \Res_{s=1} \xi_+(\beta;s).
\end{eqnarray*}

If $\lambda\ne1$, then the table above shows that $s=0$ is a pole of  $Z(\alpha,\mathcal{F}f;s)$ of order at most 1, and the residue there is given by 
\begin{eqnarray*}
\Res_{s=0} Z(\alpha,\mathcal{F}f;s) 
 &=& \xi_+(\alpha;0) \cdot \Res_{s=0} \Phi_+(\mathcal{F}f;s)
        + \xi_-(\alpha;0) \cdot \Res_{s=0} \Phi_-(\mathcal{F}f;s)   \\
 &=& \xi_e(\alpha;0) \cdot \mathcal{F}f(0).   
\end{eqnarray*}
Similarly, if $\lambda\ne1$, we can prove that  $Z^{(N)}(\beta,\mathcal{F}f_\infty;s)$ has a pole of order at most $1$ at $s=0$ with residue $\xi_e(\beta;0) \cdot \mathcal{F}f_\infty(0)$.  
Again, by the functional equation, we see that $Z(\alpha,\mathcal{F}f;s)$ has a pole of order at most $1$ at $s=2-2\lambda$ and the residue is given by 
\[
\Res_{s=2-2\lambda} Z(\alpha,\mathcal{F}f;s) 
 = -\xi_e(\beta;0) \cdot \mathcal{F}f_\infty(0). 
\]

Finally we examine the pole at $s=0$ in the case $\lambda=1$. 
If $\lambda=1$, then $\xi_\pm(\alpha;s)$ may have a pole at $s=0=2-2\lambda$, and  
the coefficient of $s^{-2}$ in the Laurent expansion of $Z(\alpha,\mathcal{F}f;s)$ is given by 
\[
\Res_{s=0} \xi_+(\alpha;s) \cdot \Res_{s=0} \Phi_+(\mathcal{F}f;s) 
 + \Res_{s=0} \xi_-(\alpha;s) \cdot \Res_{s=0} \Phi_-(\mathcal{F}f;s) 
 = \Res_{s=0} \xi_e(\alpha;s)  \cdot     \mathcal{F}f(0),  
\]
which vanishes by the assumption [A4]. 
Hence $Z(\alpha,\mathcal{F}f;s)$ has a pole of order at most $1$ at $s=0$ even in the case $\lambda=1$. 
Let us calculate the residues. 
Then, the assumption [A4] implies that $\xi_e(\alpha;s)$ is holomorphic at $s=0$ and $\xi_o(\alpha;s)$ has a pole of order at most $1$ at $s=0$ whose residue is equal to $2\Res_{s=0} \xi_+(\alpha;s)$. 
Since
\[
\Res_{s=0}\Phi_+(\mathcal{F}f;s)
 = \Res_{s=0}\Phi_-(\mathcal{F}f;s) = \mathcal{F}f(0), 
\]
$\Phi_o(\mathcal{F}f;s)$ is holomorphic at $s=0$, and 
$\Phi_e(\mathcal{F}f;s)$ has a pole of order at most $1$ at $s=0$ with residue 
$2\mathcal{F}f(0)$. 
Hence we have 
\begin{eqnarray*}
\Res_{s=0} Z(\alpha,\mathcal{F}f;s) 
 &=& \frac12 \Res_{s=0} (\xi_e(\alpha;s)\Phi_e(\mathcal{F}f;s) 
        + \xi_o(\alpha;s)\Phi_o(\mathcal{F}f;s)) \\
 &=& \xi_e(\alpha;0)\cdot \mathcal{F}f(0)
        + \Res_{s=0} \xi_+(\alpha;s) \cdot \Phi_o(\mathcal{F}f;0). 
\end{eqnarray*}
By Lemma \ref{lem:Ff0} and the assumption [A4], we obtain
\begin{eqnarray*}
\lefteqn{
\Res_{s=0} Z(\alpha,\mathcal{F}f;s) 
} \\
 & & = 
\begin{cases}
\left(\xi_e(\alpha;0) + \frac{\pi i  (1+i^\ell)}{1-i^\ell} 
      \Res_{s=0} \xi_+(\alpha;s)\right) \mathcal{F}f(0) & \\
\qquad {} - \frac{2 \pi i }{1-i^\ell} 
 \Res_{s=0} \xi_+(\alpha;s) \mathcal{F}f_\infty(0)  & (\ell \not\equiv 0 \pmod 4), \\
\xi_e(\alpha;0) \mathcal{F}f(0) 
 & (\ell \equiv 0 \pmod 4).
\end{cases}
\end{eqnarray*}
For $\lambda=1$, the functional equation in [A3] can be rewritten as 
\begin{equation}
\label{eqn:FE for lambda=1}
 \pi^{-s}
 \left(
 \begin{array}{c}
 \frac{\Gamma\left(\frac{s}2\right)}{\Gamma\left(\frac{1-s}2\right)}
    \xi_{e}(\alpha;s) \\[2ex] 
 i\frac{\Gamma\left(\frac{s+1}2\right)}{\Gamma\left(\frac{2-s}2\right)}\xi_{o}(\alpha;s)
 \end{array}\right)
 = N^{-s} \pi^{s}
    \begin{pmatrix}
    \frac{1+i^\ell}2 & \frac{1- i^\ell}2 \\ \frac{-1+i^\ell}2 & \frac{-1-i^\ell}2
\end{pmatrix}  
\left(
 \begin{array}{c}
 \frac{\Gamma\left(\frac{-s}2\right)}{\Gamma\left(\frac{1+s}2\right)}
 \xi_{e}(\beta;-s) \\[2ex] 
 i \frac{\Gamma\left(\frac{1-s}2\right)}{\Gamma\left(\frac{2+s}2\right)}
\xi_{o}(\beta;-s)
 \end{array}
 \right).
\end{equation}
Hence we have
\[
 i{\pi}\Res_{s=0} \xi_{o}(\beta;s)
 = -(1- i^{-\ell}) \xi_{e}(\alpha;0) +\frac{(1+i^{-\ell}) \pi i}2 \Res_{s=0} \xi_{o}(\alpha;s). 
\]
From this, it follows that  
\[
\Res_{s=0} Z(\alpha,\mathcal{F}f;s) 
= 
\begin{cases}
-\frac{2\pi i}{1-i^{-\ell}} 
      \Res_{s=0} \xi_+(\beta;s)\cdot \mathcal{F}f(0) & \\
\qquad {} - \frac{2 \pi i }{1-i^\ell} 
 \Res_{s=0} \xi_+(\alpha;s) \cdot \mathcal{F}f_\infty(0)  & (\ell \not\equiv 0 \pmod 4), \\
\xi_e(\alpha;0) \mathcal{F}f(0) 
 & (\ell \equiv 0 \pmod 4).
\end{cases}
\]
In case $\ell \equiv 0 \pmod 4$, 
by the functional equation \eqref{eqn:FE for lambda=1}, 
we have 
\begin{equation}
\label{eqn:exceptional identity alpha=-beta}
\xi_e(\alpha;0)=-\xi_e(\beta;0), 
\end{equation}
and by  Lemma \ref{lem:Ff0} (2), 
we have $\mathcal{F}f(0)=\mathcal{F}f_\infty(0)$. 
Therefore 
\[
\Res_{s=0} Z(\alpha,\mathcal{F}f;s) 
 = \left(-\frac12\xi_e(\beta;0) + c \right)\mathcal{F}f_\infty(0)
   - \left(-\frac12\xi_e(\alpha;0) + c \right)\mathcal{F}f(0) 
\]
for $\lambda=1$ and $\ell\equiv0\pmod4$. 
This completes the proof of  Lemma~\ref{lem:AnaProperTildeZeta}.
\end{proof}

Now Theorem \ref{thm:ConverseTheoremForSummationFormula} can be proved with the standard argument based on the Mellin inversion formula. 

\begin{proof}
[Proof of Theorem~\ref{thm:ConverseTheoremForSummationFormula}]
 Take a sufficiently large real number $\sigma_0>\max\{1,2-2\Re(\lambda)\}$ so that $\xi_\pm(\alpha;s)$ and  $\xi_\pm(\beta;s)$ are absolutely convergent if $\Re(s) \geq \sigma_0$. 
We consider the contour integral 
 \[
  \frac{1}{2\pi i}\int_{\Re(s)= \sigma_0}
 Z(\alpha, \mathcal{F}f;s)t^{-s}ds.
 \]
 By \eqref{form:UnfoldZetaTilde1}, we have 
 \[
  |Z(\alpha,\mathcal{F}f;s)| \leq
  |\xi_{+}(\alpha;s)|\cdot |\Phi_{+}(\mathcal{F}f;s)|+
 |\xi_{-}(\alpha;s)|\cdot |\Phi_{-}(\mathcal{F}f;s)|. 
 \]
Since the inequality 
\[
 |\xi_{\pm}(\alpha; s)|
 \leq \xi_{\pm}(|\alpha|; \sigma_0) := \sum_{n} |\alpha(\pm n)|n^{-\sigma_0}
\]
holds for any $s$ with $\Re(s) =\sigma_0$,  
 it follows from Lemma~\ref{lem:EstimatesOfGamma} that,  for arbitrary $\mu>0$, 
 \[
 |Z(\alpha, \mathcal{F}f;s)|= O(|\tau|^{-\mu})
 \qquad (\Re(s)=\sigma_0,\ |\tau |=|\Im(s)|\rightarrow\infty).
 \]
 Hence, by the Mellin inversion formula, we have
 \begin{equation}
 \label{form:MellinInversionForTildeZa}
 \sum
 \begin{Sb}
  n=-\infty \\
  n\neq 0
 \end{Sb}^{\infty}
 \alpha(n)(\mathcal{F}f)(nt) =
 \frac{1}{2\pi i}\int_{\Re(s)= \sigma_0}
 Z(\alpha, \mathcal{F}f;s)t^{-s}ds.
 \end{equation}
 On the other hand, the functional equation~\eqref{form:FEofTildeZetaInt} 
 implies that on the vertical line 
  $\Re(s) = 2-2\, \Re(\lambda)-\sigma_0$, 
 for arbitrary $\mu>0$, we have 
 \[
  |Z(\alpha, \mathcal{F}f;s)| =
 |Z^{(N)}(\beta, \mathcal{F}f_{\infty};
 2-2\lambda-s)|=O(|\tau |^{-\mu}) \qquad
 (|\tau |\rightarrow\infty).
 \]
 Furthermore, by the assumption~[A1] and 
 Lemma~\ref{lem:EstimatesOfGamma}, 
 in any vertical strip $\sigma_1\leq \Re(s)\leq \sigma_2$, 
 we have an estimate
 \[
  |Z(\alpha, \mathcal{F}f;s)| = O(e^{|\tau|^{\rho}})
 \qquad (|\tau|\rightarrow\infty)
 \]
 uniformly with respect to $\Re(s)$. 
 Now the Phragmen-Lindel\"{o}f principle
 (see e.g., \cite[Lemma~4.3.4]{Miyake}) implies that 
 in the vertical strip
 \[
  F:= \{s\in \C\, |\, 2-2\, \Re(\lambda)-\sigma_0\leq \Re(s)\leq \sigma_0\},
 \]
 for arbitrary $\mu>0$,  the  estimate
 \[
  |Z(\alpha, \mathcal{F}f;s)|=O(|\tau|^{-\mu}) \qquad
 (|\tau|\rightarrow\infty)
 \]
 holds uniformly with respect to $\Re(s)$. 
 Note that the poles 
 $s=0, 1, 2-2\lambda, 1-2\lambda$ of 
 $Z(\alpha, \mathcal{F}f;s)$ are contained in
 $F$. 
 Hence the standard argument of 
 ``moving the line of integration'' shows 
 \begin{align*}
  \sum
 \begin{Sb}
  n=-\infty \\
  n\neq 0
 \end{Sb}^{\infty}
 \alpha(n)(\mathcal{F}f)(nt)
 &=\frac{1}{2\pi i}\int_{\Re(s) = 2-2\, \Re(\lambda)-\sigma_0}
  Z(\alpha,\mathcal{F}f;s)t^{-s}ds \\[-2pt]
  &\quad +
  \Res_{s=0}Z(\alpha,\mathcal{F}f;s)\cdot t^{0}+
   \Res_{s=1}Z(\alpha,\mathcal{F}f;s)\cdot t^{-1}\\[2pt]
 &\quad +\Res_{s=2-2\lambda}Z(\alpha,\mathcal{F}f;s)\cdot t^{2\lambda-2}
+\Res_{s=1-2\lambda}Z(\alpha,\mathcal{F}f;s)\cdot t^{2\lambda-1}.
 \end{align*}
 Here, by the third assertion of Lemma~\ref{lem:AnaProperTildeZeta}, 
 we can rewrite the formula above
 as 
 \begin{align*}
  \sum
 \begin{Sb}
  n=-\infty \\
  n\neq 0
 \end{Sb}^{\infty}
 \alpha(n)(\mathcal{F}f)(nt)
 &=\frac{1}{2\pi i}\int_{\Re(s) = 2-2\, \Re(\lambda)-\sigma_0}
  Z(\alpha,\mathcal{F}f;s)t^{-s}ds 
  -\alpha(0)(\mathcal{F}f)(0)t^0 \\
 &\qquad +\beta(\infty)f_{\infty}(\infty)t^{-1}
+\beta(0)(\mathcal{F}f_{\infty})(0)t^{2\lambda-2}-\alpha(\infty)f(\infty)t^{2\lambda-1}.
 \end{align*}
 Furthermore,  by using the functional equation~\eqref{form:FEofTildeZetaInt},
 we can transform the first term of the right hand side as
 \begin{align*}
  \frac{1}{2\pi i}\int_{\Re(s) = 2-2\,\Re(\lambda)-\sigma_0}
 Z(\alpha,\mathcal{F}f;s)t^{-s}ds
  &= \frac{1}{2\pi i}
  \int_{\Re(s) = 2-2\,\Re(\lambda)-\sigma_0}
  Z^{(N)}(\beta, \mathcal{F}f_{\infty}; 2-2\lambda-s)
  t^{-s}ds \\[3pt]
  &= \frac{1}{2\pi i}
  \int_{\Re(s) = \sigma_0}
  Z^{(N)}(\beta,\mathcal{F}f_{\infty};s)t^{s+2\lambda-2}
  ds  \\[5pt]
  &=\frac{t^{2\lambda-2}}{2\pi i}
  \int_{\Re(s) =\sigma_0} Z^{(N)}(\beta,
  \mathcal{F}f_{\infty};s) (t^{-1})^{-s}ds\\
  &= t^{2\lambda-2}
  \sum
  \begin{Sb}
   n=-\infty \\
   n\neq 0
  \end{Sb}^{\infty}
  \beta(n)(\mathcal{F}f_{\infty})\left( 
 \frac{nt^{-1}}{N}\right),
 \end{align*}
 and we therefore conclude that
 \begin{align*}
  \sum
 \begin{Sb}
  n=-\infty \\
  n\neq 0
 \end{Sb}^{\infty}
 \alpha(n)(\mathcal{F}f)(nt)
  &= t^{2\lambda-2}
  \sum
  \begin{Sb}
   n=-\infty \\
   n\neq 0
  \end{Sb}^{\infty}
  \beta(n)(\mathcal{F}f_{\infty})\left( 
 \frac{nt^{-1}}{N}\right)
  -\alpha(0)(\mathcal{F}f)(0)t^0 \\
 &\qquad+\beta(\infty)f_{\infty}(\infty)t^{-1}
+\beta(0)(\mathcal{F}f_{\infty})(0)t^{2\lambda-2}-\alpha(\infty)f(\infty)t^{2\lambda-1}.
 \end{align*}
 Finally, letting $t=1$ and transposing two terms, we obtain 
 the summation formula 
 \[
 \alpha(\infty)f(\infty)+\sum_{n=-\infty}^{\infty} \alpha(n)(\mathcal{F}f)(n)=
 \beta(\infty) f_{\infty}(\infty)+
 \sum_{n=-\infty}^{\infty} \beta(n)(\mathcal{F}f_{\infty})\left(\frac{n}{N}\right),
 \]
as required. 
\end{proof}

\section{Converse theorems for automorphic distributions and Maass forms}

\subsection{Statement of the converse theorems}
\label{subsect:4.1}

Let $\lambda$ be a complex number with $\lambda \not\in \frac12-\frac12\Z_{\geq0}$. 
Fix an integer $\ell$ and a positive integer  $N$. 
We assume that $N$ is a multiple of $4$ when $\ell$ is odd.
Let 
 $\alpha=\{\alpha(n)\}_{n\in \Z\setminus\{0\}}$ and $\beta=\{\beta(n)\}_{n\in \Z\setminus\{0\}}$ be  complex sequences of polynomial growth, 
For $\alpha,\beta$, we can define the  $L$-functions 
 $\xi_{\pm}(\alpha;s), \xi_{\pm}(\beta;s)$
  and the completed $L$-functions
 $\Xi_{\pm}(\alpha;s), \Xi_{\pm}(\beta;s)$ by
 \begin{align*}
 \xi_{\pm}(\alpha;s) &= 
 \sum_{n=1}^{\infty}\frac{\alpha(\pm n)}{n^s}, & 
 \Xi_{\pm}(\alpha;s)&=
 (2\pi)^{-s} \Gamma(s) \xi_{\pm}(\alpha;s), \\
  \xi_{\pm}(\beta;s) &= 
 \sum_{n=1}^{\infty}\frac{\beta(\pm n)}{n^s}, & 
 \Xi_{\pm}(\beta;s)&=
 (2\pi)^{-s} \Gamma(s) \xi_{\pm}(\beta;s).  
 \end{align*}
As was proved in \S 3, 
the properties [A1] -- [A4] of  $\xi_{\pm}(\alpha;s), \xi_{\pm}(\beta;s)$ are equivalent to the summation formula 
\begin{eqnarray*}
\lefteqn{
 \alpha(\infty) f(\infty)+
 \sum_{n=-\infty}^{\infty} \alpha(n) (\mathcal{F}f)(n)
} \\
 & & = \beta(\infty) f_{\infty}(\infty) + 
\sum_{n=-\infty}^{\infty} \beta(n) (\mathcal{F}f_{\infty})\left(\frac nN \right) 
\quad (f \in \mathcal{V}^\infty_{\lambda,\ell}),
\end{eqnarray*}
where $\alpha(0), \alpha(\infty), \beta(0), \beta(\infty)$ are defined by 
$(\ref{form:DefOfA0AndAInfty})$, $(\ref{form:DefOfA0AndAInfty2})$, 
$(\ref{form:DefOfB0AndBInfty})$, $(\ref{form:DefOfB0AndBInfty2})$, 
respectively. 
Let $T$ be the automorphic distribution for $\widetilde\Delta(N)$ associated with the summation formula above (see Lemma \ref{lemma:distribution interpretation of sum formula}). 
In this section, we prove a converse theorem that describes a condition for $T$ to be automorphic for $\widetilde{\Gamma}_0(N)$ in terms of the $L$-functions twisted by Dirichlet characters. 
A Weil type converse theorem for Maass forms for $\widetilde{\Gamma}_0(N)$ will be derived immediately from the converse theorem for automorphic distributions for $\widetilde{\Gamma}_0(N)$ (see Theorem \ref{thm:MaassFormFromAutomDist}).

For an odd prime number $r$ with $(N,r)=1$ and a Dirichlet character $\psi$ mod $r$, the twisted $L$-functions 
  $\xi_{\pm}(\alpha, \psi;s), \Xi_{\pm}(\alpha, \psi;s), 
    \xi_{\pm}(\beta, \psi;s), \Xi_{\pm}(\beta, \psi;s)$ 
are defined by
 \begin{align*}
 \xi_{\pm}(\alpha, \psi;s) &= 
 \sum_{n=1}^{\infty}\frac{\alpha(\pm n)\tau_{\psi}(\pm n)}{n^s}, & 
 \Xi_{\pm}(\alpha,\psi; s)&=
 (2\pi)^{-s} \Gamma(s) \xi_{\pm}(\alpha,\psi;s), \\
  \xi_{\pm}(\beta,\psi;s) &= 
 \sum_{n=1}^{\infty}\frac{\beta(\pm n)\tau_{\psi}(\pm n)}{n^s}, & 
 \Xi_{\pm}(\beta, \psi;s)&=
 (2\pi)^{-s} \Gamma(s) \xi_{\pm}(\beta,\psi;s), 
 \end{align*}
where $\tau_{\psi}(n)$ is the Gauss sum defined by 
\begin{equation}
  \tau_{\psi}(n)= \sum
 \begin{Sb}
 m \bmod r \\
 (m, r)=1 
 \end{Sb} \psi(m) e^{2\pi i mn/r}. 
\label{eqn:def of Gauss sum}
\end{equation}
We also define  $\xi_{e}(\alpha, \psi;s)$ and  $\xi_{e}(\beta, \psi;s)$ by
\[
\xi_{e}(\alpha, \psi;s) = \xi_{+}(\alpha, \psi;s) + \xi_{-}(\alpha, \psi;s), \quad  
\xi_{e}(\beta, \psi;s) = \xi_{+}(\beta, \psi;s) + \xi_{-}(\beta, \psi;s). 
\]
We put 
\[
\tau_\psi=\tau_\psi(1)=\sum
 \begin{Sb}
 m \bmod r \\
 (m, r)=1 
 \end{Sb} \psi(m) e^{2\pi i m/r}
\]
and denote by $\psi_{r,0}$ the principal character modulo $r$.  
Recall that the Gauss sums are calculated as follows:
\begin{eqnarray}
\tau_\psi(n) 
 &=& \begin{cases}
   \overline{\psi(n)} \tau_\psi  & (n \not\equiv 0 \pmod{r}), \\
   0 &  (n \equiv 0 \pmod{r}),
      \end{cases}
      \quad \text{if $\psi\ne\psi_{r,0}$},  
\label{eqn:Gauss sum non-trivial psi}\\ 
\tau_{\psi_{r,0}}(n) 
  &=& \begin{cases}
  -1 &  (n \not\equiv 0 \pmod{r}), \\
   r-1 & (n \equiv 0 \pmod{r}).
  \end{cases}
\label{eqn:Gauss sum trivial psi}
\end{eqnarray}
We have therefore
\begin{equation}
\xi_\pm(\alpha,\psi;s) 
 = \begin{cases}
  \displaystyle \tau_\psi \sum_{\overset{\scriptstyle n=1}{(n,r)=1}}^\infty \frac{\alpha(\pm n)\overline{\psi(\pm n)}}{n^s} 
   & (\psi\ne\psi_{r,0}), \\[2ex]
  \displaystyle   r \sum_{n=1}^\infty \frac{\alpha(\pm rn)}{(rn)^s} - \xi_\pm(\alpha;s)
   & (\psi=\psi_{r,0}).
\end{cases}
\label{eqn:explicit twist}
\end{equation}
A similar identity holds also for $\xi_\pm(\beta,\psi;s)$. 
 
 Let $\chi$ be a Dirichlet character mod ${N}$
 that satisfies $\chi(-1)=i^{\ell}$ (resp.\ $\chi(-1)=1$)
 when $\ell$ is even (resp.\ odd).  
Let $\mathbb{P}_N$ be a set of odd prime numbers not dividing $N$ such that, for any positive integers $a, b$ coprime to each other, $\mathbb{P}_N$ contains a prime number $r$ of the form $r=am+b$ for some $m \in \Z_{>0}$.   
For an $r \in \mathbb{P}_N$, denote by $X_r$ the set of all Dirichlet characters mod $r$ (including the principal character $\psi_{r,0}$). 
For $\psi \in X_r$, we define the Dirichlet character  $\psi^*$  by
\begin{equation}
\psi^*(k)=\overline{\psi(k)}\left(\frac{k}{r}\right)^{\ell}. 
\label{eqn:psi-star}
\end{equation}
We put 
\begin{equation}
 C_{\ell,r} =
  \begin{cases}
  1 & (\text{$\ell$ is even}), \\
  \varepsilon_{r}^{\ell}  &   (\text{$\ell$ is odd}).
 \end{cases}
\label{eqn:def of clr}
\end{equation}
(For the definition of $\varepsilon_r$, see \eqref{eqn:def of epsilond}.)

For an $r \in \mathbb{P}_N$ and a $\psi \in X_r$, we consider the following conditions $\mathrm{[A1]}_{r,\psi}$ -- $\mathrm{[A5]}_{r,\psi}$ on  $\xi_{\pm}(\alpha,\psi;s)$ and $\xi_{\pm}(\beta,\psi^*;s)$.  
 
\begin{description}
\item[$\text{\bf [A1]}_{r,\psi}$] 
 $\xi_{\pm}(\alpha,\psi;s), \xi_{\pm}(\beta,\psi^*;s)$  have meromorphic 
 continuations to the whole $s$-plane, and $(s-1)(s-2+2\lambda)\xi_{\pm}(\alpha,\psi;s), (s-1)(s-2+2\lambda)\xi_{\pm}(\beta,\psi^*;s)$ are entire functions, which are of finite order in any vertical strip. 
\item[$\text{\bf [A2]}_{r,\psi}$] 
The residues of  $\xi_{\pm}(\alpha,\psi;s)$ and $\xi_{\pm}(\beta,\psi^*;s)$ satisfy
\[
\mathop{\mathrm{Res}}_{s=1} \xi_+(\alpha,\psi;s) = \mathop{\mathrm{Res}}_{s=1} \xi_-(\alpha,\psi;s), \quad 
\mathop{\mathrm{Res}}_{s=1} \xi_+(\beta,\psi^*;s) = \mathop{\mathrm{Res}}_{s=1} \xi_-(\beta,\psi^*;s). 
\]
\item[$\text{\bf [A3]}_{r,\psi}$] 
$\Xi_{\pm}(\alpha,\psi;s)$ and $\Xi_{\pm}(\beta,\psi^*;s)$
satisfy the following functional equation:
\begin{align*}
 \gamma(s)\left(
 \begin{array}{c}
 \Xi_{+}(\alpha,\psi;s) \\ \Xi_{-}(\alpha,\psi;s)
 \end{array}
 \right) &= \overline{\chi(r)}\cdot C_{\ell,r} \cdot  \psi^*(-N) \cdot 
 r^{2\lambda-2}\cdot (Nr^2)^{2-2\lambda-s} \\
 &\qquad \cdot \Sigma(\ell)\cdot  \gamma(2-2\lambda-s)
 \begin{pmatrix}
 \Xi_{+}\left(\beta, {\psi}^*;2-2\lambda-s\right) \\[4pt]
 \Xi_{-}\left(\beta, {\psi}^*;2-2\lambda-s\right)
 \end{pmatrix}, 
 \end{align*}
 where $\gamma(s)$ and $\Sigma(\ell)$ are the same as in Theorem \ref{thm:FEofZeta1},  namely, 
 \[
  \gamma(s) = 
  \begin{pmatrix}
  e^{\pi s i/2}  & e^{-\pi si/2} \\
  e^{-\pi si/2}  & e^{\pi si/2} 
 \end{pmatrix}, \qquad
 \Sigma(\ell) =
 \begin{pmatrix}
 0 & i^{\ell} \\
 1& 0
 \end{pmatrix}. 
 \]
\item[$\text{\bf [A4]}_{r,\psi}$] 
\begin{itemize}
\item
If $\lambda = \frac q2$ $(q \in \Z_{\geq0},\ q \geq 4)$, then
\[
\xi_+(\alpha,\psi;-k)+(-1)^k\xi_-(\alpha,\psi;-k)=0 \quad (k = 1,2,\ldots,q-3).
\]
\item
If $\lambda = 1$, then
\[
\mathop{\mathrm{Res}}_{s=2-2\lambda} \xi_e(\alpha,\psi;s) 
 = \mathop{\mathrm{Res}}_{s=2-2\lambda} \xi_e(\beta,\psi^*;s) = 0.
\]
\item 
If $\ell \equiv 0 \pmod 4$ as well as $\lambda = 1$, then 
\begin{quote}
$\xi_\pm(\alpha,\psi;s), \xi_\pm(\beta,\psi^*;s)$ are holomorphic at $s=2-2\lambda=0$.
\end{quote}
\end{itemize}

\item[$\text{\bf [A5]}_{r,\psi}$] 
The following relations between residues and special values hold: 
\end{description}
$$
\leqno{(1)} \quad 
\begin{cases}
 \xi_e(\alpha,\psi;0) 
  = \tau_\psi(0)
          \xi_e(\alpha;0) & (\lambda\ne1), \\
 \xi_e(\alpha,\psi;0) 
  = \tau_\psi(0)
     \left(\frac{\overline{\chi(r)}+1}2 \xi_e(\alpha;0) + c(\overline{\chi(r)}-1)\right)
                       & (\lambda=1,\ \ell\equiv0\pmod4), \\
 \overline{\chi(r)}\cdot \psi^*(-N)\cdot C_{\ell,r}
 {\displaystyle \Res_{s=0}}\,\xi_{\pm}(\beta, \psi^{*};s) 
  = \tau_{\psi}(0){\displaystyle \Res_{s=0}}\,\xi_{\pm}(\beta;s) 
  & (\lambda=1,\ \ell \not\equiv 0 \pmod{4}).
\end{cases}
$$
$$
\leqno{(2)} \quad
 \overline{\chi(r)}\cdot \psi^*(-N)\cdot C_{\ell,r}\cdot r^{2\lambda}
 {\displaystyle \Res_{s=1}}\,\xi_{\pm}(\beta, \psi^{*};s) 
     =\tau_{\psi}(0){\displaystyle \Res_{s=1}}\,\xi_{\pm}(\beta;s).
$$
$$
\leqno{(3)} \quad 
\begin{cases}
 \xi_e(\beta, \psi^{*};0)
   = \tau_{\psi^*}(0) \xi_e(\beta;0) & (\lambda\ne1), \\
  \overline{\chi(r)}\xi_e(\beta,\psi^*;0) 
  = \tau_{\psi^*}(0)
     \left(\frac{\overline{\chi(r)}+1}2 \xi_e(\beta;0) - c(\overline{\chi(r)}-1)\right)
                       & (\lambda=1,\ \ell\equiv0\pmod4), \\
 {\displaystyle \Res_{s=0}}\,\xi_{\pm}(\alpha, \psi;s) 
  = \overline{\chi(r)} C_{\ell,r} \psi^*(-N) \tau_{\psi^*}(0){\displaystyle \Res_{s=0}}\,\xi_{\pm}(\alpha;s) 
  & (\lambda=1,\ \ell \not\equiv 0 \pmod{4}).
\end{cases}
$$
$$
\leqno{(4)} \quad
 {\displaystyle \Res_{s=1}}\,\xi_{\pm}(\alpha,\psi;s) 
  =  \overline{\chi(r)}C_{\ell,r}\cdot r^{-2\lambda}
 \cdot \tau_{\psi^*}(0)\cdot {\displaystyle \Res_{s=1}}\,\xi_{\pm}(\alpha;s).
$$
Here $c$ appearing in (1) and (3) is the constant in the definition of $\alpha(0)$ and $\beta(0)$ in \eqref{form:DefOfA0AndAInfty} and \eqref{form:DefOfB0AndBInfty}. 
If $\chi$ is the principal character, then the constant $c$ disappears from the conditions and can be completely arbitrary. If $\chi$ is non-principal, then the conditions determine $c$ uniquely.

\begin{remark}
By \eqref{eqn:Gauss sum non-trivial psi}, the conditions in $\mathrm{[A5]}_{r,\psi}$  become simpler for non-principal characters.
If $\psi$ is different from the principal character, then the first and the second conditions can be reformulated as
$$
\leqno{(1)} \quad 
\begin{cases}
 \xi_e(\alpha,\psi;0)=0 & (\lambda\ne1\ \text{or}\ \lambda=1, \ell\equiv0\pmod4), \\
 \text{$\xi_{\pm}(\beta, \psi^*;s)$ is holomorphic at $s=0$}
  & (\lambda=1,\ \ell \not\equiv 0 \pmod{4}),
\end{cases}
$$
$$
\leqno{(2)} \quad
 \text{$\xi_{\pm}(\beta, \psi^*;s)$ is holomorphic at $s=1$}.
$$
If $\psi^*$ is different from the principal character, then the third and the fourth conditions can be reformulated as
$$
\leqno{(3)} \quad 
\begin{cases}
 \xi_e(\beta, \psi^*;0) = 0 & (\lambda\ne1\ \text{or}\ \lambda=1, \ell\equiv0\pmod4), \\
 \text{$\xi_{\pm}(\alpha,\psi;s)$ is holomorphic at $s=0$} 
  & (\lambda=1,\ \ell \not\equiv 0 \pmod{4}),
\end{cases}
$$
$$
\leqno{(4)} \quad
 \text{$\xi_{\pm}(\alpha,\psi;s)$ is holomorphic at $s=1$.} 
$$
\end{remark}

Now we can formulate our main theorems, the converse theorems for automorphic distributions and Maass forms for $\widetilde{\Gamma}_0(N)$.

\begin{theorem}
\label{thm:ConverseTheoremForCongSubgp}
Let  $\lambda\not\in \frac12-\frac12\Z_{\geq0}$. 
We assume that $\xi_\pm(\alpha;s)$ and $\xi_\pm(\beta;s)$ satisfy the  conditions
 {\rm [A1] -- [A4]} in \S \ref{subsection:3.3}, and define $\alpha(0), \alpha(\infty), \beta(0), \beta(\infty)$ by 
$(\ref{form:DefOfA0AndAInfty})$, $(\ref{form:DefOfA0AndAInfty2})$, 
$(\ref{form:DefOfB0AndBInfty})$, $(\ref{form:DefOfB0AndBInfty2})$, 
respectively.   
We assume furthermore that, for any $r \in \mathbb{P}_N$ and $\psi \in X_r$, 
 $\xi_{\pm}(\alpha,\psi;s)$ and $\xi_{\pm}(\beta,\psi^*;s)$  satisfy the conditions $\mathrm{[A1]}_{r,\psi}$ -- $\mathrm{[A5]}_{r,\psi}$.
Let 
\[
\alpha(\infty)f(\infty)+\sum_{n=-\infty}^{\infty} \alpha(n)(\mathcal{F}f)(n)
 = \beta(\infty) f_\infty(\infty)+
 \sum_{n=-\infty}^{\infty} \beta(n)(\mathcal{F}f_\infty)\left(\frac{n}{N}\right)
\quad (f \in \mathcal{V}^\infty_{\lambda,\ell})
\]
be the summation formula given by the assumptions {\rm [A1] -- [A4]} and Theorem  \ref{thm:ConverseTheoremForSummationFormula}, and $T$ be the automorphic distribution for $\widetilde\Delta(N)$ associated with this summation formula. 
 Then $T$ is an automorphic distribution for $\widetilde{\Gamma}_0(N)$
  with character $\chi$.   
\end{theorem}

As we saw in Theorem~\ref{thm:MaassFormFromAutomDist}, the image of an
automorphic distribution under the Poisson transform $\mathcal{P}_{\lambda,\ell}$
is a Maass form, and hence Theorem~\ref{thm:ConverseTheoremForCongSubgp} implies the following converse theorem for Maass forms.

\begin{theorem}
 \label{corollary:Maassforms}
Let  $\lambda\not\in \frac12-\frac12\Z_{\geq0}$. 
We assume that $\xi_\pm(\alpha;s)$ and $\xi_\pm(\beta;s)$ satisfy the  conditions
 {\rm [A1] -- [A4]} in \S \ref{subsection:3.3}, and define $\alpha(0), \alpha(\infty), \beta(0), \beta(\infty)$ by 
$(\ref{form:DefOfA0AndAInfty})$, $(\ref{form:DefOfA0AndAInfty2})$, 
$(\ref{form:DefOfB0AndBInfty})$, $(\ref{form:DefOfB0AndBInfty2})$, 
respectively.   
We assume furthermore that, for any $r \in \mathbb{P}_N$ and $\psi \in X_r$, 
 $\xi_{\pm}(\alpha,\psi;s)$ and $\xi_{\pm}(\beta,\psi^*;s)$  satisfy the conditions $\mathrm{[A1]}_{r,\psi}$ -- $\mathrm{[A5]}_{r,\psi}$.
Define the functions $F_\alpha(z)$ and  $G_\beta(z)$ on the upper half plane $\mathcal{H}$ by 
\begin{align}
 \nonumber
 F_\alpha(z) &= \alpha(\infty) \cdot y^{\lambda-\ell/4}+
 \alpha(0)\cdot i^{-\ell/2} \cdot 
 \frac{(2\pi) 2^{1-2\lambda} \Gamma(2\lambda-1)}
 {\Gamma\left(\lambda+\frac{\ell}{4}\right)
 \Gamma\left(\lambda-\frac{\ell}{4}\right)} \cdot y^{1-\lambda-\ell/4} \\[3pt]
 \label{form:Results}
  &\quad + \sum
 \begin{Sb}
 n=-\infty \\
 n\neq 0
 \end{Sb}^{\infty} \alpha(n) \cdot 
 \frac{i^{-\ell/2}\cdot \pi^{\lambda} \cdot  |n|^{\lambda-1}}
{\Gamma\left(\lambda+\frac{\sgn(n)\ell}{4}\right)} \cdot 
 y^{-\ell/4}\, W_{\frac{\sgn(n)\ell}{4}, \lambda-\frac{1}{2}}\left(4\pi|n|y\right)
\cdot \mathbf{e}[nx],  \\
 \nonumber
G_\beta(z) &=  N^\lambda \beta(\infty) \cdot y^{\lambda-\ell/4} 
  + N^{1-\lambda}  \beta(0) \cdot i^{-\ell/2} \cdot
 \frac{(2\pi) 2^{1-2\lambda} \Gamma(2\lambda-1)}
 {\Gamma\left(\lambda+\frac{\ell}{4}\right)
 \Gamma\left(\lambda-\frac{\ell}{4}\right)} \cdot y^{1-\lambda-\ell/4} \\[3pt]
 \label{form:Results2}
  &\quad + N^{1-\lambda}  \sum
 \begin{Sb}
 n=-\infty \\
 n\neq 0
 \end{Sb}^{\infty} \beta(n) \cdot 
 \frac{i^{-\ell/2} \cdot \pi^{\lambda} \cdot  |n|^{\lambda-1}}
{\Gamma\left(\lambda+\frac{\sgn(n)\ell}{4}\right)} \cdot 
 y^{-\ell/4}\, W_{\frac{\sgn(n)\ell}{4}, \lambda-\frac{1}{2}}\left(4\pi|n|y\right)
\cdot \mathbf{e}[nx]. 
\end{align}
Then $F_\alpha(z)$ (resp.\ $G_\beta(z)$) gives a Maass form 
 for $\widetilde{\Gamma}_0(N)$ of weight $\frac{\ell}{2}$ with
character $\chi^{-1}$ (resp.\ $\chi_{N,\ell}^{-1}$), and eigenvalue $(\lambda-\ell/4)(1-\lambda-\ell/4)$, where
\begin{equation}
\chi_{N,\ell}(d)=\overline{\chi(d)}\left(\frac{N}{d}\right)^\ell.
\label{eqn:chi N ell}
\end{equation}
Moreover, 
we have 
\[
\left(\left.F_\alpha\right|_\ell \tilde w_N\right)(z) = G_\beta(z) 
\]
(for the definition of $\tilde w_N$, see \eqref{form:DefOfTildeNandBarN}). 
\end{theorem}

\begin{remark}
Recall that $\alpha(0)$ and $\beta(0)$ involve an undetermined constant $c$ in the case where $\lambda=1$, $\ell \equiv 0 \pmod 4$ and $\chi$ is principal. 
Write $\ell = 4k$ $(k \in \Z)$. If $k \ne 0$, then $\lim_{\lambda\rightarrow1} 1/(\Gamma(\lambda+k)\Gamma(\lambda-k)) = 0$. 
Hence the coefficient of $y^{1-\lambda-\ell/4}$ vanishes, and the Maass forms $F_\alpha(z)$ and $G_\beta(z)$ do not depend on the constant $c$. 
If $k=0$, namely $\ell=0$, then  $y^{1-\lambda-\ell/4}=1$ and the constant function is a Maass form for  $\widetilde\Gamma_0(N)$. 
The constant $c$ in $\alpha(0)$ and $\beta(0)$ contributes to $F_\alpha(z)$ and $G_\beta(z)$ as the constant function $2\pi c$. 
\end{remark}

\begin{remark}
\label{remark:NO}
In the original case of  Weil's converse theorem for holomorphic modular forms, 
analytic properties of $L$-functions are assumed only for twists by primitive characters. 
A converse theorem for Maass forms of weight $\ell=0$ under this weaker assumptions is obtained  in Neururer and Oliver \cite{NO}, which appeared very recently in the final stage of preparation of this manuscript.  
The method of Neururer and Oliver (the two circle method) may be adapted to our situation (at least for $\lambda \not\in 1-\frac12\Z_{\geq0}$) to obtain  
Theorems \ref{corollary:Maassforms} and \ref{thm:ConverseTheoremForCongSubgp} 
under the following weaker assumptions: 
\begin{quote}
 ``for any $r \in \mathbb{P}_N$ and {\it any primitive\/} $\psi \in X_r$, 
 $\xi_{\pm}(\alpha,\psi;s)$ and $\xi_{\pm}(\beta,\psi^*;s)$  satisfy the conditions $\mathrm{[A1]}_{r,\psi}$ -- $\mathrm{[A5]}_{r,\psi}$.''
\end{quote}
If $\ell$ is even, then $\psi^*=\overline{\psi}$ and the principal character $\psi_{r,0}$ can be completely avoided from the weaker assumptions. However, if $\ell$ is odd and $\psi$ is the Legendre character, then $\psi^*$ is the principal character $\psi_{r,0}$ and imprimitive.  Hence even the weaker assumptions, which involve both $\psi$ and $\psi^*$, are not closed within primitive characters.  
This is the case also for the converse theorem of holomorphic modular forms of half-integral weight given in Shimura \cite[Section 5]{Shimura73}. 
\end{remark}

To see that Theorem~\ref{thm:ConverseTheoremForCongSubgp} actually implies the assertion for $F_\alpha(z)$ in Theorem~\ref{corollary:Maassforms}, it is enough to calculate  the Poisson transform $\mathcal{P}_{\lambda, \ell}T$ of $T$ in Theorem~\ref{thm:ConverseTheoremForCongSubgp} explicitly.  
By definition, the Poisson transform of $T$ is given by 
\[
 (\mathcal{P}_{\lambda, \ell}T)(z) = T(p_{\lambda, \ell, z}), \qquad 
 p_{\lambda, \ell, z}(t)=\frac{y^{\lambda-\ell/4}}{|z-t|^{2\lambda-\ell/2}\cdot (z-t)^{\ell/2}},
\]
and further, since the Poisson transform is continuous
(cf.\ \cite{Oshima, OS}), we have
\[
(\mathcal{P}_{\lambda, \ell}T)(z) 
 = \alpha(\infty)p_{\lambda, \ell, z}(\infty)+ \sum^{\infty}_{n=-\infty}
\alpha(n)(\mathcal{F}p_{\lambda, \ell, z})(n) 
\]
The explicit formula for $(\mathcal{F}p_{\lambda, \ell,z})(n)$ is given  by \eqref{eqn:Fourier transform of Poisson kernel} in Example \ref{example:Fourier transform of Poisson kernel}, 
Moreover, by \eqref{form:PoissonKerIntertwines}, we have 
\begin{align*}
 p_{\lambda, \ell, z}(\infty)&=  (p_{\lambda, \ell, z})_{\infty}(0)
  = (\pi_{\lambda, \ell}(\tilde{w}^{-1})p_{\lambda, \ell,z})(0) 
  =  (q_{\lambda, \ell, 0}\, \big|_{\ell} \, \tilde{w}^{-1})(z), \\
 q_{\lambda,\ell,0}(z)&= \Im(z)^{\lambda-(\ell/4)} \, \big|_{\ell} \, \tilde{w}, 
\end{align*} 
and hence 
\[
 (p_{\lambda, \ell,z})(\infty) = \Im(z)^{\lambda-(\ell/4)} = y^{\lambda-(\ell/4)}.
\]
Now it is easy to check that $(\mathcal{P}_{\lambda, \ell}T)(z)$ coincides with $F_\alpha(z)$, and Theorem~\ref{corollary:Maassforms} for $F_\alpha(z)$ follows immediately from Theorem~\ref{thm:MaassFormFromAutomDist} and Theorem~\ref{thm:ConverseTheoremForCongSubgp}. 

The assertion for $G_\beta(z)$ in Theorem~\ref{corollary:Maassforms} is a consequence of the following lemma together with Theorem~\ref{thm:ConverseTheoremForCongSubgp}. 

\begin{lemma}
\label{lem: T and T vee}
Let $T$ be a distribution on $\mathcal{V}_{\lambda,\ell}^\infty$ and put $T^\vee=\pi^*_{\lambda,\ell}(\tilde w_N^{-1})T$. 

$(1)$ 
If $T$ is the automorphic distribution for $\widetilde\Delta(N)$ associated with the summation formula 
\begin{eqnarray*}
\lefteqn{
\alpha(\infty) f(\infty)+
 \sum_{n=-\infty}^{\infty} \alpha(n) (\mathcal{F}f)(n)
} \\
 & & = \beta(\infty) f_{\infty}(\infty) + 
\sum_{n=-\infty}^{\infty} \beta(n) (\mathcal{F}f_{\infty}) \left(\frac nN\right)
\quad (f \in \mathcal{V}^\infty_{\lambda,\ell}),
\end{eqnarray*}
then $T^\vee$ is the automorphic distribution for $\widetilde\Delta(N)$ associated with the summation formula 
\begin{eqnarray}
\lefteqn{
  N^\lambda \beta(\infty)f(\infty) 
   +  N^{1-\lambda} \sum_{n=-\infty}^\infty  \beta(n)(\mathcal{F}f)(n) 
} \nonumber \\
 & & = i^{-\ell} N^{-\lambda} \alpha(\infty) f_\infty(\infty) 
   + i^{-\ell} N^{\lambda-1} \sum_{n=-\infty}^{\infty} \alpha(n) (\mathcal{F}f_\infty)\left(\frac nN\right).
\label{eqn:sum formula for Tvee}
\end{eqnarray}

$(2)$ 
The distribution $T$ is an automorphic distribution for $\widetilde{\Gamma}_0(N)$ with character $\chi$ if and only if  $T^\vee$ is an automorphic distribution for $\widetilde{\Gamma}_0(N)$ with character $\chi_{N,\ell}$. 
\end{lemma}

\begin{proof}
$(1)$ 
Since $\tilde w_N = \tilde w \tilde d(\sqrt N)$ by \eqref{eqn: wN equals d w}, we have 
\[
T^\vee(f) 
 = T\left(\pi_{\lambda,\ell}(\tilde w) \pi_{\lambda,\ell}(\tilde d(\sqrt N))f\right)
 = T_\infty\left(\pi_{\lambda,\ell}(\tilde d(\sqrt N))f\right).
\]
For simplicity, put 
\[
f^*(x)=(\pi_{\lambda,\ell}(\tilde d(\sqrt N))f)(x)=N^{-\lambda}f\left(\frac xN\right).
\]
Then 
\[
\mathcal{F}f^*\left(\frac nN\right)=N^{1-\lambda}\mathcal{F}f(n), \quad 
f^*(\infty)=N^\lambda f_\infty(0) =N^\lambda f(\infty).
\]
Hence 
\[
T^\vee(f) =  T_\infty\left(f^*\right)
 =  N^\lambda \beta(\infty)f(\infty) 
   + N^{1-\lambda} \sum_{n=-\infty}^\infty  \beta(n)(\mathcal{F}f)(n).
\]
On the other hand, by \eqref{eqn:square of infty-operation}, we have
\[
(T^\vee)_\infty(f)
 = i^{-\ell} T^\vee(f_\infty) = i^{-\ell} T^\vee(\pi_{\lambda,\ell}(\tilde w^{-1})f) 
 = i^{-\ell} T(\pi_{\lambda,\ell}(\tilde w_N \tilde w^{-1})f). 
\]
Since $\tilde w_N \tilde w^{-1}=\tilde d(1/\sqrt{N})$ by \eqref{eqn: wN w equals d}, 
\[
\pi_{\lambda,\ell}(\tilde w_N \tilde w^{-1})f(x)
  = N^{\lambda} f(Nx).
\]
Hence 
\[
(T^\vee)_\infty(f)
 = i^{-\ell} N^{\lambda} T(f(Nx)) 
 = i^{-\ell} N^{-\lambda}\alpha(\infty) f(\infty) +  i^{-\ell} N^{\lambda-1} \sum_{n=-\infty}^\infty \alpha(n) \mathcal{F}f\left(\frac nN\right).
\]
Since $T^\vee(f)=(T^\vee)_\infty(f_\infty)$, we get $(\ref{eqn:sum formula for Tvee})$.

$(2)$ $T$ is automorphic for $\widetilde{\Gamma}_0(N)$ with character $\chi$, if and only if $T^\vee$ is automorphic for $\tilde w_N^{-1} \widetilde{\Gamma}_0(N) \tilde w_N$ with character $\chi'$, where $\chi'(\tilde w_N^{-1} \tilde\gamma \tilde w_N)=\chi(\tilde\gamma)$ $(\tilde\gamma  \in \widetilde{\Gamma}_0(N))$. 
By \eqref{eqn:commutation of wN and gamma0N}, if $\ell$ is even, then $\tilde w_N^{-1} \widetilde{\Gamma}_0(N) \tilde w_N =  \widetilde{\Gamma}_0(N)$ and $\chi'(\tilde\gamma ) = \chi(a) = \overline{\chi(d)} = \overline\chi(\tilde\gamma)=\chi_{N,\ell}(\tilde\gamma)$ 
for $\tilde\gamma=\left[\begin{pmatrix} a & b \\ cN & d \end{pmatrix}, \xi\right] \in \widetilde{\Gamma}_0(N)$.  
If $\ell$ is odd, then every element in $\widetilde{\Gamma}_0(N)$ is of the form 
\[
\gamma^*=\left[\begin{pmatrix} a & b \\ cN & d \end{pmatrix}, \varepsilon_d^{-1} \left(\frac{cN}d\right)\right], \quad   
\gamma=\begin{pmatrix} a & b \\ cN & d \end{pmatrix} \in \Gamma_0(N), 
\]
and, again by \eqref{eqn:commutation of wN and gamma0N},  we have
\begin{eqnarray*}
\tilde w_N^{-1} \gamma^* \tilde w_N
 &=& \left[\begin{pmatrix} d & -c \\ -bN & a \end{pmatrix}, \varepsilon_d^{-1} \left(\frac{cN}d\right)\sigma\right] \\
 &=& \left[\begin{pmatrix} d & -c \\ -bN & a \end{pmatrix}, \varepsilon_a^{-1} \left(\frac{-bN}a\right)\right] \cdot
 \left[\begin{pmatrix} 1 & 0 \\ 0 & 1 \end{pmatrix}, \left(\frac{N}d\right)\right] \\
 &=& (w_N^{-1} \gamma w_N)^* \cdot
 \left[\begin{pmatrix} 1 & 0 \\ 0 & 1 \end{pmatrix}, \left(\frac{N}d\right)\right]. 
\end{eqnarray*}
Here the second equality follows from Shimura~\cite[Proposition~1.4]{Shimura73} (or from a direct computation based on basic properties of the quadratic residue symbol in the sense of \cite{Shimura73}).

\end{proof}

\subsection{Proof of Theorem~\ref{thm:ConverseTheoremForCongSubgp}}

The key to the proof of Theorem~\ref{thm:ConverseTheoremForCongSubgp} 
is the translation of the $\widetilde{\Gamma}_0(N)$-automorphy of the distribution $T$ into the form of Ferrar-Suzuki summation formula. 

\begin{proposition}
\label{prop:TwistsOfSumAndCongAutomDist}
Let $T$  be the automorphic distribution for $\widetilde\Delta(N)$ associated with 
the summation formula 
\[
\alpha(\infty)f(\infty)+\sum_{n=-\infty}^{\infty} \alpha(n)(\mathcal{F}f)(n)=
\beta(\infty) f_{\infty}(\infty)+
 \sum_{n=-\infty}^{\infty} \beta(n)(\mathcal{F}f_{\infty})\left(\frac{n}{N}\right) 
\quad (f \in \mathcal{V}^\infty_{\lambda,\ell}).
\]
Then the following two conditions are equivalent:
\begin{enumerate}
\def\labelenumi{(\roman{enumi})}
\item
$T$ is an automorphic distribution for $\widetilde{\Gamma}_0(N)$
with character $\chi$.
\item 
For every $r \in \mathbb{P}_N$ and every Dirichlet character $\psi \in X_r$, 
the following summation formula holds:
\begin{eqnarray}
\lefteqn{ 
\alpha(\infty)\tau_{\psi}(0)f(\infty)+
 \sum_{n=-\infty}^{\infty} \alpha(n)\tau_{\psi}(n)\left(\mathcal{F}f\right)(n) 
} \nonumber \\
 & & \nonumber
 = \overline{\chi(r)}\cdot C_{\ell,r}
 \cdot \psi^*(-N)\cdot
 \biggl\{
 r^{-2\lambda}\cdot \beta(\infty)\tau_{\psi^*}(0)f_{\infty}(\infty)
 \biggr. \\
 & &
 \hspace{4em} \left.
{} + r^{2\lambda-2}
 \sum_{n=-\infty}^{\infty} \beta(n)\tau_{\psi^*}(n)
(\mathcal{F}f_{\infty})\left(\frac{n}{Nr^2}\right)
 \right\} \quad (f \in \mathcal{V}^\infty_{\lambda,\ell}).
 \label{form:TwistedSummationForm}
\end{eqnarray}
\end{enumerate}
\end{proposition}

Theorem~\ref{thm:ConverseTheoremForCongSubgp} can easily be proved by Proposition~\ref{prop:TwistsOfSumAndCongAutomDist}. 
Indeed,  by the converse theorem for Ferrar-Suzuki summation formulas (Theorem \ref{thm:ConverseTheoremForSummationFormula}), if the assumptions $\mathrm{[A1]}_{r,\psi}$ --  $\mathrm{[A4]}_{r,\psi}$ are satisfied by $\xi_\pm(\alpha,\psi;s)$ and  $\xi_\pm(\beta,\psi^*;s)$, then the summation formula 
\begin{eqnarray}
\lefteqn{ 
\alpha_\psi(\infty)f(\infty)+\alpha_\psi(0)\mathcal{F}f(0)
 + \sum_{n \ne 0} \alpha(n)\tau_{\psi}(n)\left(\mathcal{F}f\right)(n) 
} \nonumber \\
 & & 
 = \beta_\psi(\infty)f_\infty(\infty)+\beta_\psi(0)\mathcal{F}f_\infty(0) 
   \nonumber \\
 & &
\qquad {} + \overline{\chi(r)}\cdot C_{\ell,r}
 \cdot \psi^*(-N)\cdot r^{2\lambda-2}
 \sum_{n \ne 0} \beta(n)\tau_{\psi^*}(n)
(\mathcal{F}f_{\infty})\left(\frac{n}{Nr^2}
 \right)
 \label{form:TwistedSummationForm 2}
\end{eqnarray}
holds for some constants $\alpha_\psi(0)$, $\alpha_\psi(\infty)$, $\beta_\psi(0)$, and $\beta_\psi(\infty)$. 
It is sufficient to show that this summation formula coincides with $(\ref{form:TwistedSummationForm})$. 
The four constants $\alpha_\psi(0)$, $\alpha_\psi(\infty)$, $\beta_\psi(0)$ and $\beta_\psi(\infty)$ are determined by $(\ref{form:DefOfA0AndAInfty})$, $(\ref{form:DefOfA0AndAInfty2})$, 
$(\ref{form:DefOfB0AndBInfty})$, and $(\ref{form:DefOfB0AndBInfty2})$. 
If the assumption  $\mathrm{[A5]}_{r,\psi}$ is satisfied, 
then they are actually equal to the values given in $(\ref{form:TwistedSummationForm})$.  
This is easily seen unless $\lambda=1$ and $\ell\equiv 0 \pmod 4$. 
In the exceptional case, we have to be a little careful, since $\alpha_\psi(0)$ and $\beta_\psi(0)$ involve an undetermined constant $c_\psi$.  
If the summation formulas \eqref{form:TwistedSummationForm} and \eqref{form:TwistedSummationForm 2} coincide, then $c_\psi$ can not be arbitrary. 
By using the identities
\[
\xi_e(\beta;0)=-\xi_e(\alpha;0), \quad 
\overline{\chi(r)}\cdot C_{\ell,r}
 \cdot \psi^*(-N)\cdot r^{2\lambda-2}\xi_e(\beta,\psi^*;0)=-\xi_e(\alpha,\psi;0)  
\]
obtained from \eqref{eqn:exceptional identity alpha=-beta}, 
we see that the constant $c_\psi$ is uniquely determined by the constant $c$ in  $\alpha(0)$ and $\beta(0)$ as 
\[
c_\psi=\left(\frac{\overline{\chi(r)}-1}4 \xi_e(\alpha;0)+\frac{c(\overline{\chi(r)}+1)}2\right)\tau_\psi(0).
\]
With this choice of $c_\psi$, it follows from the assumption  $\mathrm{[A5]}_{r,\psi}$ that $\alpha_\psi(0)$, $\alpha_\psi(\infty)$, $\beta_\psi(0)$ and $\beta_\psi(\infty)$ are actually equal to the values given in $(\ref{form:TwistedSummationForm})$. 
Hence the summation formula \eqref{form:TwistedSummationForm 2} coincides with \eqref{form:TwistedSummationForm}  also in the case where $\lambda=1$ and $\ell\equiv 0 \pmod 4$.   
Therefore, if all the assumptions in Theorem~\ref{thm:ConverseTheoremForCongSubgp}  are satisfied, then $T$ becomes an automorphic distribution for $\widetilde{\Gamma}_0(N)$. 
This shows that Proposition~\ref{prop:TwistsOfSumAndCongAutomDist} implies   Theorem~\ref{thm:ConverseTheoremForCongSubgp}.  

The rest of this section is devoted to the proof of  Proposition~\ref{prop:TwistsOfSumAndCongAutomDist}. 
For a Dirichlet character $\psi$ mod a prime number $r$, we put 
\begin{equation}
\label{form:DefOfTwists}
T_\psi(f) = \sum_{\substack{m\ \mathrm{mod}\ r \\ (m,r)=1}} \psi(m) \left(\pi^*_{\lambda,\ell}\left(\tilde n\left(-\frac mr\right)\right)T\right)(f).
\end{equation}
Since $T$ is invariant under $\tilde{n}(k)$ $(k\in\Z)$, 
the right hand side does not depend on the choice of 
the representatives of mod $r$. 

\begin{lemma}
\label{lem:distribution interpretation of sum formula for T psi}
The summation formula $(\ref{form:TwistedSummationForm})$ is equivalent to the identity
\begin{equation}
T_\psi(f) = \overline{\chi(r)}\cdot C_{\ell,r} \cdot \psi^*(-N) 
             \pi^*_{\lambda,\ell}(\tilde w_{Nr^2})\left((T^\vee)_{\psi^*}\right)(f). 
\label{eqn:T psi = T vee psi}
\end{equation}遯ｶ�｢
\end{lemma}

\begin{proof}
It is enough to prove that the left and right hand sides of $(\ref{eqn:T psi = T vee psi})$ 
coincide with the left and right hand sides of $(\ref{form:TwistedSummationForm})$, respectively.  
We put
\[
\varphi_t(x) = \left(\pi_{\lambda,\ell}\left(\tilde n\left(t\right)\right)f\right)(x).
\]
Then $\varphi_t(x)=f(x-t)$. 
Note that 
\begin{eqnarray*}
(\varphi_t)_\infty(x)
 &=& (\mathrm{sgn}\,x)^{\ell/2} |x|^{-2\lambda} f\left(-\frac{1+tx}x\right) \\
 &=& (\mathrm{sgn}(1+tx))^{\mathrm{sgn}(t)\ell/2} |1+tx|^{-2\lambda} 
       f_\infty\left(\frac{x}{1+tx}\right), 
\end{eqnarray*}
and hence $\varphi_t(\infty)=(\varphi_t)_\infty(0)=f_\infty(0)=f(\infty)$. 
Moreover, by \eqref{eqn:translate for FT}, 
we have $\mathcal{F}(\varphi_t)(x)=e^{2\pi i tx}\mathcal{F}f(x)$. 
By the definition of $T_\psi$, we have
\begin{eqnarray*}
T_\psi(f) 
 &=& \sum_{\substack{m\ \mathrm{mod}\ r \\ (m,r)=1}} 
  \psi(m) T\left(\varphi_{m/r}\right) \\
 &=& \sum_{\substack{m\ \mathrm{mod}\ r \\ (m,r)=1}} 
  \psi(m) \cdot
 \left(
 \alpha(\infty) f(\infty)
 + \sum_{n=-\infty}^{\infty} \alpha(n) e^{2\pi i mn/r} \mathcal{F}f(n)
 \right) \\
 &=& \tau_\psi(0) \alpha(\infty) f(\infty)
 + \sum_{n=-\infty}^{\infty} \alpha(n) \tau_\psi\left(n\right) \mathcal{F}f(n).
\end{eqnarray*}
Thus $T_\psi(f)$ coincides with the left hand side of  $(\ref{form:TwistedSummationForm})$.  
Applying the same argument to the Fourier expansion of $T^\vee$ in Lemma \ref{lem: T and T vee}, we get 
\[
(T^\vee)_{\psi^*}(f) 
 =  N^\lambda \tau_{\psi^*}(0)\beta(\infty)f(\infty) 
   + N^{1-\lambda} \sum_{n=-\infty}^\infty  \tau_{\psi^*}(n)\beta(n)(\mathcal{F}f)(n). 
\]
If we put $f_{\infty,Nr^2}(x)=  f_\infty(Nr^2x)$, we have 
\[
\pi_{\lambda,\ell}(\tilde w_{Nr^2}^{-1})f(x) 
 = \pi_{\lambda,\ell}\left(\tilde d\left(\frac{1}{r\sqrt N}\right) \cdot \tilde w^{-1}\right)f(x) 
 = (Nr^2)^{\lambda} f_\infty(Nr^2x)
 = (Nr^2)^{\lambda} f_{\infty,Nr^2}(x), 
\]
and 
\begin{gather*}
f_{\infty,Nr^2}(\infty) = i^\ell (Nr^2)^{-2\lambda} f(0) = (Nr^2)^{-2\lambda} f_\infty(\infty), \\
\mathcal{F}f_{\infty,Nr^2}(n)=(Nr^2)^{-1}\mathcal{F}f_\infty\left(\frac{n}{Nr^2}\right). 
\end{gather*}
Therefore 
\begin{eqnarray*}
\lefteqn{
\pi^*_{\lambda,\ell}(\tilde w_{Nr^2})\left((T^\vee)_{\psi^*}\right)(f) 
} \\
 & & =   (Nr^2)^{\lambda} (T^\vee)_{\psi^*}(f_{\infty,Nr^2}) \\
 & & = (Nr^2)^{\lambda} 
   \left\{
    N^\lambda \tau_{\psi^*}(0)\beta(\infty) (Nr^2)^{-2\lambda} f_\infty(\infty) 
    \vphantom{ \sum_{n=-\infty}^\infty }
   \right. \\
 & & \hspace{5em} \left.
   + N^{1-\lambda} \sum_{n=-\infty}^\infty  \tau_{\psi^*}(n)\beta(n)   
       (Nr^2)^{-1}\mathcal{F}f_\infty\left(\frac{n}{Nr^2}\right)
   \right\} \\
 & & = r^{-2\lambda} \tau_{\psi^*}(0)\beta(\infty) f_\infty(\infty) 
   + r^{2\lambda-2}\sum_{n=-\infty}^\infty  \tau_{\psi^*}(n)\beta(n)   
       \mathcal{F}f_\infty\left(\frac{n}{Nr^2}\right).  
\end{eqnarray*}
This shows that $\overline{\chi(r)}\cdot C_{\ell,r} \cdot \psi^*(-N) 
             \pi^*_{\lambda,\ell}(\tilde w_{Nr^2})\left((T^\vee)_{\psi^*}\right)(f)$ coincides with the right hand side of  $(\ref{form:TwistedSummationForm})$.  
\end{proof}

To complete the proof of Proposition~\ref{prop:TwistsOfSumAndCongAutomDist}, it suffices to prove Lemma~\ref{prop:AutomorphyAndTwists} below. 

\begin{lemma}
\label{prop:AutomorphyAndTwists}
Let $T$ be an automorphic distribution for $\widetilde\Delta(N)$, and $\chi$ a Dirichlet character mod $N$. 
Then, for every odd prime number $r$ with $(r, N)=1$, 
the following conditions
(A) and (B) are equivalent.
\begin{enumerate}
\renewcommand{\labelenumi}{(\Alph{enumi})} 
\item 
$\pi^*_{\lambda, \ell}(\tilde{\gamma})T^{\vee}=\chi_{N,\ell}(r)\cdot T^{\vee}$
for any  $\gamma\in \Gamma_0(N)$
of the form
$\gamma=
\begin{pmatrix}
\ast & \ast\\
\ast & r\\
\end{pmatrix}$.
\item $\pi^*_{\lambda, \ell}(\tilde{w}_{Nr^2}^{-1})(T_{\psi})=\overline{\chi(r)}\cdot 
C_{\ell,r}\cdot \psi^*(-N)\cdot
(T^{\vee})_{\psi^*}$
for any Dirichlet character $\psi$ of mod~$r$. 
\end{enumerate}
Here, for a given $\gamma\in \Gamma_0(N)$, 
we put $\tilde{\gamma}=[\gamma,1]$ if $\ell$ is even, 
and $\tilde{\gamma}=\gamma^*$ if $\ell$ is odd. 
(For the definition of $\gamma^*$, see \eqref{form:DefOfGammaStar}.) 
\end{lemma}

Before giving a proof of Lemma~\ref{prop:AutomorphyAndTwists}, let us see that the lemma actually yields Proposition~\ref{prop:TwistsOfSumAndCongAutomDist}.   
By Lemma~\ref{lem:distribution interpretation of sum formula for T psi}, 
the condition (ii) in Proposition~\ref{prop:TwistsOfSumAndCongAutomDist} is equivalent to the condition 
\begin{description}
\item[$(*)$]
$\pi^*_{\lambda, \ell}(\tilde{w}_{Nr^2}^{-1})(T_{\psi})=\overline{\chi(r)}\cdot 
C_{\ell,r}\cdot \psi^*(-N)\cdot
(T^{\vee})_{\psi^*}$
for  any $r \in \mathbb{P}_N$ and $\psi \in X_r$. 
\end{description}
By Lemma~\ref{prop:AutomorphyAndTwists}, this is equivalent to the condition 
\begin{description}
\item[$(**)$]
$\pi^*_{\lambda, \ell}(\tilde{\gamma})T^{\vee}=\chi_{N,\ell}(r)\cdot T^{\vee}$
for any  $\gamma\in \Gamma_0(N)$
of the form
$\gamma=
\begin{pmatrix}
\ast & \ast\\
\ast & r
\end{pmatrix}$ $(r \in \mathbb{P}_N)$.
\end{description}
It is wellknown that $\Gamma_0(N)$ is generated by elements $\gamma$ of the form appearing in $(**)$ and $\begin{pmatrix} 1 & 1 \\ 0 & 1 \end{pmatrix}$ (see, for example, \cite[Proof of Theorem 4.3.15]{Miyake} or \cite[p.~154]{Weil}). 
Hence $\widetilde{\Gamma}_0(N)$ is generated by 
\[
\tilde\gamma\ (\gamma= \begin{pmatrix}
\ast & \ast\\
\ast & r\\
\end{pmatrix} \in \Gamma_0(N),\ r \in \mathbb{P}_N), \quad \tilde n(1), \quad (\text{and $[1_2,-1]$, if $\ell$ is even}).
\]
The distribution $T^\vee$ is invariant under the action of $\tilde n(1)$, since $T^\vee$ is periodic of period $1$, and, if $\ell$ is even, then $T^\vee$ is invariant also under the action of $[1_2,-1]$. 
Hence,  the condition $(**)$ is equivalent to the condition 
\begin{description}
\item[$(***)$]
$\pi^*_{\lambda, \ell}(\tilde{\gamma})T^{\vee}=\chi_{N,\ell}(\tilde{\gamma})\cdot T^{\vee}$
for any  $\tilde\gamma\in \widetilde{\Gamma}_0(N)$.
\end{description}
By Lemma~\ref{lem: T and T vee} (2), this is equivalent to the condition (i) in   Proposition~\ref{prop:TwistsOfSumAndCongAutomDist}.   
This proves Proposition~\ref{prop:TwistsOfSumAndCongAutomDist}.    

\begin{proof}[Proof of Lemma~\ref{prop:AutomorphyAndTwists}]
By a straightforward calculation using Lemma~\ref{thm:FactorSystem}, for any  $\gamma = \begin{pmatrix} u & k \\ mN & r\end{pmatrix} \in \Gamma_0(N)$, we have 
\begin{equation}
\label{form:CreateElementOfGamma}
\tilde{w}_{Nr^2}^{-1} \cdot \tilde{n}\left(-\frac{m}{r}\right) =
\tilde{n}\left(-\frac{k}{r}\right)\cdot
\tilde{\gamma}
\cdot \tilde{w}_{N}^{-1}\cdot 
\begin{cases}
\left[1_2, 1\right] & (\text{$\ell$ is even}), \\[5pt]
\left[1_2, \varepsilon_r\left(\dfrac{mN}{r}\right)\right]
& (\text{$\ell$ is odd}).
\end{cases}
\end{equation}
For any $r,m$ with $(mN,r)=1$, take a $\gamma(r,m) \in \Gamma_0(N)$ of the form  $\begin{pmatrix} u & k \\ mN & r \end{pmatrix}$. Then $k \bmod r$ is uniquely determined by $r$ and $m$, and  $\psi(m)=\overline{\psi(-kN)}$ for $\psi \in X_r$. 
Moreover, by  \eqref{form:CreateElementOfGamma}, 
$\tilde{n}\left(-\frac{k}{r}\right)\cdot\tilde{\gamma}(r,m)$ is independent of the choice of $\gamma(r,m)$. 
Since  
\[
\pi^*_{\lambda, \ell}\left(
\left[1_2, \varepsilon_r\left(\dfrac{mN}{r}\right)\right]
\right)T = \varepsilon_r^{\ell} \left(\frac{mN}{r}\right)^{\ell} T = 
\varepsilon_r^{\ell} \left(\frac{-k}{r}\right)T  
\]
for an odd $\ell$, 
we obtain
\begin{eqnarray*}
\lefteqn{
 \pi^*_{\lambda, \ell}(\tilde{w}_{Nr^2}^{-1})\left(T_{\psi}\right)
  =  \sum \begin{Sb}   m\!\! \mod{r}  \\   (m,r)=1 \end{Sb}
 \psi(m)\cdot \pi^*_{\lambda, \ell}(\tilde{w}_{Nr^2}^{-1})\pi^*_{\lambda, \ell}
 \left(\tilde{n}\left(-\frac{m}{r}\right)\right)T 
} \\
 & & = C_{\ell,r}\left(\frac{-1}r\right)^\ell\cdot \overline{\psi(-N)}
        \left(\sum \begin{Sb}   m\!\! \mod{r}  \\   (m,r)=1 \end{Sb}
 \psi^*(k)\cdot 
\pi^*_{\lambda, \ell}
 \left(\tilde{n}\left(-\frac{k}{r}\right)\right)
\pi^*_{\lambda, \ell}\left(
\tilde{\gamma}(r,m)
\right)T^{\vee} \right). 
\end{eqnarray*}
If the condition (A) is satisfied, then we have
\[
\pi^*_{\lambda, \ell}\left(
\tilde{\gamma}(r,m)
\right)T^{\vee} = \chi_{N,\ell}(r) \cdot T^{\vee}=
\left(\dfrac{N}{r}\right)^\ell  \overline{\chi(r)}\cdot T^{\vee}
\]
and hence 
\[
 \pi^*_{\lambda, \ell}(\tilde{w}_{Nr^2}^{-1})\left(T_{\psi}\right)
   =  \overline{\chi(r)}\cdot C_{\ell,r} \cdot \psi^*(-N)
      \sum \begin{Sb}   m\!\! \mod{r}  \\   (m,r)=1 \end{Sb}
        \psi^*(k)\cdot 
      \pi^*_{\lambda, \ell} \left(\tilde{n}\left(-\frac{k}{r}\right)\right) T^{\vee}. 
\]
Note that, as $m$ ranges over a reduced residue system modulo $r$, 
so does $k$, too. 
Hence we see that the right hand side is equal to  
\[
\overline{\chi(r)}\cdot C_{\ell,r}\cdot 
\psi^*(-N)\cdot (T^{\vee})_{\psi^*}
\]
and we obtain the identity in the condition (B).  

Next we prove that the condition (B) implies the condition (A). 
From the orthogonality relation of Dirichlet characters, it follows that
\begin{equation}
\label{form:ByCharacterSummation}
\pi^*_{\lambda, \ell}\left(\tilde{n}\left(-\frac{m}{r}\right)\right)T=
\frac{1}{r-1}
\sum_{\psi \in X_r}
\overline{\psi(m)}\cdot T_{\psi}.
\end{equation}
For any  
$\gamma=\begin{pmatrix}
u & k \\
mN & r
\end{pmatrix}\in \Gamma_0(N)$,   
we have by \eqref{form:CreateElementOfGamma}
\[
\pi^*_{\lambda, \ell}(\tilde{w}_{Nr^2}^{-1})\cdot 
\pi^*_{\lambda, \ell}\left(\tilde{n}\left(-\frac{m}{r}\right)\right)T
 = C_{\ell,r}\left(\frac{-k}r\right)^\ell
\pi^*_{\lambda, \ell}\left(\tilde{n}\left(-\frac{k}{r}\right)\right)\cdot 
\pi^*_{\lambda, \ell}(\tilde{\gamma})T^{\vee}.
\]
On the other hand, if the condition (B) is satisfied, then from \eqref{form:ByCharacterSummation} and $\psi(m)=\overline{\psi(-kN)}$, it follows that  
\begin{eqnarray*}
\pi^*_{\lambda, \ell}(\tilde{w}_{Nr^2}^{-1})\cdot 
\pi^*_{\lambda, \ell}\left(\tilde{n}\left(-\frac{m}{r}\right)\right)T
&=&
\frac{1}{r-1}
\sum_{\psi \in X_r}
\overline{\psi(m)}\cdot 
\pi^*_{\lambda, \ell}(\tilde{w}_{Nr^2}^{-1})
\left(T_{\psi}\right) 
\\
&=& 
\frac{1}{r-1}
\sum_{\psi \in X_r}
\overline{\psi(m)}\cdot 
\overline{\chi(r)}\cdot 
C_{\ell,r}\cdot 
\psi^*(-N)\cdot
(T^{\vee})_{\psi^{*}} 
\\
&=& C_{\ell,r}\left(\frac{-k}r\right)^\ell \chi_{N,\ell}(r)
\cdot 
\frac{1}{r-1}
\sum_{\psi \in X_r}
\overline{\psi^*(k)}
(T^{\vee})_{\psi^{*}} 
\\
&=& C_{\ell,r}\left(\frac{-k}r\right)^\ell \chi_{N,\ell}(r)
\cdot 
\pi^*_{\lambda, \ell}\left(\tilde{n}\left(-\frac{k}{r}\right)\right)
T^{\vee}.
\end{eqnarray*}
By equating these two, we obtain
\[
\pi^*_{\lambda, \ell}
 \left(\tilde{n}\left(-\frac{k}{r}\right)\right)\pi^*_{\lambda, \ell}\left(
 \tilde{\gamma}\right)
 T^{\vee}=
\chi_{N,\ell}(r)\cdot 
\pi^*_{\lambda, \ell}\left(\tilde{n}\left(-\frac{k}{r}\right)\right)
T^{\vee},
\]
and by letting $\pi^*_{\lambda, \ell}(\tilde{n}(k/r))$ act on both sides, 
we conclude that $\pi^*_{\lambda, \ell}\left(\tilde{\gamma}\right)
 T^{\vee}=\chi_{N,\ell}(r)\cdot T^{\vee}$.
\end{proof}




\section{Maass forms arising from certain zeta functions in two variables}

In this section, applying our converse theorem for Maass forms (Theorem \ref{corollary:Maassforms}), we show that the two-variable zeta functions related to quadratic forms studied by Peter \cite{Peter} and the fourth author \cite{Ueno} independently can be viewed as $L$-functions associated with Maass forms for $\widetilde{\Gamma}_0(N)$.

\subsection{Zeta functions in two variables associated with quadratic forms}

We recall the definition of the zeta functions as given in \cite{Ueno}.
Let $Q(x)$ be a non-degenerate integral quadratic form on $V=\C^{m+2}$.
We assume that $Q(x)$ is of the form
\[
 Q(x)= x_{0}x_{m+1}+\sum_{1\leq i, j\leq m}
 a_{ij} x_{i}x_{j} \quad \text{with}\; \; x={}^t(x_0,x_1,\ldots,x_{m+1}),
\]
where $a_{ij}= a_{ji}\in \frac{1}{2}\Z \; (i\neq j)$ and 
$a_{ii}\in \Z$. The matrix of 
$Q$ is given by
\[
 \begin{pmatrix}
  0 & 0 & \frac{1}{2} \\
  0 & A& 0 \\
  \frac{1}{2} &  0 & 0
 \end{pmatrix}, \qquad A:= (a_{ij}).
\]
Let $p$ (resp.\ $q$) be the number of 
positive (resp.\ negative) eigenvalues of $A$, and put
\[
 D=\det(2A).
\]
We define another quadratic forms $Q^*$ by 
\[
 Q^{*}(y)= y_{0}y_{m+1} + \frac{1}{4}
 \sum_{1\leq i, j\leq m} a_{ij}^{*} y_{i}y_{j} \quad
 \text{with}\; \; A^{-1} = (a_{ij}^{*})\; \; \text{and} 
 \; \; y={}^t(y_0,y_1,\ldots,y_{m+1}).
\]
Let $N$ be the level of $2Q$.
By definition, $N$ is the smallest positive
integer such that $NQ^*(y)$ is an integral quadratic form. 
For $l\in \Z_{>0}$ and $n\in \Z$, we put
\begin{align*}
 r(l,n) &= \sharp\left\{
 v\in \Z^{m}/l\Z^m \, |\, {}^tvAv\equiv n \pmod{l}
 \right\}, \\
 r^{*}(l,  n) &=
  \sharp \left\{
  v^* \in \Z^{m}/2lA \Z^m \, |\, 4^{-1}N\cdot {}^tv^*A^{-1}v^*\equiv n \pmod{Nl}  
    \right\},
\end{align*}
and define the Dirichlet series $Z(n, w)$ and $Z^{*}(n,w)$ by
\begin{equation}
 \label{form:DefOfZandZstar}
 Z(n,w)=\sum_{l=1}^{\infty}\frac{r(l,n)}{l^{w}}, \qquad
 Z^{*}(n,w)=\sum_{l=1}^{\infty} \frac{r^{*}(l,n)}{l^w}.
\end{equation}
The Dirichlet series $Z(n, w)$ and $Z^{*}(n,w)$ converge absolutely for $\Re(w)>m$ and have analytic continuations to meromorphic functions of $w$ on the whole complex plane (\cite[Lemmas 3.8 and 3.9]{PeterCrelle}, \cite[Lemma 3.3]{Ueno}). 
The zeta functions considered in  \cite{Peter} and \cite{Ueno} are 
\begin{equation}
  \zeta_{\pm}(w, s) =  \sum_{n=1}^{\infty}Z(\pm n, w)n^{-s},\qquad 
  \zeta_{\pm}^{*}(w,s) = N^{s}\cdot  \sum_{n=1}^{\infty} Z^{*}(\pm n, w) n^{-s}, 
\label{eqn:Def of Ueno zeta}
\end{equation}
which are absolutely convergent for $\Re(w)>m$ and $\Re(s)>1$. 
The zeta functions $\zeta_{\pm}(w, s)$ and $\zeta_{\pm}^{*}(w,s)$ have analytic  continuations to meromorphic functions of $(w,s)$ in $\C^2$ and satisfied the functional equation 
 \begin{equation}
  \label{form:FEofUenoZeta}
  \begin{pmatrix}
   \zeta_{+}^{*}\left(w, \ \frac{m}{2}+1-w-s\right) \\
   \zeta_{-}^{*}\left(w, \ \frac{m}{2}+1-w-s\right)
  \end{pmatrix}
   =  \Gamma(w, s)
 \begin{pmatrix}
 \zeta_{+}(w, s) \\
 \zeta_{-}(w, s)
 \end{pmatrix},
 \end{equation}
 where
\[
  \Gamma(w,s) = 2|D|^{-1/2} (2\pi)^{m/2-w-2s}
 \Gamma(s)\Gamma\left(w+s-\frac{m}{2}
 \right)
 \begin{pmatrix}
  \cos\left( \frac{\pi(w+2s-p)}{2}\right) & \cos\left( \frac{\pi(w-p)}{2}\right)  \\[5pt]
  \cos\left( \frac{\pi(w-q)}{2}\right)  & \cos\left(\frac{\pi(w+2s-q)}{2}\right) 
 \end{pmatrix}.
 \]
If $w$ is specialized to an integer satisfying $w \equiv p+1 \pmod 2$ (resp.\  $w \equiv q+1 \pmod 2$), then the Gamma matrix $\Gamma(w,s)$ becomes a triangular matrix and the functional equation relates $\zeta_+^*$ (resp.\  $\zeta_-^*$) to $\zeta_+$  (resp.\ $\zeta_-$) only.  
The specialized functional equation can be transformed into a functional equation of Hecke type, and it is proved in \cite[Theorem 1.1]{Peter} and \cite[Theorem 4.5]{Ueno} that the specialized zeta functions are the Mellin transforms of holomorphic modular forms for $\Gamma_0(k|D|)$ ($k=1,2$).  
In the following we construct Maass forms from the original zeta functions $\zeta_\pm$ and $\zeta^*_\pm$ by using Theorem \ref{corollary:Maassforms}. 

We normalize  the zeta functions by
\begin{align}
\xi_\pm(w,s) &= e^{\pi i(p-q)/4}|D|^{-1/2} 
  \zeta_{\pm}(w,s)
  =\sum_{n=1}^{\infty}
  \frac{e^{\pi i(p-q)/4}|D|^{-1/2} Z(\pm n,w)}{n^s}, \\
\xi^*_\pm(w,s)  &= N^{-s}\cdot  
  \zeta_{\pm}^*(w,s)=\sum_{n=1}^{\infty} \frac{Z^*(\pm n,w)}{n^s},
\end{align}
and define the completed zeta functions
$\Xi_\pm(w,s)$ and $\Xi_\pm^*(w,s)$ by
\[
 \Xi_\pm(w,s)= (2\pi)^{-s}\Gamma(s)\cdot 
 \xi_\pm(w,s), 
 \qquad 
  \Xi_\pm^*(w,s)= (2\pi)^{-s}\Gamma(s)\cdot \xi_\pm^*(w,s).
\]
Then, 
by an elementary calculation using the Euler reflection formula 
\[
 \frac{1}{\Gamma\left(w+s-\frac{m}{2}\right)}=
  \frac{\Gamma\left(\frac{m}{2}+1-w-s\right)\sin \pi
  \left(w+s-\frac{m}{2}\right)}{\pi}
\]
we can rewrite the functional equation~\eqref{form:FEofUenoZeta} as
\begin{eqnarray}
 \gamma(s)\begin{pmatrix}
\Xi_{+}(w,s) \\[3pt]
\Xi_{-}(w,s)
\end{pmatrix} 
 & = &
N^{m/2+1-w-s}\cdot 
\Sigma(p-q)
 \gamma\left(\frac{m}{2}+1-w-s\right) \nonumber \\
 & & {} \times 
  \begin{pmatrix}
   \Xi_{+}^{*}\left(w, \ \frac{m}{2}+1-w-s\right) \\[3pt]
   \Xi_{-}^{*}\left(w, \ \frac{m}{2}+1-w-s\right)
  \end{pmatrix}
\label{form:FEofLambda}
\end{eqnarray}
(for the definitions of $\gamma(s)$ and $\Sigma(p-q)$, see \eqref{form:DefOfGammaAndSigma}).  
This is precisely the same functional equation in the condition [A3] in \S \ref{subsection:3.3} with 
\[
 \ell \equiv p-q \pmod{4}, \qquad 2-2\lambda=\frac{m}{2}+1-w. 
\]
Therefore it is reasonable to expect further that Theorem \ref{corollary:Maassforms} can apply to the normalized zeta functions $\xi_\pm(w,s)$ and $\xi^*_\pm(w,s)$ to obtain Maass forms. 
Indeed, we can obtain the following theorem. 

\begin{theorem}
\label{thm:Maass form for Ueno zeta}
Let $\lambda$ be a complex number with $\Re(\lambda)>(m+2)/4$, 
and $\ell$ an integer such that $\ell \equiv p-q \pmod{4}$. 
Define $C^{\infty}$-functions $F(z)$ and $G(z)$ on $\mathcal{H}$ by
\begin{align*}
F(z) &= 
\zeta\left(2\lambda-\frac{m}{2}\right) \cdot y^{\lambda-(\ell/4)}+
 (-1)^{(p-q-\ell)/4} \frac{Z\left(0, 2\lambda+\frac{m}{2}-1\right)}{|D|^{1/2}}
 \frac{(2\pi) 2^{1-2\lambda} \Gamma(2\lambda-1)}
 {\Gamma\left(\lambda+\frac{\ell}{4}\right)
 \Gamma\left(\lambda-\frac{\ell}{4}\right)} \cdot y^{1-\lambda-\ell/4} \\[5pt]
 &\; +  (-1)^{(p-q-\ell)/4} \sum
 \begin{Sb}
 n=-\infty \\
 n\neq 0
 \end{Sb}^{\infty} 
\frac{Z\left(n, 2\lambda+\frac{m}{2}-1\right)}{|D|^{1/2}}
\cdot 
 \frac{\pi^{\lambda} \cdot  |n|^{\lambda-1}}
{\Gamma\left(\lambda+\frac{\sgn(n)\ell}{4}\right)} \cdot 
 y^{-\ell/4}\, W_{\frac{\sgn(n)\ell}{4}, \lambda-\frac{1}{2}}\left(4\pi|n|y\right)
\cdot \mathbf{e}[nx], \\
G(z) &= 
 e^{-\pi i(p-q)/4} N^\lambda \frac{\zeta\left(2\lambda-\frac{m}{2}\right)}{|D|^{1/2}} 
 \cdot y^{\lambda-(\ell/4)} \\
 &\; +   i^{-\ell/2} N^{1-\lambda} \cdot \frac{Z^*\left(0, 2\lambda+\frac{m}{2}-1\right)}{|D|}
 \frac{(2\pi) 2^{1-2\lambda} \Gamma(2\lambda-1)}
 {\Gamma\left(\lambda+\frac{\ell}{4}\right)
 \Gamma\left(\lambda-\frac{\ell}{4}\right)} \cdot y^{1-\lambda-\ell/4} \\[5pt]
 &\; + i^{-\ell/2} \sum
 \begin{Sb}
 n=-\infty \\
 n\neq 0
 \end{Sb}^{\infty} 
  N^{1-\lambda} \cdot Z^*\left(n, 2\lambda+\frac{m}{2}-1\right)
\cdot 
 \frac{\pi^{\lambda} \cdot  |n|^{\lambda-1}}
{\Gamma\left(\lambda+\frac{\sgn(n)\ell}{4}\right)} \cdot 
 y^{-\ell/4}\, W_{\frac{\sgn(n)\ell}{4}, \lambda-\frac{1}{2}}\left(4\pi|n|y\right)
\cdot \mathbf{e}[nx].
\end{align*}
Then, $F(z)$(resp.\ $G(z)$) is a Maass form for $\widetilde{\Gamma}_0(N)$ of weight
$\ell/2$ with character $\chi_K$ (resp.\ $\chi_{K_N}$). 
Here we denote by $\chi_K$ and $\chi_{K_N}$ the Kronecker characters associated to the fields
\[
K =
  \begin{cases}
  \displaystyle
  \Q\left(\sqrt{(-1)^{m/2}D}\right)
  & (\text{$\ell$ is even}), \\[8pt]
  \displaystyle
  \Q\left(\sqrt{2|D|}\right) &  (\text{$\ell$ is odd})
  \end{cases}
\]
and
\[
K_N =
  \begin{cases}
  \displaystyle
  \Q\left(\sqrt{(-1)^{m/2}D}\right)
  & (\text{$\ell$ is even}), \\[8pt]
  \displaystyle
  \Q\left(\sqrt{2|D|N}\right) &  (\text{$\ell$ is odd}),
  \end{cases}
\]
respectively. 
Further we have 
\[
(F\big|_\ell \widetilde w_N)(z) = G(z).
\]
\end{theorem}

\begin{remark}
In case $A$ is positive definite, the theorem above is obtained by Mizuno \cite{Mizuno}. His proof is based on the fact that the Maass forms in the theorem appear as the coefficients of the theta  expansion of the Jacobi-Eisenstein series. 
In case $m=1$ and $A=(1)$, then the zeta functions $\zeta_\pm(w,s)$ and $\zeta^*_\pm(w,s)$ are the double Dirichlet series studied first by Shintani \cite{Shintani} (and later in \cite {SatoFE} for general $A=(a)$ and  \cite{SaitoBQF}). 
In this special case, by using their converse theorem for double Dirichlet series,  Diamantis and Goldfeld (\cite{DG} and \cite{DG2}) proved more precisely that the corresponding Maass forms are linear combinations of metaplectic Eisenstein series for $\widetilde{\Gamma}_0(4)$. 
\end{remark}

For a Dirichlet character $\psi$ modulo an odd prime number $r$ not dividing $N$, 
the twists of $\xi_\pm(w,s)$ and $\xi^*_\pm(w,s)$ by $\psi$ (in the sense of \S \ref{subsect:4.1}) are given by 
\begin{align}
\xi_\pm(\psi;w,s) &= \sum_{n=1}^{\infty}
  \frac{e^{\pi i(p-q)/4}|D|^{-1/2} \tau_\psi(\pm n)Z(\pm n,w)}{n^s}, \\
\xi^*_\pm(\psi;w,s)  &= \sum_{n=1}^{\infty} \frac{\tau_\psi(\pm n)Z^*(\pm n,w)}{n^s}.
\end{align}
For the proof of Theorem \ref{thm:Maass form for Ueno zeta}, it is sufficient to check the conditions  [A1], [A2], [A3], and [A4] in \S \ref{subsection:3.3} for $\xi_\pm(w,s)$  and $\xi^*_\pm(w,s)$, and the conditions $\mathrm{[A1]}_{r,\psi}$, $\mathrm{[A2]}_{r,\psi}$, $\mathrm{[A3]}_{r,\psi}$, $\mathrm{[A4]}_{r,\psi}$, and $\mathrm{[A5]}_{r,\psi}$ in \S \ref{subsect:4.1} for $\xi_\pm(\psi;w,s)$  and $\xi^*_\pm(\psi^*;w,s)$. 
Information necessary for the proof of these conditions is essentially contained in \cite{Ueno}.  We therefore give only a brief explanation on the proof.

\subsection{Prehomogeneous zeta functions}

The method employed in \cite{Ueno} is based on the theory of prehomogeneous vector spaces. 
We first introduce the prehomogeneous vector spaces with which 
the zeta functions $\zeta_\pm(w,s)$ and $\zeta^*_\pm(w,s)$ are associated. 

Let $H$ be the subgroup of 
the special orthogonal group
$SO(Q)=\{g\in SL_{m+2}(\C)\, |\, Q(gx)= Q(x) \; \text{for all} \;
 x\in V\}$ defined by
\[
 H =
 \left\{
 \left.
 \begin{pmatrix}
  a& -2a {}^{t}uA & -a {}^tuAu \\
  0 & 1_m & u \\
  0 & 0 & a^{-1}
 \end{pmatrix}
 \; \right|\;
 \begin{array}{c}
  a\in \C, \, a\neq 0 \\
  u\in \C^{m}
 \end{array}
 \right\}.
\]
The group
$H\times GL_{1}(\C)$ acts on $V$ by
\[
 \begin{array}{ll}
  \rho(h, t)x = th x &
   (x\in V, \, (h,t)\in H\times GL_{1}(\C)),    \\
  \rho^{*}(h,t)y = t^{-1}\; {}^{t} h^{-1} y &
   (y\in V, \, (h, t)\in H\times GL_{1}(\C)).
 \end{array}
\]
Here we have identified the dual space $V^*$ with
$V$ via the bilinear functional
\[
 \left\langle x, y\right\rangle=\sum_{i=0}^{m+1}
 x_{i}y_{i} \qquad (x={}^{t}(x_{0}, x_{1},\dots, x_{m+1}),\;
 y= {}^{t}(y_{0}, y_{1}, \dots, y_{m+1})\in V),
\]
so that the contragradient representation of $\rho$ coincides with
$\rho^*$. Then the triplets 
$(H\times GL_{1}(\C), \rho, V)$ and 
$(H\times GL_{1}(\C), \rho^{*}, V)$
are prehomogeneous vector spaces with the singular sets
\begin{align*}
 S&=\{x\in V\, |\, x_{m+1}= 0\} \cup \{x\in V\, |\, Q(x)= 0\}, \\
 S^{*}&=\{y\in V\, |\, y_{0}=0\}\cup \{y\in V\, |\, Q^{*}(y)= 0\}.
\end{align*}
Let $V_{\R}=\R^{m+2}$ and 
\begin{align*}
 V_{\pm}&=
 \{x\in V_{\R}\, |\, x_{m+1}Q(x)\ne0,\ Q(x)/|Q(x)|=\pm1\}, \\
 V_{\pm}^{*} &=
 \{y\in V_{\R}\, |\, y_0Q^*(y)\ne0,\ Q^*(y)/|Q^*(y)|=\pm1\}.
\end{align*}
Further, we put
\begin{gather*}
 \Gamma_H = H \cap GL_{m+2}(\Z) = 
 \left\{
 \left.
 \begin{pmatrix}
  \epsilon & -2\epsilon\, {}^{t}\! uA & - \epsilon\,{}^tuAu \\
  0 & 1_{m} &u \\
  0 & 0 & \epsilon
 \end{pmatrix}\;
 \right| \; u\in\Z^{m},\  \epsilon=\pm1
 \right\}, \\
 V_{\Q} = \Q^{m+2}, \qquad  V_{\Z} = \Z^{m+2}.
\end{gather*}

We call a function $\phi:V_\Q \rightarrow \C$ a {\it Schwartz-Bruhat\/} function, if there exist lattices $L_1 \subset L_2$ in $V_\Q$ such that $\phi$ vanishes outside $L_2$ and factors through $L_2/L_1$. 
For Schwartz-Bruhat functions $\phi$ and $\phi^*$ on $V_\Q$ satisfying 
\[
\phi(\rho(\gamma)x)=\phi(x), \quad 
\phi^*(\rho^*(\gamma)x)=\phi^*(x) \quad 
(\gamma \in \Gamma_H,\ x \in V_\Q), 
\]
we can define the zeta functions
$ \zeta_{\varepsilon}(\phi;w,s)$ and $\zeta_{\eta}^{*}(\phi^*;w,s)$ associated with the prehomogeneous vector spaces $(H\times GL_{1}(\C), \rho, V)$ and 
$(H\times GL_{1}(\C), \rho^{*}, V)$ by
\begin{align}
 \label{form:DefOfZeta}
 \zeta_{\pm}(\phi;w,s)&=
 \sum_{x\in \Gamma_H\backslash V_{\pm}\cap
 V_{\Q}} \phi(x) |x_{m+1}|^{-w}\cdot |Q(x)|^{-s} \\
 \label{form:DefOfZetaStar} 
 \zeta_{\pm}^{*}(\phi^*;w,s)&=
 \sum_{y\in \Gamma_H\backslash V_{\pm}^{*}\cap V_{\Q}}
 \phi^*(y) |y_{0}|^{-w}\cdot |Q^{*}(y)|^{-s}, 
\end{align}
where $w, s$ are complex variables.
By \cite[Theorem~1]{SatoConv} (or \cite[Theorem~1.2]{Saito}), $ \zeta_{\varepsilon}(\phi;w,s)$
and $\zeta_{\eta}^{*}(\phi^*;w,s)$ converge absolutely for 
$\Re(s)>1$, $\Re(w)>m$. 
For $n\in \Q$, we put 
\[
 Z(\phi;n,w) = \sum_{\nidan{x \in \rho(\Gamma_H)\backslash V_\Q}{x_{m+1}\ne0,\ Q(x)=n}} \frac{\phi(x)}{\abs{x_{m+1}}^w}, \quad 
 Z^*(\phi^*;n,w) = \sum_{\nidan{y \in \rho^*(\Gamma_H)\backslash V_\Q}{y_{0}\ne0,\ Q^*(y)=n}} \frac{\phi^*(y)}{\abs{y_{0}}^w}. 
\]
Then we have 
\begin{equation}
\zeta_\pm(\phi;w,s) 
 = \sum_{n \in \Q^\times_{>0}} \frac{Z(\phi;\pm n,w)}{n^s}, \quad
\zeta^*_\pm(\phi^*;w,s) 
 = \sum_{n \in \Q^\times_{>0}} \frac{Z^*(\phi^*;\pm n,w)}{n^s}
\label{eqn:DS hyoji}
\end{equation}

We define the Fourier transform $\widehat{\phi}$ of a Schwartz-Bruhat function $\phi$ by
\[
\widehat{\phi}(y) = \frac{1}{[V_\Z:L]} \sum_{x \in V_\Q/L} \phi(x) e^{-2\pi i \langle x,y\rangle}, 
\]
where $L$ is a sufficiently small lattice so that $L \subset V_\Z$ and the value $\phi(x) e^{-2\pi i \langle x,y\rangle}$ depends only on the residue class $x+L$. 
Though $L$ is not unique, the value $\widehat{\phi}(y)$ does not depend on the choice of $L$.  
Then, by the general theory of prehomogeneous vector spaces (\cite{SatoFE}, \cite[\S 4]{SatoZetaDist}), there exists a functional equation that relates $\zeta_\pm(\phi;w,s)$ to $\zeta^*_\pm(\widehat{\phi};w,\frac m2+1-w-s)$.

\begin{theorem}
\label{thm:FE for general phi}
$(1)$ For a $\rho(\Gamma_H)$-invariant Schwartz-Bruhat function  $\phi$, the zeta functions $\zeta_\pm\left(\phi;w,s\right)$ and  $\zeta^*_\pm\left(\widehat{\phi};w,s\right)$ have analytic continuations to meromorphic functions of $(s,w)$ 
in $\C^2$. 
For a fixed $w$ with $\Re(w)>m$, the zeta functions multiplied by $(s-1)(s+w-\frac m2-1)$ are entire functions of $s$ of finite order in any vertical strip.  

$(2)$ The residues at $s=1$ and $s=\frac m2+1-w$ are given by
\begin{gather*}
\mathop{\mathrm{Res}}_{s=1} \zeta_+(\phi;w,s) 
             = \mathop{\mathrm{Res}}_{s=1} \zeta_-(\phi;w,s) 
             = \eta(\phi;w), \\
\mathop{\mathrm{Res}}_{s=m/2+1-w} \zeta_\epsilon(\phi;w,s) 
             = 2\abs{D}^{-1/2}(2\pi)^{m/2 -w} \Gamma\left(w - \frac m2\right)Z^*(\widehat{\phi};0,w) \times 
 \begin{cases}
 \cos\frac\pi2(w-p) & (\epsilon=+), \\
 \cos\frac\pi2(w-q) & (\epsilon=-), \\
\end{cases} \\
\mathop{\mathrm{Res}}_{s=1} \zeta^*_+(\widehat{\phi};w,s) 
             = \mathop{\mathrm{Res}}_{s=1} \zeta^*_-(\widehat{\phi};w,s) 
             = \eta^*(\phi;w), \\
\mathop{\mathrm{Res}}_{s=m/2+1-w} \zeta^*_\epsilon(\widehat{\phi};w,s) 
             = 2\abs{D}^{-1/2}(2\pi)^{m/2 -w} \Gamma\left(w - \frac m2\right)Z(\phi;0,w) \times 
 \begin{cases}
 \cos\frac\pi2(w-p) & (\epsilon=+), \\
 \cos\frac\pi2(w-q) & (\epsilon=-),
\end{cases} 
\end{gather*}
where we put 
\begin{eqnarray*}
\eta(\phi;w) &=& \frac12  \sum_{x_{m+1}\in\Q^\times} \frac{\calA\phi(x_{m+1})}{|x_{m+1}|^{w-m+1}},  \quad
\calA\phi(x_{m+1}) 
 = \frac{1}{M^{m+1}}\sum_{(x_0,x')\in\Q^{m+1}/M\Z^{m+1}} 
       \phi(x_0,x',x_{m+1}), \\
\eta^*(\phi;w)  &=& \frac12 \sum_{y_0\in\Q^\times} \frac{\calA^*\phi(y_0)}{|y_{0}|^{w-m+1}}, 
\quad  \calA^*\phi(y_0) 
 = \frac{1}{M}\sum_{x_0\in\Q/M\Z} 
       \phi(x_0,0,0) \mathbf{e}\left[-x_0y_0\right].
\end{eqnarray*}
Here $M$ is a sufficiently large positive integer, and the coefficients $\calA\phi(x_{m+1})$ and  $\calA^*\phi(y_0)$ are independent of the choice of $M$.

$(3)$ The zeta functions $\zeta_\pm\left(\phi;w,s\right)$ and  $\zeta^*_\pm\left(\widehat{\phi};w,s\right)$ satisfy the following functional equation:
\begin{eqnarray*}
\begin{pmatrix}
\zeta^*_+\left(\widehat{\phi};w,\frac m2+1-w-s\right) \\
\zeta^*_-\left(\widehat{\phi};w,\frac m2+1-w-s\right) 
\end{pmatrix}
 &=& 2 \abs{D}^{-1/2} (2\pi)^{m/2-w-2s}\Gamma(s)\Gamma\left(w+s-\frac m2\right) \\
 & &  \times
\begin{pmatrix}
   \cos\left(\frac{\pi(w+2s-p)}2\right) &
   \cos\left(\frac{\pi(w-p)}2\right) \\
   \cos\left(\frac{\pi(w-q)}2\right) & 
   \cos\left(\frac{\pi(w+2s-q)}2\right)
\end{pmatrix}
\begin{pmatrix}
\zeta_+\left(\phi;w,s\right) \\
\zeta_-\left(\phi;w,s\right) 
\end{pmatrix}.
\end{eqnarray*}
\end{theorem}

We note that the functional equation above is rewritten as
\begin{eqnarray}
\lefteqn{ e^{-\pi i (p-q)/4} \abs{D}^{-1/2} (2\pi)^{-s}\Gamma(s)
\gamma(s)
\begin{pmatrix}
\zeta_+\left(\phi;w,s\right) \\
\zeta_-\left(\phi;w,s\right) 
\end{pmatrix}
} \nonumber \\
 & & \quad = (2\pi)^{-(m/2+1-w-s)}\Gamma\left(\frac m2+1-w-s\right) 
       \nonumber \\
 & & \quad \quad {} \times \Sigma(p-q)\gamma\left(\frac m2+1-w-s\right)
\begin{pmatrix}
\zeta^*_+\left(\widehat{\phi};w,\frac m2+1-w-s\right) \\
\zeta^*_-\left(\widehat{\phi};w,\frac m2+1-w-s\right) 
\end{pmatrix}. 
\label{eqn:normalized FE}
\end{eqnarray}

The theorem above was proved in \cite[Theorem 4.1, Corollary 4.2, Proof of Theorem 4.5]{Ueno} for special cases of $\phi$ ($\phi_0$ and $\phi_\psi$ with non-principal $\psi$ defined in \S \ref{subsection:5.3} below). 
The proof for general Schwartz-Bruhat functions $\phi$ is almost the same as that in \cite{Ueno}, and is omitted.   

\begin{remark}
Since we leave the details of the proof of Theorem \ref{thm:FE for general phi} to \cite{Ueno},  we give here a list of corrections to \cite{Ueno}:
\begin{description}
\item[{\rm p.3, line 5:}]
``$D=\det A$'' should be ``$D=-\det(2A)$''. 
Note that, in the present paper,  $D$ is defined to be $\det(2A)$, since some of identities become simpler. 
\item[{\rm p.3, Theorem; p.16, Lemma 3.4; p.17 Remark 2 (2); p.29, lines 8 and 13; p.31, line 8:}] If $m$ is even, then $(-1)^{m/2}\det(2A)\equiv 0, 1\pmod 4$, and hence the case $|D|\equiv2\mod4$ cannot occur and should be deleted. 
\item[{\rm p.3, line 13:}] 
``$G_{c(\epsilon,k)+1}(2|D|,\mathrm{id}_{2|D|})$'' should be ``$G_{c(\epsilon,k)+1}\left(2|D|,\left(\frac{2|D|}{*}\right)\right)$''. 
\item[{\rm p.13, line 3; p.26, line 5 from bottom; p.27, line 4:}]
``$q-p$'' should be ``$p-q$''. 
\item[{\rm p.13, line 3 from bottom:}]
``Proposition 3.2'' should be ``Lemma 3.2''. 
\item[{\rm p.18, lines 15 and 17; p.20, line 6:}]
``$G_\R$'' should be ``$G^+$''. 
\item[{\rm p.21, line 10:}]
``$y_0\Z^m$'' should be ``$2y_0A\Z^m$''. 
\item[{\rm p.22, line 6 from bottom:}]
``$da\,dt$'' should be ``$dt\,da$''. 
\item[{\rm p.22, line 5 from bottom:}]
``$\displaystyle{\sum_{y \in L_1^*} |y_0|^{-w}}$'' should be ``$Z^*(0,w)$''. 
\item[{\rm p.22, line 4 from bottom:}]
``$\displaystyle{\sum_{x \in L_1} |x_{m+1}|^{-w}}$'' should be ``$Z(0,w)$''. 
\item[{\rm p.24, line 5:}]
``$a^{m-w-2}$'' should be ``$a^{m-w-1}$''. 
\item[{\rm p.32. lines 4--2 from bottom:}]
``{\it or $\mathfrak{G}_{c(\epsilon,k)+1}\ \cdots {}\bmod 4$ respectively}'' should be deleted.  
\end{description} 
\end{remark}

\begin{corollary}
\label{cor:5.5}
$(1)$
For any non-negative integer $k$, we have
\begin{eqnarray*}
\zeta_+\left(\phi;w,-k\right) +  (-1)^k\zeta_-\left(\phi;w,-k\right) 
 &=& \begin{cases}
   -Z(\phi;0,w) & (k=0), \\
     0 & (k>0),
    \end{cases} \\
\zeta^*_+\left(\widehat{\phi};w,-k\right) +  (-1)^k\zeta^*_-\left(\widehat{\phi};w,-k\right) 
 &=& \begin{cases}
   -|D|^{-1}Z^*(\widehat{\phi};0,w) & (k=0), \\
     0 & (k>0).
    \end{cases} 
\end{eqnarray*}

$(2)$
Put $\lambda=\frac12(w-\frac m2 +1)$. 
If $\lambda=1$, namely, $w=\frac m2+1$, then we have 
\[
\mathop{\mathrm{Res}}_{s=0}\zeta_+(\phi;w,s) 
      + \mathop{\mathrm{Res}}_{s=0}\zeta_-(\phi;w,s) 
       = \mathop{\mathrm{Res}}_{s=0}\zeta^*_+(\widehat{\phi};w,s) 
         + \mathop{\mathrm{Res}}_{s=0}\zeta^*_-(\widehat{\phi};w,s) 
 = 0.  
\]
Further, if $p-q\equiv 0 \pmod 4$ and $\lambda=1$, then 
\[
\mathop{\mathrm{Res}}_{s=0}\zeta_+(\phi;w,s) 
      = \mathop{\mathrm{Res}}_{s=0}\zeta_-(\phi;w,s) 
       = \mathop{\mathrm{Res}}_{s=0}\zeta^*_+(\widehat{\phi};w,s) 
         = \mathop{\mathrm{Res}}_{s=0}\zeta^*_-(\widehat{\phi};w,s) 
 = 0. 
\]
\end{corollary}

The corollary follows easily from the functional equation and the residue formulas at $s=\frac m2+1-w$ in Theorem \ref{thm:FE for general phi}.

\subsection{Proof of Theorem \ref{thm:Maass form for Ueno zeta}}
\label{subsection:5.3}

First we show that the zeta functions $\zeta_\pm(w,s)$ and $\zeta_\pm^*(w,s)$ defined in \eqref{eqn:Def of Ueno zeta} and their twists by Dirichlet characters are special cases of the zeta functions $ \zeta_{\pm}(\phi;w,s)$ and $\zeta_{\pm}^{*}(\widehat{\phi};w,s)$.

Denote by $\phi_0$ the characteristic function of $V_\Z$. 
Then,  we have  $\widehat{\phi_0}=\phi_0$ and 
\begin{gather}
Z(\phi_0;n,w) 
 = \begin{cases}
\displaystyle \sum_{l=1}^\infty \frac{r(l,n)}{l^w} =: Z(n,w) & (n \in \Z), \\[2ex] 
0 & (n \not\in \Z),
\end{cases}
\label{eqn:z-phi0}
\\
Z^*(\widehat{\phi_0};\frac{n}N,w) 
 = \begin{cases}
\displaystyle \sum_{l=1}^\infty \frac{r^*(l,n)}{l^w} =: Z^*(n,w) & (n \in \Z), \\ 
0 & (n \not\in \Z).
\end{cases}
\label{eqn:z-star-phi0}
\end{gather}
Hence we obtain
\[
\zeta_\pm(w,s) = \zeta_\pm(\phi_0;w,s), \quad 
\zeta^*_\pm(w,s) = \zeta^*_\pm(\widehat{\phi_0};w,s).
\]
For a Dirichlet character $\psi$ modulo an odd prime number $r$ $((r,N)=1)$, we consider the function 
\[
\phi_\psi(x) = \begin{cases}
                   \tau_\psi(Q(x)) & (x \in V_\Z), \\
                   0  & (x \not\in V_\Z)
                   \end{cases}
\]
on $V_\Q$, where $\tau_\psi$ is the Gauss sum defined by \eqref{eqn:def of Gauss sum}. 
Then we have
\begin{equation}
Z(\phi_\psi;n,w) 
 = \begin{cases}
\tau_\psi(n) Z(n,w) & (n \in \Z), \\
0 & (n \not\in \Z),
\end{cases} 
\label{eqn:Z-star-phi}
\end{equation}
and 
\[
 \zeta_{\pm}(\phi_\psi;w,s) 
 =  \sum_{n=1}^{\infty}\tau_{\psi}(\pm n) Z(\pm n, w)n^{-s}. 
\]
This is the twist of $\zeta_\pm(w,s)$ by $\psi$ as introduced in \S \ref{subsect:4.1}. 
The twist of  $\zeta^*_\pm(w,s)$ by $\psi$ is obtained from $\zeta^*_\pm(\widehat{\phi_\psi};w,s)$. 
The explicit form of the Fourier transform $\widehat{\phi_\psi}$ is given by the following lemma, which is easily derived from Stark \cite[Lemmas 5 and 6]{Stark}. 

\begin{lemma}
\label{lem:CalOfFTofGaussSum}
If $y\not\in r^{-1}V_{\Z}$, then
$\widehat{\phi_{\psi}}(y)=0$. If $y\in r^{-1}V_{\Z}$, then
\begin{gather*}
 \widehat{\phi_{\psi}}(y) =  r^{-m/2-1} K_{r,\psi} \tau_{\psi^*} \left(NQ^*(ry)\right), \\
K_{r,\psi} 
 = C_{p-q,r} \psi^*(-N) \chi_K(r), 
\end{gather*}
where $\chi_K$ is the same as in Theorem \ref{thm:Maass form for Ueno zeta},  and $\psi^*$ and $C_{p-q,r}$ are given by \eqref{eqn:psi-star} and \eqref{eqn:def of clr}, respectively,  for $\ell=p-q$. 
\end{lemma}
It follows from Lemma \ref{lem:CalOfFTofGaussSum} that 
\begin{equation}
Z^*(\widehat{\phi_\psi};\frac{n}{Nr^2},w) 
 = \begin{cases}
r^{w-m/2-1} K_{r,\psi} \tau_{\psi^*}(n) Z^*(n,w) & (n \in \Z), \\ 
0 & (n \not\in \Z)
\end{cases}
\label{eqn:Z-star-psi}
\end{equation}
and
\begin{equation}
(Nr^2)^{-s}\zeta^*_\pm(\widehat{\phi_\psi};w,s) 
 = r^{w-m/2-1} K_{r,\psi} \sum_{n=1}^\infty 
\tau_{\psi^*} \left(\pm n\right) Z^*(\pm n,w) n^{-s}.
\label{form:CalOnRazarTypeZeta}
\end{equation}
Thus we obtain the following lemma. 

\begin{lemma}
Let $\phi_0(x)$ be the characteristic function of $V_\Z$. 
For a Dirichlet character $\psi$ module an odd prime number $r$ not dividing $N$, 
we put $\phi_\psi(x)=\tau_\psi(Q(x))\phi_0(x)$.  
Then we have 
\begin{eqnarray*}
\xi_\pm(w,s) &=& e^{\pi i(p-q)/4}|D|^{-1/2}\zeta_\pm(\phi_0;w,s), \\
\xi^*_\pm(w,s) &=& N^{-s} \zeta^*_\pm(\widehat{\phi_0};w,s), \\
\xi_\pm(\psi;w,s) &=& e^{\pi i(p-q)/4}|D|^{-1/2}\zeta_\pm(\phi_\psi;w,s), \\
\xi^*_\pm(\psi^*;w,s) &=& \frac{(Nr^2)^{-s} r^{-w+m/2+1}}{K_{r,\psi}}
                    \zeta^*_\pm(\widehat{\phi_\psi};w,s).
\end{eqnarray*}
\end{lemma}

We need the explicit formulas for the Dirichlet series $\eta(\phi;w)$ and  $\eta^*(\phi;w)$ appearing in the residues of $\zeta_\pm(\phi;w,s)$ and $\zeta^*_\pm(\widehat\phi;w,s)$ at $s=1$ for $\phi=\phi_0$ and $\phi_\psi$. 

\begin{lemma}
\label{lemma:exlicit eta}
We have 
\begin{eqnarray*}
\eta(\phi_0;w) &=& \eta^*(\phi_0;w) = \zeta(w-m+1), \\
\eta(\phi_\psi;w) &=& r^{m/2-w-1}\varepsilon_r^m\left(\frac{2^mD}r\right) 
 \tau_{\psi^*}(0)\zeta(w-m+1), \\
\eta^*(\phi_\psi;w) &=& \tau_{\psi}(0)\zeta(w-m+1).
\end{eqnarray*}遯ｶ�｢
\end{lemma}

\begin{proof}
We give the proof only for the expression of $\eta(\phi_\psi;w)$, since the proof of the other expressions are rather easy. 
Let $\psi$ be a Dirichlet character modulo an odd prime number $r$ not dividing $N$. 
By the definition of  $\phi_\psi$, it is obvious that $\calA\phi_\psi(x_{m+1}) = 0$ for $x_{m+1}\not\in \Z$. If $x_{m+1} \in \Z$, then by \cite[Lemma 5]{Stark}, we have 
\begin{eqnarray*}
\calA\phi_\psi(x_{m+1}) 
 &=& \frac{1}{r^{m+1}} \sum_{(x_0,x') \in \Z^{m+1}/r\Z^{m+1}} \tau_\psi(x_0x_{m+1}+{}^tx'Ax') \\ 
 &=& \frac{1}{r^{m+1}} \sum_{(k,r)=1} \psi(k)
 \sum_{x' \in \Z^{m}/r\Z^{m}} \mathbf{e}\left[\frac{k}{r}\cdot{}^tx'Ax'\right]  
 \sum_{x_0 \in \Z/r\Z} \mathbf{e}\left[\frac{kx_0x_{m+1}}r\right] \\ 
 &=& \frac{1}{r^{m}} \sum_{(k,r)=1} \psi(k)
 \sum_{x' \in \Z^{m}/r\Z^{m}} \mathbf{e}\left[\frac{k}{r}\cdot{}^tx'Ax'\right]  
 \times 
 \begin{cases}
 1 & (x_{m+1} \equiv 0 \pmod r) \\ 
 0 & (x_{m+1} \not\equiv 0 \pmod r)
\end{cases} \\
 &=& r^{-m/2} \varepsilon_r^m \left(\frac{2^m D}r\right) \cdot \tau_{\psi^*}(0) \times 
 \begin{cases}
 1 & (x_{m+1} \equiv 0 \pmod r), \\ 
 0 & (x_{m+1} \not\equiv 0 \pmod r).
\end{cases} 
\end{eqnarray*}
The expression of $\eta(\phi_\psi;w)$ follows immediately from this identity. 
\end{proof}

Now we are in a position to complete the proof of Theorem \ref{thm:Maass form for Ueno zeta}. 
We denote by $\alpha(\pm n)$ and $\beta(\pm n)$ $(n \in \Z_{>0})$ the coefficients of $n^{-s}$ in $\xi_\pm(w,s)$ and  $\xi^*_\pm(w,s)$, respectively. 
Namely
\[
\alpha(n) = e^{\pi i (p-q)/4}|D|^{-1/2} Z(n,w), \quad 
\beta(n) = Z^*(n,w)  
\quad (n \in \Z,\ n\ne0), 
\] 
Then, by Corollary \ref{cor:5.5} (1) for $k=0$, \eqref{eqn:z-phi0} and \eqref{eqn:z-star-phi0}, we see that 
 $\alpha(0)$ and $\beta(0)$ defined by \eqref{form:DefOfA0AndAInfty} and 
 \eqref{form:DefOfB0AndBInfty} are given by 
\[
\alpha(0) = e^{\pi i (p-q)/4}|D|^{-1/2} Z(0,w), \quad 
\beta(0) = |D|^{-1} Z^*(0,w).
\]
From the residue formulas in Theorem \ref{thm:FE for general phi} (2) 
and Lemma \ref{lemma:exlicit eta}, we get the following explicit expressions for 
 $\alpha(\infty)$ and $\beta(\infty)$ defined by \eqref{form:DefOfA0AndAInfty2} and 
 \eqref{form:DefOfB0AndBInfty2}:
\[
\alpha(\infty) =  \zeta(w-m+1), \quad 
\beta(\infty) =  e^{-\pi i(p-q)/4} |D|^{-1/2} \zeta(w-m+1). 
\]
The conditions [A1] and $\mathrm{[A1]}_{r,\psi}$ follow immediately from Theorem \ref{thm:FE for general phi} (1),  
the conditions [A2] and $\mathrm{[A2]}_{r,\psi}$ from Theorem \ref{thm:FE for general phi} (2), 
the functional equations in [A3] and $\mathrm{[A3]}_{r,\psi}$ from the functional equation \eqref{eqn:normalized FE},  
the conditions [A4] and $\mathrm{[A4]}_{r,\psi}$ from Corollary \ref{cor:5.5}, 
and the conditions $\mathrm{[A5]}_{r,\psi}$ from Theorem \ref{thm:FE for general phi} (2) and Corollary \ref{cor:5.5}.

\end{document}